\documentclass{amsart} 
\usepackage{amssymb}
\usepackage{amsmath}
\usepackage{amsfonts}

\begin{document}
\newtheorem{theo}{Theorem}[section]
\newtheorem{prop}[theo]{Proposition}
\newtheorem{lemma}[theo]{Lemma}
\newtheorem{coro}[theo]{Corollary}
\theoremstyle{definition}
\newtheorem{exam}[theo]{Example}
\newtheorem{defi}[theo]{Definition}
\newtheorem{rem}[theo]{Remark}


\newcommand{\Bb}{{\bf B}}
\newcommand{\Cb}{{\bf C}}
\newcommand{\Nb}{{\bf N}}
\newcommand{\Qb}{{\bf Q}}
\newcommand{\Rb}{{\bf R}}
\newcommand{\Zb}{{\bf Z}}
\newcommand{\Ac}{{\mathcal A}}
\newcommand{\Bc}{{\mathcal B}}
\newcommand{\Cc}{{\mathcal C}}
\newcommand{\Dc}{{\mathcal D}}
\newcommand{\Ec}{{\mathcal E}}
\newcommand{\Fc}{{\mathcal F}}
\newcommand{\Ic}{{\mathcal I}}
\newcommand{\Jc}{{\mathcal J}}
\newcommand{\Kc}{{\mathcal K}}
\newcommand{\Lc}{{\mathcal L}}
\newcommand{\Mx}{{\mathcal M}}
\newcommand{\Nc}{{\mathcal N}}
\newcommand{\Oc}{{\mathcal O}}
\newcommand{\Pc}{{\mathcal P}}
\newcommand{\Qc}{{\mathcal Q}}
\newcommand{\Rc}{{\mathcal R}}
\newcommand{\Sc}{{\mathcal S}}
\newcommand{\Tb}{{\bf T}}
\newcommand{\Tc}{{\mathcal T}}
\newcommand{\TC}{{\mathcal TC}}
\newcommand{\Ub}{{\bf U}}
\newcommand{\Uc}{{\mathcal U}}
\newcommand{\Vb}{{\bf V}}
\newcommand{\Vc}{{\mathcal V}}
\newcommand{\Wb}{{\bf W}}
\newcommand{\Wc}{{\mathcal W}}
\newcommand{\btu}{\bigtriangleup}

\author{Nik Weaver}

\title [Quantum relations]
       {Quantum relations}

\address {Department of Mathematics\\
          Washington University in Saint Louis\\
          Saint Louis, MO 63130}

\email {nweaver@math.wustl.edu}


\date{May 3, 2010}

\begin{abstract}
We define a ``quantum relation'' on a von Neumann algebra $\Mx
\subseteq \Bc(H)$ to be a weak* closed operator bimodule over its
commutant $\Mx'$.  Although this definition is framed in terms of a
particular representation of $\Mx$, it is effectively representation
independent. Quantum relations on $l^\infty(X)$ exactly correspond to
subsets of $X^2$, i.e., relations on $X$. There is also a good definition
of a ``measurable relation'' on a measure space, to which quantum
relations partially reduce in the general abelian case.

By analogy with the classical setting, we can identify structures such as
quantum equivalence relations, quantum partial orders, and quantum graphs,
and we can generalize Arveson's fundamental work on weak* closed
operator algebras containing a masa to these cases. We are also able to
intrinsically characterize the quantum relations on $\Mx$ in terms of
families of projections in $\Mx {\overline{\otimes}} \Bc(l^2)$.
\end{abstract}

\maketitle


This paper arose out of a joint project between Greg Kuperberg and the
author \cite{KW}. That project involved a new definition of quantum
metrics that is suited to the von Neumann algebra setting and is especially
motivated by the metric aspect of quantum error correction. In the
course of that investigation, weak* closed operator bimodules over the
commutant of a von Neumann algebra emerged as centrally important, and
it became apparent that they were playing the role of a quantum version
of relations on a set.

There is an obvious von Neumann algebra version of the notion of
a relation on a set $X$, i.e., a subset of $X\times X$. Passing from $X$
to a von Neumann algebra $\Mx$, the standard translation dictionary tells
us to replace $X \times X$ with the von Neumann algebra tensor product
$\Mx {\overline{\otimes}} \Mx$. A subset of $X\times X$ would then
presumably correspond to a projection in $\Mx {\overline{\otimes}} \Mx$.
Thus, it would be natural to take a ``quantum
relation'' on $\Mx$ to be a projection in $\Mx {\overline{\otimes}} \Mx$.

But although this definition is simple and natural, it is not particularly
fruitful. If we try to identify conditions which could serve as quantum
analogs of, say, reflexivity or transitivity of a relation, this line of
thought becomes complicated and does not seem to lead anywhere interesting.
In light of the fundamental role played in classical mathematics by the
concept of a relation on a set, together with the fact that we do have
robust quantum analogs of large portions of classical mathematics
(\cite{Con}; see also \cite{W5}), this is rather disappointing.

A pessimistic conclusion which could be drawn is that the classical concept of a
relation on a set is simply not ``rigid'' or ``algebraic'' enough to have
a good quantum analog. To the contrary, we claim that there is a very good
quantum analog, it is just not the obvious one.

We define a {\it quantum relation} on a von Neumann algebra
$\Mx \subseteq \Bc(H)$ to be a weak* closed operator bimodule over $\Mx'$
(Definition \ref{quantrel}). Although this definition is framed in terms
of a particular representation of $\Mx$, it is effectively representation
independent (Theorem \ref{repindep}). In the atomic abelian case, quantum
relations on $l^\infty(X)$ correspond to subsets of $X \times X$, i.e.,
classical relations on $X$ (Proposition \ref{atomiccase}). We are also
able to give a reasonable definition of a ``measurable relation''
(Definition \ref{measrel}) to which quantum relations partially
reduce in the general abelian case (Theorem \ref{abelianrel} and
Corollary \ref{qreflex}).

Quantum analogs of such properties as reflexivity and transitivity are
easily identified using the algebraic structure available in $\Bc(H)$, and
this allows us to define such things as quantum partial orders and quantum
graphs (Definition \ref{quantthings}). These turn out to be well-known
structures familiar from the standard operator algebra toolkit; for
instance, a quantum preorder on $\Mx$ is just a weak* closed operator
algebra containing $\Mx'$.

This new point of view pays dividends. For instance, we can elegantly
generalize Arveson's fundamental results on weak* closed operator
algebras containing a masa \cite{Arv}. Some of this work roughly duplicates
results of Erdos \cite{Erd}, providing a new context for that material. The
advantages of our point of view are exhibited in an attractive
characterization of reflexive operator space and operator system bimodules
over a maximal abelian von Neumann algebra (Theorem \ref{char}).

While much of the content of Sections \ref{abcs} and \ref{apps} could
actually be derived from Arveson's work, in some ways the new point of
view provides a broader perspective that could be useful. For example,
we can define a pullback of quantum relations (Proposition \ref{qpb});
this is not possible in the setting of weak* closed operator algebras
because pullbacks are not compatible with products in general. (Thus,
the pullback of a quantum preorder need not be a quantum preorder.)

Our most substantial result is an intrinsic characterization of
quantum relations. Given a von Neumann algebra
$\Mx$, let $\Pc$ be the set of projections in $\Mx {\overline{\otimes}}
\Bc(l^2)$, equipped with the restriction of the weak operator topology.
We define an {\it intrinsic quantum relation} on $\Mx$ to be an
open subset $\Rc \subset \Pc \times \Pc$ satisfying
\begin{quote}
\noindent (i) $(0,0)\not\in \Rc$

\noindent (ii) $(\bigvee P_\lambda, \bigvee Q_\kappa) \in \Rc$
$\Leftrightarrow$ some $(P_\lambda, Q_\kappa) \in \Rc$

\noindent (iii) $(P,[BQ]) \in \Rc$ $\Leftrightarrow$ $([B^*P],Q) \in \Rc$
\end{quote}
for all projections $P,Q,P_\lambda,Q_\kappa \in \Pc$ and all
$B \in I \otimes \Bc(l^2)$. Here square brackets denote range projection.
We prove that intrinsic quantum relations on $\Mx$ naturally correspond to
quantum relations on $\Mx$ (Theorem \ref{abstractchar}).

To illustrate the tractability of quantum relations, we introduce a notion
of translation invariance for quantum relations on quantum torus von Neumann
algebras (Definition \ref{titori}) and characterize the structure of quantum
relations on quantum tori with this property (Theorem \ref{transinvqr} and
Corollary \ref{transinvcor}).

We work with complex scalars throughout.

\tableofcontents

\section{Measurable relations}\label{mrs}

A relation on a set $X$ is a subset of $X^2$. Before discussing
a noncommutative version of this notion, we first consider
the measurable setting since the idea of a ``measurable relation'' is
already interesting. Some readers may prefer to skip to Section 2 and
refer back to this section as needed.

\subsection{Basic definitions}
The obvious definition of a measurable relation on a measure space $(X,\mu)$
would be a measurable subset of $X^2$ up to modification on a null set. But
if $(X,\mu)$ is nonatomic then under this definition the condition that a
relation be reflexive becomes vacuous, because the diagonal of $X\times X$
has measure zero. The notion of transitivity also becomes problematic.
Therefore we seek a better-behaved candidate to play the role
of a measurable relation. Our proposed definition is more
complicated than the one suggested above, but it has elegant
properties and generates a robust theory.

To avoid pathology we assume
that all measure spaces are {\it finitely decomposable}. This means
that the space $X$ can be partitioned into a (possibly uncountable)
family of finite measure subspaces $X_\lambda$ such that a set
$S \subseteq X$
is measurable if and only if its intersection with each $X_\lambda$
is measurable, in which case $\mu(S) = \sum \mu(S \cap X_\lambda)$
(\cite{W4}, Definition 6.1.1).
Finite decomposability is a generalization of $\sigma$-finiteness.
Counting measure on any set is also finitely decomposable.
The spaces $L^\infty(X,\mu)$ with $(X,\mu)$ finitely decomposable
are precisely the abelian von Neumann algebras.

If $(X,\mu)$ is finitely decomposable then the projections in the
abelian von Neumann algebra
$L^\infty(X,\mu)$ constitute a complete lattice. These projections
are precisely the characteristic functions $\chi_S$ of measurable subsets
$S\subseteq X$ up to null sets. Thus, working with projections in
$L^\infty(X,\mu)$ is a
convenient way to factor out equivalence modulo null sets.

Our approach will be to model a measurable relation by specifying
which pairs of projections belong to it. This differs from the classical
pointwise notion but the two can be made equivalent in the atomic case
if we adopt an appropriate axiom specifying how the pairs of projections
belonging to the relation must cohere. The suitable coherence axiom
(($*$) in Definition \ref{measrel} below) is motivated by the following
probably well-known result.

\begin{prop}\label{motive}
Let $(X,\mu)$ be a finitely decomposable measure space and let $\Fc$ be a
family of nonzero projections in $L^\infty(X,\mu)$ with the property that
for any family of nonzero projections $\{p_\lambda\}$ in $L^\infty(X,\mu)$
we have
$$\bigvee p_\lambda \in \Fc \quad \Leftrightarrow \quad
\hbox{some }p_\lambda \in \Fc.$$
Then there is a projection $p \in L^\infty(X,\mu)$ such that $q \in \Fc
\Leftrightarrow pq \neq 0$.
\end{prop}

\begin{proof}
Let $r$ be the join of all the projections that are not in $\Fc$.
Then the displayed condition implies that $r$ is not in $\Fc$.
Moreover, if $q$ is any projection less than $r$ then
$r \vee q = r \not\in \Fc$ implies $q \not\in \Fc$. So a projection
is not in $\Fc$ if and only if it lies below $r$. Letting $p = 1 - r$,
we then have $q \in \Fc \Leftrightarrow pq \neq 0$.
\end{proof}

Since the converse of Proposition \ref{motive} is trivial ---
for any projection $p$ the set of projections $q$ such that
$pq \neq 0$ does have the stated property --- this result gives us a
(somewhat roundabout) characterization of the measurable subsets of $X$ up to
null sets. Namely, having a measurable subset is the same as having a family
$\Fc$ of nonzero projections with the property that $\bigvee p_\lambda
\in \Fc \Leftrightarrow$ some $p_\lambda \in \Fc$. (Such a family is
just a proper complete filter of the lattice of projections in
$L^\infty(X,\mu)$.) This makes the following definition plausible.

\begin{defi}\label{measrel}
Let $(X,\mu)$ be a finitely decomposable measure space.
A {\it measurable relation} on $X$ is a family $\Rc$ of ordered pairs
of nonzero projections in $L^\infty(X,\mu)$ such that
$$\left(\bigvee p_\lambda,\bigvee q_\kappa\right) \in\Rc
\quad\Leftrightarrow\quad
\hbox{some }(p_\lambda, q_\kappa) \in \Rc\eqno{(*)}$$
for any pair of families of nonzero projections $\{p_\lambda\}$ and
$\{q_\kappa\}$. Equivalently, we can impose the two conditions
$$p' \leq p,\quad q' \leq q,\quad (p',q') \in \Rc\qquad
\Rightarrow\qquad (p,q) \in \Rc\eqno{(*')}$$
and
$$\left(\bigvee p_\lambda, \bigvee q_\kappa\right) \in \Rc
\quad\Rightarrow\quad\hbox{some } (p_\lambda,q_\kappa) \in \Rc.\eqno{(*'')}$$
(It is easy to check that ($*'$) is equivalent to the reverse
implication in ($*$), using the fact that $p' \leq p$
$\Leftrightarrow$ $p' \vee p = p$.)
\end{defi}

The generalization to a measurable relation on a pair of finitely
decomposable measure spaces $(X,\mu)$ and $(Y,\nu)$ would be a
family of ordered pairs of nonzero projections $p \in L^\infty(X,\mu)$
and $q \in L^\infty(Y,\nu)$ satisfying condition ($*$) (or,
alternatively, conditions ($*'$) and ($*''$)) in
Definition \ref{measrel}. We need not develop this more general
notion separately since measurable relations on $X$ and $Y$ can be identified
with measurable relations on the disjoint union $X \coprod Y$ that do not
contain the pairs $(\chi_X, \chi_X)$, $(\chi_Y, \chi_X)$, or
$(\chi_Y, \chi_Y)$ (i.e., that live in the $X\times Y$ corner of
$(X \coprod Y)^2$).

The intuition behind Definition \ref{measrel} is that a pair of projections
$(\chi_S, \chi_T)$ belongs to $\Rc$ if and only if some point in $S$
is related to some point in $T$. This pointwise condition is not
meaningful in the general measurable case, but we will now show that
in the atomic case
it effectively reduces measurable relations on $X$ to subsets of $X^2$.
Let $e_x = \chi_{\{x\}}$ be the characteristic function of
singleton $x$.

\begin{prop}\label{atommeas}
Let $\mu$ be counting measure on a set $X$.
If $R$ is a relation on $X$ then
\begin{eqnarray*}
\Rc_R &=&\{(\chi_S, \chi_T): (x,y) \in R\hbox{ for some }x \in S, y \in T\}\cr
&=& \{(\chi_S,\chi_T): (S \times T) \cap R \neq \emptyset\}
\end{eqnarray*}
is a measurable relation on $X$; conversely, if $\Rc$ is a
measurable relation on $X$ then
$$R_\Rc = \{(x,y) \in X^2: (e_x, e_y) \in \Rc\}$$
is a relation on $X$. The two constructions are inverse to each other.
\end{prop}

\begin{proof}
$\Rc_R$ is a measurable relation because $(\bigcup S_\lambda) \times
(\bigcup T_\kappa) = \bigcup_{\lambda,\kappa} (S_\lambda \times T_\kappa)$
intersects $R \subseteq X^2$ if and only if
$S_\lambda \times T_\kappa$ intersects $R$ for some $\lambda$, $\kappa$.
The fact that $R_\Rc$ is a relation is trivial. Now let $R$ be a relation,
let $\Rc = \Rc_R$, and let $\tilde{R} = R_\Rc$. Then
$$(x,y) \in R\quad\Leftrightarrow\quad (e_x, e_y) \in \Rc
\quad\Leftrightarrow\quad (x,y) \in \tilde{R}$$
for all $x,y \in X$,
so $R = \tilde{R}$. Finally, let $\Rc$ be a measurable relation,
let $R = R_\Rc$, and let $\tilde{\Rc} = \Rc_R$. Since $\Rc$ is a measurable
relation and $\chi_S = \bigvee_{x \in S} e_x$, $\chi_T =
\bigvee_{y \in T} e_y$ we have
\begin{eqnarray*}
(\chi_S,\chi_T) \in \Rc&\Leftrightarrow&
(e_x, e_y) \in \Rc\hbox{ for some }x \in S,y \in T\cr
&\Leftrightarrow& (x,y) \in R\hbox{ for some }x \in S,y \in T\cr
&\Leftrightarrow& (\chi_S,\chi_T) \in \tilde{\Rc}
\end{eqnarray*}
for all nonzero projections $\chi_S$ and $\chi_T$.
So $\Rc = \tilde{\Rc}$.
\end{proof}

We also want to mention the following alternative characterization
of measurable relations. This will be used in Section \ref{apps}
to establish a connection with operator space reflexivity.

\begin{prop}\label{erd}
Let $(X,\mu)$ be a finitely decomposable measure space.
If $\Rc$ is a measurable relation on $X$ then the map
$$\phi_\Rc: q \mapsto 1 - \bigvee \{p: (p,q) \not\in \Rc\},$$
from the set of projections in $L^\infty(X,\mu)$ to itself,
takes $0$ to $0$ and preserves arbitrary joins. If $\phi$ is
a map from the set of projections in $L^\infty(X,\mu)$ to itself
that takes $0$ to $0$ and preserves arbitrary joins then
$$\Rc_\phi = \{(p,q): p\phi(q) \neq 0\}$$
is a measurable relation on $X$. The two constructions are inverse
to each other.
\end{prop}

\begin{proof}
We start with a simple observation.
Let $\Rc$ be a measurable relation and let $p$ and $q$ be projections
in $L^\infty(X,\mu)$. It is immediate from the definition of measurable
relations that $(1 - \phi_\Rc(q),q) \not\in \Rc$, so if
$p\phi_\Rc(q) = 0$ then $p \leq 1 - \phi_\Rc(q)$ and hence
$(p,q) \not\in \Rc$. Conversely, if $p\phi_\Rc(q) \neq 0$ then
$(p,q) \in \Rc$ by the definition of $\phi_\Rc$.
So $(p,q) \in \Rc$ if and only if $p\phi_\Rc(q) \neq 0$.

It is clear that $\phi_\Rc(0) = 0$.
Let $\{q_\kappa\}$ be any family of projections in $L^\infty(X,\mu)$.
If $p\cdot \bigvee \phi_\Rc(q_\kappa) = 0$ then $p\phi_\Rc(q_\kappa) = 0$
for all $\kappa$, so that $(p,q_\kappa) \not\in \Rc$ for all $\kappa$,
which implies that $(p,\bigvee q_\kappa) \not\in \Rc$. Conversely, if
$p\cdot \bigvee \phi_\Rc(q_\kappa) \neq 0$ then $p\phi_\Rc(q_\kappa)
\neq 0$ for some $\kappa$, so that $(p,q_\kappa) \in \Rc$ for that
$\kappa$, which implies that $(p,\bigvee q_\kappa) \in \Rc$. So we
have shown that $(p,\bigvee q_\kappa) \in \Rc$ if and only if
$p\cdot \bigvee \phi_\Rc(q_\kappa) \neq 0$. By the preceding paragraph
we also have that $(p,\bigvee q_\kappa) \in \Rc$ if and only if
$p\phi_\Rc(\bigvee q_\kappa) \neq 0$, and this
implies that $\phi_\Rc(\bigvee q_\kappa) = \bigvee \phi_\Rc(q_\kappa)$.
Thus $\phi_\Rc$ takes $0$ to $0$ and preserves arbitrary joins.

Now let $\phi$ be any map which takes $0$ to $0$ and preserves arbitrary
joins and let $\{p_\lambda\}$ and $\{q_\kappa\}$ be families of nonzero
projections. Then 
\begin{eqnarray*}
(\bigvee p_\lambda, \bigvee q_\kappa) \in \Rc_\phi
&\Leftrightarrow&
\left(\bigvee p_\lambda\right) \phi\left(\bigvee q_\kappa\right) \neq 0\cr
&\Leftrightarrow&
\left(\bigvee p_\lambda\right) \left(\bigvee \phi(q_\kappa)\right) \neq 0\cr
&\Leftrightarrow&
p_\lambda\phi(q_\kappa) \neq 0\hbox{ for some }\lambda,\kappa\cr
&\Leftrightarrow&
(p_\lambda, q_\kappa) \in \Rc_\phi\hbox{ for some }\lambda,\kappa.
\end{eqnarray*}
Also, since $\phi(0) = 0$ it is clear that $(p,q) \not\in \Rc_\phi$ if
either $p$ or $q$ is $0$. So $\Rc_\phi$ is a measurable relation.

The fact that $\Rc = \Rc_{\phi_\Rc}$ follows immediately from the observation
made in the first paragraph of the proof. The identity $\phi =
\phi_{\Rc_\phi}$ follows from the fact that
$\bigvee\{p: p\phi(q) = 0\} = 1 - \phi(q)$.
\end{proof}

\subsection{Constructions with measurable relations}
The following are basic constructions with measurable relations.

\begin{prop}\label{measverify}
Let $(X,\mu)$ be a finitely decomposable measure space.

\noindent (a) The set of pairs of projections $p$ and $q$ in
$L^\infty(X,\mu)$ such that $pq \neq 0$ is a measurable relation on $X$.

\noindent (b) If $\Rc$ is a measurable relation on $X$ then so is
$\{(q,p): (p,q) \in \Rc\}$.

\noindent (c) If $\Rc$ and $\Rc'$ are measurable relations on $X$ then
a pair of nonzero projections $(p,r)$ satisfies
\begin{quote}
for every projection $q$, either $(p,q) \in \Rc$
or $(1-q,r) \in \Rc'$
\end{quote}
if and only if it satisfies
\begin{quote}
there exists a nonzero projection $q$ such that
$(p,q') \in \Rc$ and $(q',r) \in \Rc'$ for every nonzero $q' \leq q$
\end{quote}
and the set of all pairs satisfying these conditions constitutes a
measurable relation.

\noindent (d) Any union of measurable relations on $X$ is a measurable
relation on $X$.

\noindent (e) If $\Rc$ is a measurable relation on a finitely decomposable
measure space $(Y,\nu)$ and $\phi: L^\infty(X,\mu) \to L^\infty(Y,\nu)$ is
a unital weak* continuous $*$-homomorphism then
$$\phi^*(\Rc) = \{(p,q): (\phi(p), \phi(q)) \in \Rc\}$$
is a measurable relation on $X$.
\end{prop}

\begin{proof}
Parts (a), (b), (d), and (e) are easy. For part (c), fix measurable relations
$\Rc$ and $\Rc'$ and nonzero projections $p$ and $r$. Suppose first that
every projection $q$ satisfies either $(p,q) \in \Rc$ or
$(1-q,r) \in \Rc'$. Let $r'$ be the join of all the projections $r''$
such that $(p,r'') \not\in \Rc$ and let $p'$ be the join of all the
projections $p''$ such that $(p'', r) \not\in \Rc'$. If $p' \vee r' = 1$
then $q = r'$ falsifies our assumption on the pair $(p,r)$. Therefore
$q = 1 - p'\vee r'$ is a nonzero projection such that $(p,q') \in \Rc$
and $(q', r) \in \Rc'$ for every nonzero $q' \leq q$.

Next, suppose
there exists a projection $q$ satisfying $(p,q) \not\in \Rc$
and $(1-q,r) \not\in \Rc'$. Then every nonzero projection $q'$ either
satisfies $q'q \neq 0$, in which case $q'$ lies over a nonzero projection
$q'q$ with $(p,q'q) \not\in \Rc$, or else $q'(1-q) \neq 0$, in which case
$q'$ lies over a nonzero projection $q'(1-q)$ with $(q'(1-q),r) \not\in \Rc'$.
This completes the proof that the two conditions are equivalent.

For the second assertion in part (c), we verify conditions ($*'$) and ($*''$)
in Definition \ref{measrel}. The first condition is trivial. For the
second, suppose that for every $\lambda$ and $\kappa$ there
exists a projection $q_{\lambda,\kappa}$ such that
$(p_\lambda, q_{\lambda,\kappa}) \not\in \Rc$ and $(1-q_{\lambda,\kappa},
r_\kappa) \not\in \Rc'$. The projection
$$q = \bigwedge_\lambda \bigvee_\kappa q_{\lambda,\kappa}$$
then satisfies $(\bigvee p_\lambda,q) \not\in \Rc$ and
$(1 - q,\bigvee r_\kappa) \not\in \Rc'$ (because $1 - q =
\bigvee_\lambda \bigwedge_\kappa (1 - q_{\lambda,\kappa})$).
This shows that condition ($*''$) from Definition
\ref{measrel} holds. This completes the proof.
\end{proof}

This justifies the following definition.

\begin{defi}\label{meastypes}
Let $(X,\mu)$ be a finitely decomposable measure space.

\noindent (a) The {\it diagonal measurable relation $\Delta$} on $X$ is
defined by setting $(p,q) \in \Delta$ if $pq \neq 0$.

\noindent (b) The {\it transpose} of a measurable relation $\Rc$ is the
measurable relation $\Rc^T = \{(q,p): (p,q) \in \Rc\}$.

\noindent (c) The {\it product} of two measurable relations
$\Rc$ and $\Rc'$ is the measurable relation
$$\Rc\cdot\Rc' = \{(p,r):\hbox{ either condition in Proposition
\ref{measverify} (c) holds}\}.$$

\noindent (d) A measurable relation $\Rc$ on $X$ is
\begin{quote}
(i) {\it reflexive} if $\Delta \subseteq \Rc$

\noindent (ii) {\it symmetric} if $\Rc^T = \Rc$

\noindent (iii) {\it antisymmetric} if $\Rc \wedge \Rc^T \subseteq \Delta$

\noindent (iv) {\it transitive} if $\Rc^2 \subseteq \Rc$.
\end{quote}
\end{defi}

In (d) (iii), $\Rc \wedge \Rc^T$ denotes the greatest lower bound
of $\Rc$ and $\Rc^T$ under the partial ordering $\subseteq$. The set
of measurable relations is a complete lattice by Proposition
\ref{measverify} (d).

The diagonal relation intuitively consists of those pairs $(\chi_S,\chi_T)$
such that $S \times T$ intersects the diagonal in a set of positive measure.
Of course, this intuition is not accurate since the diagonal could
have measure zero. Notice that under our definition the diagonal relation
on any nonzero measure space $X$ is different from the zero relation, even
if the set-theoretic diagonal of $X^2$ has measure zero.

We note that pullbacks (Proposition \ref{measverify} (e)) are not compatible
with products; indeed, the atomic version of this statement (with
reversed arrows) already fails.
For example, let $X = \{x,y,z\}$ and $Y = \{x,y_1,y_2,z\}$ and define
$f: Y \to X$ by $f(x) = x$, $f(y_i) = y$ ($i = 1,2$), and $f(z) = z$.
Then the relations $R = \{(x,y_1)\}$ and $R' = \{(y_2,z)\}$ on $Y$ satisfy
$R\cdot R' = \emptyset$, so that $f_*(R\cdot R') = \emptyset$, but the
product $f_*(R)\cdot f_*(R')$ of their pushforwards is $\{(x,z)\}$.

We also note that the above definitions generalize the
corresponding classical constructions with relations on sets.

\begin{prop}\label{measgen}
Let $\mu$ be counting measure on a set $X$, let $R_1$, $R_2$, and $R_3$
be relations on $X$, and let $\Rc_i = \Rc_{R_i}$ ($i = 1,2,3$) be the
corresponding measurable relations as in Proposition \ref{atommeas}. Then

\noindent (a) $R_1 \subseteq R_2$ $\Leftrightarrow$ $\Rc_1 \subseteq \Rc_2$;

\noindent (b) $R_1$ is the diagonal relation $\Leftrightarrow$ $\Rc_1$ is the
diagonal measurable relation;

\noindent (c) $R_1$ is the transpose of $R_2$ $\Leftrightarrow$ $\Rc_1$ is the
transpose of $\Rc_2$; and

\noindent (d) $R_3$ is the product of $R_1$ and $R_2$ $\Leftrightarrow$
$\Rc_3$ is the product of $\Rc_1$ and $\Rc_2$.
\end{prop}

The proof of this proposition is straightforward.

Using Definition \ref{meastypes} we can define measurable versions
of equivalence relations, preorders, partial orders, and graphs:

\begin{defi}\label{measthings}
Let $(X,\mu)$ be a finitely decomposable measure space.

\noindent (a) A {\it measurable equivalence relation} on $X$ is
a reflexive, symmetric, transitive measurable relation on $X$.

\noindent (b) A {\it measurable preorder} on $X$ is a reflexive,
transitive measurable relation on $X$.

\noindent (c) A {\it measurable partial order} on $X$ is a
reflexive, antisymmetric, transitive measurable relation on $X$.

\noindent (d) A {\it measurable graph} on $X$ is a reflexive,
symmetric measurable relation on $X$.
\end{defi}

It follows from Proposition \ref{measgen} that the
preceding definitions generalize the corresponding classical
definitions.

Part (d) requires some explanation.
A (simple) graph is usually defined to be a vertex set $V$ together
with a family of 2-element subsets of $V$. But the same information
determines and is determined by a reflexive, symmetric relation
on $V$ and so a graph may equivalently be defined as a set
equipped with such a relation. This justifies our definition of a
measurable graph.

The expected definition of a measurable equivalence relation is
probably a von Neumann subalgebra of $L^\infty(X,\mu)$, or equivalently
(\cite{Tak}, Theorem II.4.24) a complete Boolean subalgebra of the Boolean
algebra of projections in $L^\infty(X,\mu)$. We will now show that this is
equivalent to our definition, and more generally, that measurable preorders
correspond to complete 0,1-sublattices of the lattice of projections
in $L^\infty(X,\mu)$.

Given a measurable relation $\Rc$ on $X$, say that $S \subseteq X$ is
a {\it lower set} if $(1 - \chi_S, \chi_S) \not\in \Rc$. In the atomic
case this means that all points below any point in $S$ belong to $S$.

\begin{theo}\label{versuslattice}
Let $(X,\mu)$ be a finitely decomposable measure space. If $\Rc$
is a measurable preorder on $X$ then
\begin{eqnarray*}
\Lc_\Rc &=& \{p \in L^\infty(X,\mu): p\hbox{ is a projection and }
(1-p,p) \not\in \Rc\}\cr
&=& \{\chi_S: S \subseteq X\hbox{ is a lower set}\}
\end{eqnarray*}
is a complete 0,1-sublattice of the lattice of projections
in $L^\infty(X,\mu)$. If $\Lc$ is a complete 0,1-sublattice of the
lattice of projections in $L^\infty(X,\mu)$ then
$$\Rc_\Lc = \{(p,q): pq' \neq 0\hbox{ for all }q' \in \Lc
\hbox{ with }q' \geq q\}$$
is a measurable preorder on $X$. The two constructions
are inverse to each other. This correspondence between measurable
preorders and complete 0,1-sublattices restricts to a correspondence
between measurable equivalence relations and complete Boolean
subalgebras.
\end{theo}

\begin{proof}
Let $\Rc$ be a measurable preorder on $X$. It is clear that $0$ and $1$
belong to $\Lc_\Rc$, and if $\Rc$ is symmetric then $\Lc_\Rc$ is
closed under complementation. Now let $\{p_\lambda\}$ be any family of
projections in $\Lc_\Rc$ and let $p = \bigvee p_\lambda$. Then
$(1-p_\lambda, p_\lambda) \not\in \Rc$ for every $\lambda$, hence
$(1-p, p_\lambda) \not\in \Rc$ for every $\lambda$ (since
$1 - p \leq 1 - p_\lambda$), hence $(1-p, p) \not\in \Rc$ (since
$p = \bigvee p_\lambda$). This shows that $p \in \Lc_\Rc$. Letting
$q = \bigwedge p_\lambda$, we also have that $(1-p_\lambda, p_\lambda)
\not\in \Rc$ for every $\lambda$ implies $(1 - p_\lambda,q) \not\in \Rc$
for every $\lambda$. Since $1 - q = \bigvee (1 - p_\lambda)$, it follows
that $(1-q, q) \not\in \Rc$, i.e., $q \in \Lc_\Rc$.
So $\Lc_\Rc$ is a complete 0,1-sublattice of the
lattice of projections in $L^\infty(X,\mu)$, and it is a Boolean algebra
if $\Rc$ is a measurable equivalence relation.

Next let $\Lc$ be any complete 0,1-sublattice of the lattice of projections.
We first check that $\Rc_\Lc$ satisfies the pair of conditions stated in
Definition \ref{measrel}. Condition ($*'$) is trivial. For
condition ($*''$), let $\{p_\lambda\}$ and $\{q_\kappa\}$
be any families of nonzero projections and suppose $(p_\lambda,
q_\kappa) \not\in \Rc_\Lc$ for all $\lambda$ and $\kappa$. There must
exist $q_{\lambda,\kappa} \in \Lc$ with
$q_\kappa \leq q_{\lambda,\kappa}$
and $p_\lambda q_{\lambda,\kappa} = 0$. Then
$$q' = \bigvee_\kappa\bigwedge_\lambda q_{\lambda,\kappa} \in \Lc$$
satisfies $\bigvee q_\kappa \leq q'$ and $(\bigvee p_\lambda)q'= 0$,
which shows that $(\bigvee p_\lambda, \bigvee q_\kappa)
\not\in \Rc_\Lc$. This verifies condition ($*''$),
so $\Rc_\Lc$ is a measurable relation.

Reflexivity of $\Rc_\Lc$ is easy, as is symmetry if $\Lc$ is Boolean.
To verify transitivity, let $p$ and $r$ be nonzero projections such
that for every projection $q$, either $(p,q) \in \Rc_\Lc$ or
$(1-q,r) \in \Rc_\Lc$; we must show that $(p,r) \in \Rc_\Lc$. Let $r'$
be a projection in $\Lc$ such that $r \leq r'$. Then $(1-r', r) \not\in
\Rc_\Lc$ directly from the definition of $\Rc_\Lc$ since $r'$ is a projection
in $\Lc$ that contains $r$ and is disjoint from $1-r'$. So the condition
we assumed on $p$ and $r$ yields $(p,r') \in \Rc_\Lc$. From this it follows
that $pr' \neq 0$ (since $r'$ is a projection in $\Lc$ with $r' \leq r'$),
and we conclude that $(p,r) \in \Rc_\Lc$, as desired. Thus, we have shown
that $\Rc_\Lc$ is a measurable preorder, and that it is a measurable
equivalence relation if $\Lc$ is Boolean.

Now let $\Rc$ be a measurable preorder, let $\Lc = \Lc_\Rc$, and let
$\tilde{\Rc} = \Rc_\Lc$. If
$(p,q) \in \Rc$ then for any $q' \in \Lc$ with $q \leq q'$
we must have $(p,q') \in \Rc$, which by the definition
of $\Lc$ implies $pq' \neq 0$ (since $(1-q', q') \not\in \Rc$).
Thus $(p,q) \in \Rc$ implies
$(p,q) \in \tilde{\Rc}$. For the reverse inclusion, suppose
$(p,q) \not\in \Rc$, let $p'$ be the join of all the projections
$r$ such that $(r,q) \not\in \Rc$, and let $q' = 1-p'$; we claim
that $q \leq q'$ and $q' \in \Lc$. Since $p \leq p'$ and hence
$pq' = 0$, this will verify that $(p,q) \not\in \tilde{\Rc}$.
First, if $(r,q) \not\in \Rc$ then $rq = 0$ by reflexivity
of $\Rc$, and this shows that $p'q = 0$, i.e., $q \leq q'$. To see that
$q' \in \Lc$ we must show that $(p',q') \not\in \Rc$. Let $r'$
be the join of all the projections $r$ such that $(p',r) \not\in \Rc$.
We must have $r' \leq q'$ by reflexivity. But if $r' \neq q'$ then
every nonzero projection $r \leq q' - r'$ satisfies $(p',r) \in \Rc$
(since $rr' = 0$) and $(r,q) \in \Rc$ (since $r\leq q' = 1-p'$); by
transitivity, this implies that $(p',q) \in \Rc$, a contradiction.
Therefore $r' = q'$ and hence $(p',q') \not\in \Rc$. This completes
the proof of the reverse inclusion.

Finally, let $\Lc$ be a complete 0,1-lattice of projections,
let $\Rc = \Rc_\Lc$, and let $\tilde{\Lc} = \Lc_\Rc$.
Then it is trivial that a projection $p$ belongs to $\Lc$
if and only if there exists $p' \in \Lc$ such that $p \leq p'$
and $(1-p)p' = 0$. Thus $p \in \tilde{\Lc}$ $\Leftrightarrow$
$(1-p,p) \not\in \Rc$ $\Leftrightarrow$ $p \in \Lc$. So
$\Lc = \tilde{\Lc}$.
\end{proof}

\subsection{Conversion to classical relations}\label{reduct}
In the general measurable setting the correspondence identified in
Proposition \ref{atommeas} remains partially valid: if $R$ is a measurable
subset of $X^2$ then
$$\Rc = \{(\chi_S,\chi_T): (S \times T) \cap R\hbox{ is nonnull}\}$$
is a measurable relation on $X$. However, this is of limited
value because, first, there typically exist measurable relations that do not
derive from subsets of $X^2$ in the above manner; the diagonal
measurable relation (Definition \ref{meastypes} (a)) on any nonatomic
measure space is an example of this phenomenon. Second, distinct subsets
of $X^2$ (that is, distinct even modulo null sets)
can give rise to the same measurable relation. For example,
there exists a measurable subset $R$ of $[0,1]^2$, of measure
strictly less than 1, that has positive measure intersection with every
positive measure subset of the form $S \times T$. The measurable relation
derived from such a set in the manner suggested above is the same as the
measurable relation derived from the full relation $[0,1]^2$. Such a set
$R$ can be constructed as the set of pairs $(x,y) \in [0,1]^2$ such that
$x - y$ belongs to a dense open subset $U$ of $[-1,1]$ with measure
strictly less than 2. For if $S$ and $T$ are any positive
measure subsets of $[0,1]$, let $s$ be a Lebesgue point of $S$ and
$t$ a Lebesgue point of $T$, and find $\epsilon > 0$ such that
$$\mu((s - \epsilon , s + \epsilon) \cap S),
\mu((t - \epsilon , t + \epsilon) \cap T) > \frac{3\epsilon}{2}.$$
Then for all $a \in [-\epsilon,\epsilon]$ we have
$$\mu((S + t + a) \cap I) > \frac{\epsilon}{2}\qquad{\rm and}\qquad
\mu((T + s)\cap I) > \frac{3\epsilon}{2}$$
where $I$ is the interval $(s + t -\epsilon, s + t + \epsilon)$ of
length $2\epsilon$. It follows that
$\mu((S + t + a) \cap (T + s)) > 0$ for all $a \in [-\epsilon,\epsilon]$.
Since $U$ is dense in $[-1, 1]$, find $a \in [-\epsilon,\epsilon]$ such that
$s - t - a \in U$, and let $u$ be a Lebesgue point of $(S + t + a)\cap
(T + s)$; then $s' = u - t - a$ and $t' = u - s$ are respectively
Lebesgue points of $S$ and $T$, and $s' - t' = s - t - a \in U$.
It follows that for sufficiently small $\delta$ the sets
$(s' - \delta, s' + \delta) \cap S$ and
$(t' - \delta, t' + \delta) \cap T$ have positive measure and their
product is contained in $R$. Thus $(S\times T) \cap R$ has
positive measure, as claimed.

We can, however, always convert measurable relations to pointwise relations
in the following standard way. Let $(X,\mu)$ be a finitely decomposable
measure space and
recall that the {\it carrier space} $\Omega$ of $L^\infty(X,\mu)$
is the set of nonzero homomorphisms from $L^\infty(X,\mu)$ into $\Cb$,
that $\Omega$ is a compact Hausdorff space with weak* topology inherited
from the dual of $L^\infty(X,\mu)$, and that there is a natural isomorphism
$\Phi: L^\infty(X,\mu) \cong C(\Omega)$ (\cite{Tak}, Theorem I.4.4).

(If $\mu$ is $\sigma$-finite then we can go further:
letting $g \in L^1(X,\mu)$ be positive and nowhere-zero, integration
against $g$ on $L^\infty(X,\mu)$
corresponds to a bounded positive linear functional on $C(\Omega)$,
and hence is given by integration against a regular Borel measure
$\nu$ on $\Omega$. That is,
$$\int fg\, d\mu = \int \Phi(f)\, d\nu$$
for all $f \in L^\infty(X,\mu)$. Then $C(\Omega) \cong L^\infty(\Omega,\nu)$
by essentially the identity map, so that we can regard $\Phi$ as an
isomorphism between $L^\infty(X,\mu)$ and $L^\infty(\Omega,\nu)$
(\cite{Tak}, Theorem III.1.18).)

If $p$ is a projection in $L^\infty(X,\mu)$ and $\phi \in \Omega$ then
$\phi(p) = 0$ or $1$. Now if $\Rc$ is a measurable relation on $X$ then
we can define a corresponding relation $R_\Rc$ on $\Omega$ by setting
$(\phi,\psi) \in R_\Rc$ if $(p,q) \in \Rc$
for every pair of projections $p$ and $q$ in $L^\infty(X,\mu)$ such that
$\phi(p) = \psi(q) = 1$.

\begin{theo}\label{pointwise}
Let $(X,\mu)$ be a $\sigma$-finite measure space,
let $\Rc_1$, $\Rc_2$, and $\Rc_3$ be measurable relations on $X$,
and let $R_i = R_{\Rc_i}$ ($i = 1,2,3$) be the corresponding relations on
$\Omega$ as defined above. Then

\noindent (a) $R_1$ is a closed relation on $\Omega$

\noindent (b) $\Rc_1 \subseteq \Rc_2$ $\Leftrightarrow$ $R_1 \subseteq R_2$

\noindent (c) $\Rc_1$ is the diagonal measurable relation $\Leftrightarrow$
$R_1$ is the diagonal relation

\noindent (d) $\Rc_1$ is the transpose of $\Rc_2$ $\Leftrightarrow$
$R_1$ is the transpose of $R_2$

\noindent (e) $\Rc_3$ is the product of $\Rc_1$ and $\Rc_2$ $\Leftrightarrow$
$R_3$ is the product of $R_1$ and $R_2$.
\end{theo}

\begin{proof}
$R_1$ is closed because if $\phi_\lambda \to \phi$ and $\psi_\lambda
\to \psi$ weak* with $(\phi_\lambda, \psi_\lambda) \in R_1$ for
all $\lambda$, and $p$ and $q$ are projections with $\phi(p) = \psi(q) = 1$,
then eventually $\phi_\lambda(p) = \psi_\lambda(q) = 1$, which implies that
$(p,q) \in \Rc_1$; this shows that $(\phi,\psi) \in R_1$.

It is clear that $\Rc_1 \subseteq \Rc_2$ implies $R_1 \subseteq R_2$.
For the converse, we claim that $(p,q) \in \Rc_1$ if and only if
there exist $\phi, \psi \in \Omega$ such that $\phi(p) = \psi(q) = 1$
and $(\phi,\psi) \in R_1$; this obviously yields that $R_1 \subseteq R_2$
implies $\Rc_1 \subseteq \Rc_2$.
The reverse implication in the claim is immediate from the definition of
$R_1$. For the forward implication, suppose $(p,q) \in \Rc_1$, say
$p = \chi_S$ and $q = \chi_T$. For any
finite partitions $\Sc = \{S_1, \ldots, S_m\}$ of $S$ and
$\Tc = \{T_1, \ldots, T_n\}$ of $T$, choose $1 \leq i \leq m$
and $1 \leq j \leq n$ such that $(\chi_{S_i},\chi_{T_j}) \in \Rc_1$
and define $\phi_{\Sc,\Tc}, \psi_{\Sc,\Tc}: L^\infty(X,\mu) \to \Cb$ by
$$\phi_{\Sc,\Tc}(f) = \frac{1}{\mu(S_i)} \int_{S_i}f\, d\mu\quad
{\rm and}\quad\psi_{\Sc,\Tc}(f) = \frac{1}{\mu(T_j)} \int_{T_j}f\, d\mu.$$
The limit $(\phi,\psi)$ of any weak* convergent subnet of the net
$\{(\phi_{\Sc,\Tc},\psi_{\Sc,\Tc})\}$
will then be a pair of complex homomorphisms with the property that
$\phi(p) = \psi(q) = 1$. To see that $(\phi,\psi) \in R_1$ let
$p'$ and $q'$ be any projections in $L^\infty(X,\mu)$
such that $\phi(p') = \psi(q') = 1$.
Say $p' = \chi_{S'}$ and $q' = \chi_{T'}$. Then the net
$\{(\phi_{\Sc,\Tc},\psi_{\Sc,\Tc})\}$ has the property that eventually
$S' \cap S$ is a finite union of blocks in the partition $\Sc$, and
similarly for $T' \cap T$, so that $\phi(p') = \psi(q') = 1$ implies
that in the subnet that converges to $(\phi,\psi)$,
eventually $S_i \subseteq S'$ and $T_j \subseteq T'$ for the
choice of $i$ and $j$ used to define $\phi_{\Sc,\Tc}$ and $\psi_{\Sc,\Tc}$.
This shows
that $(p',q') \in \Rc_1$, and we conclude that $(\phi,\psi) \in R_1$.
This completes the proof that $R_1 \subseteq R_2$ implies
$\Rc_1 \subseteq \Rc_2$.

Observe that part (b) implies $\Rc_1 = \Rc_2$ $\Leftrightarrow$
$R_1 = R_2$.

Next, suppose $\Rc_1$ is the diagonal measurable relation on $X$ and let
$\phi \in \Omega$.
For any projections $p$ and $q$ with $\phi(p) = \phi(q) = 1$ we must have
$\phi(pq) = 1$ since $\phi$ is a homomorphism,
and therefore $pq \neq 0$. Thus $(p,q) \in \Rc_1$ for any
$p$ and $q$ with $\phi(p) = \phi(q) = 1$, so that $(\phi,\phi) \in R_1$.
Conversely, if $\phi, \psi \in \Omega$ are distinct then there is a
projection $p$ such that $\phi(p) \neq \psi(p)$. Say $\phi(p) = 1$ and
$\psi(p) = 0$. Then $\psi(1-p) = 1$, and the pair $(p,1-p)$ does not
belong to $\Rc_1$, so $(\phi,\psi)$ cannot belong to $R_1$. Thus $R_1$ is the
diagonal relation on $\Omega$.

It is easy to see that the measurable transpose of $\Rc_1$ is taken to the
classical transpose of $R_1$. To verify compatibility with products,
fix $\phi,\theta \in \Omega$; we must show that $(\phi,\theta) \in R_1R_2$
if and only if $(p,r) \in \Rc_1\Rc_2$
for all projections $p$ and $r$ such that $\phi(p) = \theta(r) = 1$.
For the forward implication, let $p$ and $r$ be projections such that
$\phi(p) = \theta(r) = 1$ and suppose
that $(p,r) \not\in \Rc_1\Rc_2$. Let $q'$ be the join
of all projections $p'$ such that $(p,p') \not\in \Rc_1$ and let $q''$ be the
join of all projections $r'$ such that $(r',r) \not\in \Rc_2$. Then
$(p,q') \not\in \Rc_1$ and $(q'',r) \not\in \Rc_2$, and we must have
$q' \vee q'' = 1$ as otherwise $1 - (q' \vee q'')$ would witness
$(p,r) \in \Rc_1\Rc_2$. Then for any $\psi \in \Omega$ we have either
$\psi(q') = 1$ or $\psi(q'') = 1$, which implies that either $(\phi,\psi)
\not\in R_1$ or $(\psi,\theta) \not\in R_2$. This shows that
$(\phi,\theta) \not\in R_1R_2$, which completes the proof of the forward
implication. For the reverse implication, suppose that $(p,r) \in \Rc_1\Rc_2$
for all projections $p$ and $r$ such that $\phi(p) = \theta(r) = 1$.
For any finite partition $\Sc = \{S_1, \ldots, S_m\}$ of $X$ we have
$\phi(\chi_{S_i}) = \theta(\chi_{S_k}) = 1$ for precisely one choice of
$i$ and $k$. Since $(\chi_{S_i},\chi_{S_k}) \in \Rc_1\Rc_2$ there exists
a nonzero
projection $q$ such that $(\chi_{S_i},q') \in \Rc_1$ and $(q',\chi_{S_k})
\in \Rc_2$ for every $q' \leq q$. Choose a value of $j$ such that
$q\chi_{S_j} \neq 0$; then $(\chi_{S_i},\chi_{S_j}) \in \Rc_1$ and
$(\chi_{S_j},\chi_{S_k}) \in \Rc_2$. Define $\psi_\Sc: L^\infty(X,\mu)
\to \Cb$ by
$$\psi_\Sc(f) = \frac{1}{\mu(S_j)} \int_{S_j} f\, d\mu$$
and let $\psi$ be a weak* cluster point of the net $\{\psi_\Sc\}$.
Then as in an earlier part of the proof we have
$(\phi,\psi) \in R_1$ and $(\psi,\theta) \in R_2$, so
$(\phi,\theta) \in R_1R_2$. This completes the proof of the
reverse implication.
\end{proof}

It immediately follows that all of the notions introduced in Definition
\ref{measthings} reduce to their classical analogs under the
conversion described above.

For our purposes there is no need to convert to pointwise relations,
so we will not use Theorem \ref{pointwise}. Nonetheless, it may be cited
as evidence that measurable relations are a reasonable generalization of
pointwise relations to the measurable setting.

\subsection{Basic results}\label{br}
In this section we present three basic tools for working with measurable
relations. The first is easy but the other two are more substantive.

\begin{prop}\label{reduce}
Let $\Rc$ be a measurable relation on a finitely decomposable measure
space and suppose $(p,q) \in \Rc$. Then there
exist nonzero projections $p' \leq p$ and $q' \leq q$ such that
$$(p',q'')\in\Rc\quad{\rm and}\quad (p'', q') \in \Rc$$
for all nonzero projections $p'' \leq p'$ and $q'' \leq q'$.
\end{prop}

\begin{proof}
Let $r$ be the join of $\{r' \leq p: (r',q) \not\in \Rc\}$ and let
$s$ be the join of $\{s' \leq q: (p,s') \not\in \Rc\}$; then take
$p'= p - r$ and $q' = q - s$. By the definition of measurable
relations we must have $(r,q), (p,s) \not\in \Rc$, so that $r \neq p$
and $s \neq q$. Therefore $p'$ and $q'$ are nonzero.

Now let $p'' \leq p'$ be nonzero. Then $p'' \not\leq r$, so
$(p'',q) \in \Rc$. But $(p'',s) \not\in \Rc$ because $(p,s) \not\in \Rc$
and $p'' \leq p$; since $q = s \vee q'$ and $(p'',q) \in \Rc$ this
implies $(p'', q') \in \Rc$. The analogous statement for nonzero
$q'' \leq q'$ holds by symmetry.
\end{proof}

Say that a projection $p$ in $L^\infty(X,\mu)$ is an {\it atom}
if there is no projection $q$ with $0 < q < p$, and is {\it nonatomic}
if it does not lie above any atoms. Also, say that $p$ has {\it finite
measure} if $\int p\, d\mu < \infty$.

\begin{lemma}\label{partit}
Let $\Rc$ be a measurable relation on a finitely decomposable measure
space and suppose $(p,q) \in \Rc$ with $p$ and
$q$ both nonatomic with finite measure. Say $p = \chi_S$ and $q = \chi_T$.
Then there exists $a > 0$ such that any finite partitions
$\{S_1, \ldots, S_m\}$ and $\{T_1, \ldots, T_n\}$ of $S$ and $T$
can be refined to partitions $\{S_1', \ldots, S_{m'}'\}$ and
$\{T_1',\ldots, T_{n'}'\}$ with the properties that (1) we have
$$2\cdot \min_{i,j}\{\mu(S_i'), \mu(T_j')\} \geq
\max_{i,j}\{\mu(S_i'), \mu(T_j')\}$$
and (2) after reordering, there exists $k$ such that
$\mu(S_1' \cup \cdots \cup S_k') \geq a$ and
$(\chi_{S_l'}, \chi_{T_l'}) \in \Rc$ for $1 \leq l \leq k$.
\end{lemma}

\begin{proof}
Fix partitions $\{S_1,\ldots, S_m\}$ and $\{T_1, \ldots, T_n\}$
of $S$ and $T$. Let $b = \min\{\mu(S_i),\mu(T_j)\}$.
Since $p$ and $q$ are nonatomic and have finite measure we can
refine to partitions $\{S_1',\ldots, S_{m'}'\}$ and
$\{T_1', \ldots, T_{n'}'\}$ such that $b \leq \mu(S_i'), \mu(T_j')
\leq 2b$ for all $i$ and $j$ (see \cite{Tak}, Theorem III.1.22).
We then follow a greedy algorithm, pairing $S$'s and $T$'s that
belong to the relation until no further pair can be found. Then after
reordering there exists $k$ such that $(\chi_{S_l'}, \chi_{T_l'}) \in \Rc$
for $1 \leq l \leq k$ and $(\chi_{S_i'}, \chi_{T_j'}) \not\in \Rc$ for any
$i,j > k$.

Let $a = \mu(S_1' \cup \cdots \cup S_k')$. We have to show that there
is a positive lower bound on $a$ independent of the construction we just
performed and of the original choice of partitions of $S$ and $T$.
Suppose not. Then for each $n$ we can find some such construction with
$a \leq 2^{-n}$. Let $U_n = S_{k+1}' \cup \cdots \cup S_{m'}'$ and
$V_n = T_{k+1}' \cup \cdots \cup T_{n'}'$. Then $(\chi_{U_n},\chi_{V_n}) \not\in \Rc$
since $(\chi_{S_i'}, \chi_{T_j'}) \not\in \Rc$ for any $i,j > k$.
So for every $N$ we have
$$\left(\bigvee_{n \geq N}\chi_{U_n}, \bigwedge_{n \geq N}\chi_{V_n}\right)
\not\in \Rc,$$
and $p = \bigvee_{n \geq N} \chi_{U_n}$, so $(p, \bigwedge_{n \geq N}
\chi_{V_n}) \not\in \Rc$. Since $q = \bigvee_N \bigwedge_{n \geq N}
\chi_{V_n}$ this implies $(p,q) \not\in \Rc$, a contradiction. Therefore
there is a positive lower bound on $a$, as claimed.
\end{proof}

For $f \in L^\infty(X,\mu)$ let $M_f \in B(L^2(X,\mu))$ be the multiplication
operator $M_f: g \mapsto fg$.

\begin{theo}\label{connect}
Let $\Rc$ be a measurable relation on a finitely decomposable measure
space and suppose $(p,q) \in \Rc$ with $p = \chi_S$ and $q = \chi_T$.
Then there is a nonzero bounded
operator $A: L^2(T, \mu|_T) \to L^2(S, \mu|_S)$ such that $f \geq 0$ implies
$Af \geq 0$ and
$$(p', q') \not\in \Rc \quad\Rightarrow\quad
M_{p'}AM_{q'}= 0$$
for $p' \leq p$ and $q' \leq q$.
\end{theo}

\begin{proof}
Let $r$ be the join of the atoms lying under $p$. If $(r,q) \in \Rc$
then there is an atom $\chi_{S'} \leq r$ such that $(\chi_{S'},q) \in \Rc$.
Applying Proposition \ref{reduce}, we can find a nonzero finite measure
projection $q' \leq q$ such that
$(\chi_{S'}, q'') \in \Rc$ for every $q'' \leq q'$. Writing $q' =
\chi_{T'}$, we can then take $A$ to be the operator
$$f \mapsto \left(\int_{T'} f\, d\mu\right)\cdot\chi_{S'}.$$
This has the desired properties. So we can assume $(r,q) \not\in \Rc$,
and replacing $p$ with $p - r$ we may assume $p$ is nonatomic. A similar
argument reduces to the case where $q$ is also nonatomic.

Since $p$ is the join of the finite measure projections lying under $p$ and
the same is true of $q$, we can find finite measure projections $p' \leq p$
and $q' \leq q$ such that $(p',q') \in \Rc$. Replacing $p$ with $p'$
and $q$ with $q'$, we may now suppose that $p$ and $q$ are both
nonatomic with finite measure. We can therefore apply Lemma \ref{partit}.
Fix $a > 0$ as in this lemma and for any finite partitions
$\Sc = \{S_1, \ldots, S_m\}$ of $S$ and
$\Tc = \{T_1, \ldots, T_n\}$ of $T$, let $\{S_1', \ldots, S_{m'}'\}$
and $\{T_1', \ldots, T_{n'}'\}$ be the refined partitions and $k$ the
integer provided by the lemma. Then define
$A_{\Sc,\Tc}: L^2(T, \mu|_T) \to L^2(S,\mu|_S)$ by
$$A_{\Sc, \Tc}: f \mapsto \sum_{l = 1}^k \frac{1}{\mu(T_l')}
\left(\int_{T_l'} f\, d\mu\right)\cdot \chi_{S_l'}.$$

The operator norm of $A_{\Sc,\Tc}$ is
$$\|A_{\Sc,\Tc}\| = \max_{1 \leq l \leq k}
\sqrt{\frac{\mu(S_l')}{\mu(T_l')}} \leq \sqrt{2}$$
so the net $\{A_{\Sc, \Tc}\}$ is bounded. Let $A$ be a weak operator cluster
point of this net. We immediately have $f \geq 0 \Rightarrow Af \geq 0$
since this is true of every $A_{\Sc, \Tc}$. Also, $A \neq 0$ because
$$\langle A_{\Sc,\Tc}(\chi_T), \chi_S\rangle =
\int_S A_{\Sc,\Tc}(\chi_T)\, d\mu = \mu(S_1' \cup \cdots \cup S_k') \geq a$$
for all $\Sc, \Tc$, which implies that $\langle A(\chi_T), \chi_S\rangle \geq a$.
Finally, let $S' \subseteq S$ and $T' \subseteq T$ satisfy
$(\chi_{S'},\chi_{T'}) \not\in \Rc$. Then for any partitions
$\Sc$ and $\Tc$ which respectively refine the partitions
$\{S', S -S'\}$ and $\{T', T-T'\}$
we have $M_{\chi_{S'}}A_{\Sc,\Tc}M_{\chi_{T'}} = 0$ by construction
(since no projection under $\chi_{S'}$ will be paired with a projection
under $\chi_{T'}$). Taking the weak operator limit then shows that
$M_{\chi_{S'}}AM_{\chi_{T'}} = 0$.
\end{proof}

Our last result in this section is the most powerful.
Its proof uses basic von Neumann algebra techniques.

\begin{theo}\label{derivation}
Let $\Rc$ be a measurable relation on a finitely decomposable measure
space $(X,\mu)$ and suppose $(p,q) \in \Rc$. Then there is a finitely
decomposable measure space $(Y,\nu)$ and a pair of unital weak* continuous
$*$-homomorphisms $\pi_l,\pi_r: L^\infty(X,\mu) \to L^\infty(Y,\nu)$
such that $\pi_l(p) = \pi_r(q) = 1$ and
$$(p',q') \not\in \Rc\quad\Rightarrow\quad
\pi_l(p')\pi_r(q') = 0.$$
\end{theo}

\begin{proof}
Say $p = \chi_S$ and $q = \chi_T$ and let $A$ be the operator provided
by Theorem \ref{connect}. As in the proof of Theorem \ref{connect}
we may assume $\mu(S), \mu(T) < \infty$. Define
a linear functional $\tau$ on the algebraic tensor product
$\Ac = L^\infty(S,\mu|_S) \otimes L^\infty(T,\mu|_T)$ by
$$\tau(f\otimes g) = \int fAg\, d\mu|_S = \langle Ag, \bar{f}\rangle$$
(and extending linearly). We claim that $\tau$ is positive in the
sense that
$$\tau\left( \left(\sum f_i\otimes g_i\right)
\left(\sum \bar{f}_i\otimes\bar{g}_i\right)\right) \geq 0$$
for any
$\sum f_i \otimes g_i \in \Ac$.
By continuity it is enough to check this when the $f_i$ and $g_i$ are simple.
So fix partitions $\{S_1,\ldots, S_m\}$ and $\{T_1, \ldots, T_n\}$ of
$S$ and $T$ and suppose the $f_i$ and $g_i$ are constant on each
$S_k$ and each $T_l$, respectively. For $1 \leq k \leq m$ and $1 \leq l
\leq n$ let $a_{kl} = \langle A(\chi_{T_l}), \chi_{S_k}\rangle$ and
observe that each $a_{kl} \geq 0$ since $f \geq 0$ $\Rightarrow$
$Af \geq 0$.

Say $f_i = \sum b_{ik}\chi_{S_k}$ and $g_i = \sum c_{il}\chi_{T_l}$.
We now compute
\begin{eqnarray*}
\tau\left( \left(\sum_i f_i\otimes g_i\right)
\left(\sum_i \bar{f}_i\otimes\bar{g}_i\right)\right)
&=& \sum_{i,j} \langle A(g_i\bar{g}_j), \bar{f}_i f_j\rangle\cr
&=& \sum_{i,j,k,l} b_{ik}\bar{b}_{jk}c_{il}\bar{c}_{jl}a_{kl}\cr
&=& \sum_{k,l} a_{kl} \left| \sum_i b_{ik}c_{il}\right|^2\cr
&\geq& 0.
\end{eqnarray*}
This establishes the claim.

Now define a pre-inner product on $\Ac$ by setting
$$\langle f\otimes g,f'\otimes g'\rangle
= \tau((f\otimes g)(\bar{f}'\otimes\bar{g}'))$$
(and extending linearly)
and let $H$ be the Hilbert space formed by factoring out null vectors
and completing. Define representations $\pi_l,\pi_r: L^\infty(X,\mu)
\to \Bc(H)$ by $\pi_l(h)(f\otimes g) = h|_S f\otimes g$ and
$\pi_r(h)(f\otimes g) = f\otimes h|_T g$. The representation $\pi_l$ is
both well-defined and contractive by the calculuation
\begin{eqnarray*}
\left\| \pi_l(h)\left(\sum_i f_i\otimes g_i\right)\right\|_H^2
&=& \tau\left(\sum_{i,j} |h|_S|^2f_i\bar{f}_j\otimes g_i\bar{g}_j\right)\cr
&\leq& \tau\left(\sum_{i,j} \|h\|_\infty^2 f_i\bar{f}_j\otimes g_i\bar{g}_j\right)\cr
&=& \|h\|_\infty^2\cdot \tau\left(\sum_{i,j} f_i\bar{f}_j\otimes g_i\bar{g}_j\right)\cr
&=& \|h\|_\infty^2 \left\|\sum_i f_i\otimes g_i\right\|_H^2,
\end{eqnarray*}
and a similar calculuation shows the same for $\pi_r$. (The
inequality that replaces $|h|_S|^2$ with $\|h\|^2_\infty$ follows
from positivity of $\tau$, or it can be proven first
for simple $f,g,h$ by a computation similar to the one used above to
prove that $\tau$ is positive.) The sets
$\pi_l(L^\infty(X,\mu))$ and $\pi_r(L^\infty(X,\mu))$ generate an
abelian von Neumann algebra $\Mx \cong L^\infty(Y,\nu)$. If
$(p',q') \not\in \Rc$ then $M_{p'}AM_{q'} = 0$ implies that
$\pi_l(p')\pi_r(q') = 0$ because
\begin{eqnarray*}
\langle \pi_l(p')\pi_r(q')(f\otimes g), f'\otimes g'\rangle
&=& \tau((p'|_S f\otimes q'|_T g)(\bar{f}'\otimes\bar{g}'))\cr
&=& \langle A(q'|_T g\bar{g}'), p'|_S\bar{f}f')\cr
&=& \langle (M_{p'}AM_{q'})g\bar{g}', \bar{f}f'\rangle\cr
&=& 0
\end{eqnarray*}
for all $f,f' \in L^\infty(S,\mu|_S)$ and $g,g' \in L^\infty(T,\mu|_T)$.
All other properties of $\pi_l$ and $\pi_r$ are routinely verified.
\end{proof}

\subsection{Measurable metrics}
Measurable metric spaces were introduced in \cite{W0} and have
subsequently been studied in connection with derivations
\cite{W4, W5, W3} and local Dirichlet forms \cite{H, H1, H2}.
Unfortunately, as was pointed out by Francis Hirsch, there is an error
in one of the basic results, Lemma 6.1.6 of \cite{W4}, which was used
heavily in developing the general theory of these structures.
Most of the resulting problems can be fixed without too much trouble ---
some of this has been done in \cite{H} --- but in some places the faulty
lemma was really used in an essential way.

The machinery we developed in Section \ref{br}, particularly Theorem
\ref{derivation}, can be used to quickly correct all of the problems
stemming from the use of the erroneous lemma. We do this now.

We recall the basic definition:

\begin{defi}\label{measpm}
(\cite{W4}, Definition 6.1.3)
Let $(X,\mu)$ be a finitely decomposable measure space and let
$\Pc$ be the set of nonzero projections in $L^\infty(X,\mu)$.
A {\it measurable pseudometric} on $(X,\mu)$ is
a function $\rho: \Pc^2 \to [0,\infty]$ such that
\begin{quote}
\noindent (i) $\rho(p, p) = 0$

\noindent (ii) $\rho(p,q) = \rho(q,p)$

\noindent (iii) $\rho(\bigvee p_\lambda, \bigvee q_\kappa) =
\inf_{\lambda,\kappa} \rho(p_\lambda, q_\kappa)$

\noindent (iv) $\rho(p,r) \leq \sup_{q' \leq q} (\rho(p,q') + \rho(q',r))$
\end{quote}
for all $p,q,r,p_\lambda,q_\kappa \in \Pc$. It is a
{\it measurable metric} if for all
disjoint $p$ and $q$ there exist nets $\{p_\lambda\}$ and $\{q_\lambda\}$
such that $p_\lambda \to p$ and $q_\lambda \to q$ weak*
and $\rho(p_\lambda, q_\lambda) > 0$ for all $\lambda$.
\end{defi}

If either $p$ or $q$ is (or both are) the zero projection then the
appropriate convention is $\rho(p,q) = \infty$. (Note that in the
measurable triangle inequality, property (iv), $q'$ ranges over
nonzero projections.)

Say that the {\it closure} of $p$ is the complement of
$\bigvee\{q: \rho(p,q) > 0\}$, and say that $p$ is {\it closed} if
it equals its closure. Equivalently, $p$ is closed if for every nonzero
projection $q$ that is disjoint from $p$ there exists a nonzero
projection $q' \leq q$
such that $\rho(p,q') > 0$. Then $\rho$ is a measurable metric
if and only if the closed projections generate $L^\infty(X,\mu)$ as
a von Neumann algebra. This was actually the definition of measurable
metric used in \cite{W4}, and the equivalence with the condition given
above follows from (\cite{W4}, Theorem 6.2.10).

We begin with a simple observation connecting measurable metrics
to measurable relations. Its proof is an easy verification.

\begin{lemma}\label{metricrel}
Let $(X,\mu)$ be a finitely decomposable measure space, let $\rho$ be
a measurable pseudometric on $X$, and let $t > 0$. Then
$$\Rc_t = \{(p,q): \rho(p,q) < t\}$$
is a reflexive, symmetric measurable relation on $X$. We have
$\Rc_s\Rc_t \subseteq \Rc_{s + t}$ for all $s,t > 0$.
\end{lemma}

\begin{proof}
The first assertion is straightforward. To verify the second, suppose
$(p,r) \in \Rc_s\Rc_t$ and find a nonzero projection $q$ such that
$(p,q') \in \Rc_s$ and $(q',r) \in \Rc_t$ for every nonzero $q' \leq q$
(Proposition \ref{measverify} (c)). Then $\rho(p,q) < s$, so let
$\epsilon = (s - \rho(p,q))/2$ and define $q_1 = q(1 - \tilde{q})$ where
$$\tilde{q} = \bigvee \{q': \rho(p,q') \geq \rho(p,q) + \epsilon\}.$$
Note that
$$\rho(p,q\tilde{q}) \geq \rho(p,\tilde{q}) \geq \rho(p,q) + \epsilon$$
so $q\tilde{q} \neq q$, i.e., $q_1$ is nonzero. We then have that
$\rho(p,q') < \rho(p,q) + \epsilon = s - \epsilon$ for all $q' \leq q_1$,
and we still have $\rho(q', r) < t$ for all $q' \leq q_1$ since $q_1 \leq q$.
Then the measurable triangle inequality (Definition \ref{measpm} (iv))
implies $\rho(p,r) \leq s + t - \epsilon$, so that $(p,r) \in \Rc_{s + t}$.
\end{proof}

It follows that using the method of Theorem \ref{pointwise} we can
convert measurable pseudometrics on $(X,\mu)$ into pointwise pseudometrics
on the carrier space of $L^\infty(X,\mu)$. This result is an improved
version of Theorem 6.3.9 of \cite{W4}. We retain the notation of Section
\ref{reduct}.

\begin{theo}
Let $\rho$ be a measurable pseudometric on a finitely decomposable
measure space $(X,\mu)$. Then $d_\rho: \Omega^2 \to [0,\infty]$ defined by
$$d_\rho(\phi,\psi) = \sup\{\rho(p,q): \phi(p) = \psi(q) = 1\}$$
is a pseudometric on $\Omega$, and we have
$$\rho(p,q) = \inf\{d_\rho(\phi,\psi): \phi(p) = \psi(q) = 1\}$$
for any nonzero projections $p$ and $q$ in $L^\infty(X,\mu)$.
\end{theo}

\begin{proof}
For any $\phi \in \Omega$ and any projections $p,q \in L^\infty(X,\mu)$,
if $\phi(p) = \phi(q) = 1$ then $\phi(pq) = 1$ and hence
$pq \neq 0$; by Definition \ref{measpm} (i) and (iii)
this implies $\rho(p,q) = 0$. This shows that
$d_\rho(\phi,\phi) = 0$ for all $\phi \in \Omega$. Symmetry of $d_\rho$ follows
immediately from symmetry of $\rho$. For the triangle inequality, let
$\phi,\psi,\theta \in \Omega$ and fix projections $p,r \in L^\infty(X,\mu)$
such that $\phi(p) = \theta(r) = 1$. We may assume that
$d_\rho(\phi, \psi), d_\rho(\psi,\theta) < \infty$.
Let $\epsilon > 0$, let $p'$
be the join of the projections $p''$ such that $\rho(p,p'') \geq
d_\rho(\phi,\psi) + \epsilon$, and let $r'$ be the join of the projections
$r''$ such that $\rho(r'',r) \geq d_\rho(\psi,\theta) + \epsilon$. Then
$\rho(p,p') \geq d_\rho(\phi,\psi) + \epsilon$, and since $\phi(p) = 1$
we must therefore have $\psi(p') = 0$ by the definition of
$d_\rho(\phi,\psi)$. Similarly $\psi(r') = 0$, so letting $q' = (1-p')(1-r')$
we must have $\psi(q') = \psi(1-p')\psi(1-r') = 1$. Therefore $q' \neq 0$, so
$$\rho(p,r) \leq  \sup_{q'' \leq q'} (\rho(p,q'') + \rho(q'',r))
\leq d_\rho(\phi,\psi) + d_\rho(\psi,\theta) + 2\epsilon$$
since any $q'' \leq q'$ is disjoint from both $p'$ and $r'$. Taking
$\epsilon \to 0$ and then taking the supremum over $p$ and $r$ yields
$d_\rho(\phi,\theta) \leq d_\rho(\phi,\psi) + d_\rho(\psi,\theta)$.
So $d_\rho$ is a pseudometric on $\Omega$.

Let $p$ and $q$ be nonzero projections in $L^\infty(X,\mu)$.
It is immediate from the definition of $d_\rho(\phi,\psi)$ that
$\rho(p,q) \leq d_\rho(\phi,\psi)$ for any $\phi,\psi \in \Omega$
satisfying $\phi(p) = \psi(q) = 1$. Conversely, given $\epsilon > 0$
we need to find $\phi,\psi \in \Omega$ such that $\phi(p) = \psi(q) = 1$
and $d_\rho(\phi,\psi) \leq \rho(p,q) + \epsilon$. We may assume
$\rho(p,q) < \infty$.
Use Lemma \ref{metricrel} and Theorem \ref{derivation} to find a
finitely decomposable measure space $(Y,\nu)$ and a
pair of unital weak* continuous $*$-homomorphisms $\pi_l,\pi_r:
L^\infty(X,\mu) \to L^\infty(Y,\nu)$ such that $\pi_l(p) =\pi_r(q) =
1$ and $\pi_l(p')\pi_r(q') = 0$ for any nonzero projections
$p'$ and $q'$ with $\rho(p',q') \geq \rho(p,q) + \epsilon$. Let
$\phi$ be any nonzero homomorphism from $L^\infty(Y,\nu)$ to $\Cb$; then
$\phi \circ\pi_l$ and $\phi\circ\pi_r$ belong to $\Omega$ and
satisfy $\phi\circ\pi_l(p) = \phi\circ\pi_r(q) = 1$. Also, if
$p'$ and $q'$ are any nonzero projections in $L^\infty(X,\mu)$
such that $\phi\circ\pi_l(p') = \phi\circ\pi_r(q') = 1$ then
$\phi(\pi_l(p')\pi_r(q')) = 1$, so $\pi_l(p')\pi_r(q') \neq 0$,
and hence $\rho(p',q') < \rho(p,q) + \epsilon$.
Thus $d_\rho(\phi\circ\pi_l,\phi\circ\pi_r)
\leq \rho(p,q) + \epsilon$, as desired.
\end{proof}

Next, we define measurable Lipschitz numbers.
Recall that the {\it essential range} of a function
$f \in L^\infty(X,\mu)$ is the set of all $a \in \Cb$ such that
$f^{-1}(U)$ has positive measure for every open neighborhood $U$
of $a$. Equivalently, it is the spectrum of the operator
$M_f \in \Bc(L^2(X,\mu))$. If $p \in L^\infty(X,\mu)$ is a projection
then we denote the essential range of $f|_{{\rm supp}(p)}$ by
${\rm ran}_p(f)$.

\begin{defi}\label{measlipnum}
(\cite{W4}, Definition 6.2.1)
Let $(X,\mu)$ be a finitely decomposable measure space and let
$\rho$ be a measurable pseudometric on $X$. The {\it Lipschitz
number} of $f \in L^\infty(X,\mu)$ is the quantity
$$L(f) = \sup\left\{\frac{d({\rm ran}_p(f), {\rm ran}_q(f))}{\rho(p,q)}\right\},$$
where the supremum is taken over all nonzero projections $p,q \in
L^\infty(X,\mu)$ and we use the convention $\frac{0}{0} = 0$. Here $d$
is the usual (minimum) distance between compact subsets of $\Cb$.
We call $L$ the {\it Lipschitz gauge} associated to $\rho$ and we
define ${\rm Lip}(X,\mu) = \{f \in L^\infty(X,\mu): L(f) < \infty\}$.
\end{defi}

Now we introduce the key tool for studying Lipschitz numbers.

\begin{defi}\label{deLeeuw}
(\cite{W4}, Definition 6.3.1)
Let $(X,\mu)$ be a finitely decomposable measure space and let $\rho$
be a measurable pseudometric on $X$. For any nonzero projections
$p,q \in L^\infty(X,\mu)$ and any $\epsilon > 0$, let $\Rc_{p,q,\epsilon}$
be the measurable relation defined in Lemma \ref{metricrel} with
$t = \rho(p,q) + \epsilon$, so that $(p,q) \in \Rc_{p,q,\epsilon}$, and find
$\pi^{p,q,\epsilon}_l$, $\pi^{p,q,\epsilon}_r$, and $(Y_{p,q,\epsilon},
\nu_{p,q,\epsilon})$
as in Theorem \ref{derivation}. Then let $Y = \coprod Y_{p,q,\epsilon}$
and $\nu = \coprod \nu_{p,q,\epsilon}$, and for $f \in L^\infty(X,\mu)$ define
$\Phi(f)$ on $Y$ by
$$\Phi(f)|_{Y_{p,q,\epsilon}} = \frac{\pi^{p,q,\epsilon}_l(f) -
\pi^{p,q,\epsilon}_r(f)}{\rho(p,q) + \epsilon}.$$
Also define $\pi_l, \pi_r: L^\infty(X,\mu) \to L^\infty(Y,\nu)$
by $\pi_l = \bigoplus \pi^{p,q,\epsilon}_l$ and $\pi_r = \bigoplus
\pi^{p,q,\epsilon}_r$. We call $\Phi$ a {\it measurable de Leeuw map}.
\end{defi}

\begin{theo}\label{fixmm}
(\cite{W4}, Theorem 6.3.2)
Let $(X,\mu)$ be a finitely decomposable measure space, let $\rho$
be a measurable pseudometric on $X$, and let $\Phi$ be a measurable
de Leeuw map.

\noindent (a) For all $f \in L^\infty(X,\mu)$ we have $L(f) =
\|\Phi f\|_\infty$.

\noindent (b) $\Phi$ is linear and we have $\Phi(fg) = \pi_l(f)\Phi(g)
+ \Phi(f)\pi_r(g)$ for all $f,g \in L^\infty(X,\mu)$.

\noindent (c) The graph of $\Phi: {\rm Lip}(X,\mu) \to L^\infty(Y,\nu)$
is weak* closed in $L^\infty(X \coprod Y)$.
\end{theo}

\begin{proof}
(a) Let $p, q \in L^\infty(X,\mu)$ be nonzero projections. Since
$\pi^{p,q,\epsilon}_l(p) = \pi^{p,q,\epsilon}_r(q) = 1_{Y_{p,q,\epsilon}}$,
the essential ranges of $\pi_l^{p,q,\epsilon}(f)$ and $\pi_r^{p,q,\epsilon}(f)$
are respectively contained in ${\rm ran}_p(f)$ and ${\rm ran}_q(f)$, so
$$\|\pi^{p,q,\epsilon}_l(f) - \pi^{p,q,\epsilon}_r(f)\|_\infty \geq
d({\rm ran}_p(f), {\rm ran}_q(f)).$$
Thus
$$\|\Phi(f)\|_\infty \geq
\frac{d({\rm ran}_p(f), {\rm ran}_q(f))}
{\rho(p,q) + \epsilon};$$
taking $\epsilon \to 0$ and the supremum over $p$ and $q$ then
yields $\|\Phi(f)\|_\infty \geq L(f)$.

To verify the reverse inequality, fix $p$, $q$, and $\epsilon$; we will
show that $\|\Phi(f)|_{Y_{p,q,\epsilon}}\|_\infty \leq L(f)$. We may assume
that $\|\pi_l^{p,q,\epsilon}(f) - \pi_r^{p,q,\epsilon}(f)\|_\infty > 0$. Let
$0 < \delta < \|\pi_l^{p,q,\epsilon}(f) - \pi_r^{p,q,\epsilon}(f)\|_\infty/2$
and partition $p$ and $q$ as $p = \sum_1^m p_i$,
$q = \sum_1^n q_j$ so that ${\rm ran}_{p_i}(f)$ and ${\rm ran}_{q_j}(f)$
have diameter at most $\delta$ for all $i$ and $j$. Then for some
choice of $i$ and $j$ we must have
$\pi_l^{p,q,\epsilon}(p_i)\pi_r^{p,q,\epsilon}(q_j) \neq 0$ and
$$d({\rm ran}_{p_i}(f), {\rm ran}_{q_j}(f)) \geq
\|\pi_l^{p,q,\epsilon}(f) - \pi_r^{p,q,\epsilon}(f)\|_\infty - 2\delta.$$
Then $\rho(p_i,q_j) \leq \rho(p,q) + \epsilon$ and hence
$$\frac{\|\pi^{p,q,\epsilon}_l(f) - \pi^{p,q,\epsilon}_r(f)\|_\infty}
{\rho(p,q) + \epsilon} \leq
\frac{d({\rm ran}_{p_i}(f), {\rm ran}_{q_j}(f)) + 2\delta}
{\rho(p_i,q_j)}.$$
If $\rho(p_i,q_j) > 0$ then taking $\delta \to 0$ shows
that $\|\Phi(f)|_{Y_{p,q,\epsilon}}\|_\infty \leq L(f)$; if
$\rho(p_i,q_j) = 0$ then $L(f)= \infty$ since $d({\rm ran}_{p_i}(f),
{\rm ran}_{q_j}(f)) > 0$, so again
$\|\Phi(f)|_{Y_{p,q,\epsilon}}\|_\infty \leq L(f)$. We conclude
that $\|\Phi(f)\|_\infty \leq L(f)$.

(b) This is trivially verified on each $Y_{p,q,\epsilon}$ separately.

(c) Let $\{f_\lambda\}$ be a net in ${\rm Lip}(X,\mu)$ and suppose
$f_\lambda \oplus \Phi(f_\lambda) \to f \oplus g$ weak* in
$L^\infty(X \coprod Y)$. Restricting to each $Y_{p,q,\epsilon}$
shows that $g = \Phi(f)$. Then $L(f) = \|g\|_\infty < \infty$ by part (a)
so $f \in {\rm Lip}(X,\mu)$, and we conclude that the graph of $\Phi$
is weak* closed.
\end{proof}

\begin{coro}\label{basiclip}
(\cite{W4}, Lemma 6.2.6 and Theorem 6.2.7)
Let $(X,\mu)$ be a finitely decomposable measure space and let $\rho$
be a measurable pseudometric on $X$.

\noindent (a) $L(af) = |a|\cdot L(f)$, $L(\bar{f}) = L(f)$,
$L(f + g) \leq L(f) + L(g)$, and
$L(fg) \leq \|f\|_\infty L(g) + \|g\|_\infty L(f)$ for all $f,g \in
{\rm Lip}(X,\mu)$ and $a \in \Cb$.

\noindent (b) If $\{f_\lambda\} \subseteq L^\infty(X,\mu)$ is a net that
converges weak* to $f \in L^\infty(X,\mu)$ then
$L(f) \leq \sup L(f_\lambda)$.

\noindent (c) ${\rm Lip}(X,\mu)$ is a self-adjoint unital subalgebra of
$L^\infty(X,\mu)$.
It is a dual Banach space for the norm $\|f\|_L = \max\{\|f\|_\infty, L(f)\}$.

\noindent (d) The real part of the unit ball of ${\rm Lip}(X,\mu)$ is
a complete sublattice of the real part of the unit ball of $L^\infty(X,\mu)$.
\end{coro}

\begin{proof}
(a) This follows easily from Theorem \ref{fixmm} (a) and (b).

(b) If $f_\lambda \to f$ weak* then $\Phi(f_\lambda)|_{Y_{p,q,\epsilon}}
\to \Phi(f)|_{Y_{p,q,\epsilon}}$ weak* for each $p$, $q$, $\epsilon$.
It follows that
$$L(f) = \|\Phi(f)\|_\infty \leq \sup \|\Phi(f_\lambda)\|_\infty =
\sup L(f_\lambda).$$

(c) The first assertion follows immediately from part (a) and the
second assertion follows from Theorem \ref{fixmm} (c) because
$\|f\|_L = \|f \oplus \Phi(f)\|_\infty$, so that ${\rm Lip}(X,\mu)$
equipped with this norm is isometric to the graph of $\Phi$.

(d) If $f,g \in {\rm Lip}(X,\mu)$ are real-valued then $L(f \vee g) \leq
\max\{L(f), L(g)\}$ because
$$d({\rm ran}_p(f \vee g), {\rm ran}_q(f \vee g)) \leq
\max\{d({\rm ran}_p(f), {\rm ran}_q(f)), d({\rm ran}_p(g), {\rm ran}_q(g)\}$$
for all $p$ and $q$, and $L(f \wedge g) \leq \max\{L(f), L(g)\}$
similarly.
So the real part of the unit ball of ${\rm Lip}(X,\mu)$ is a sublattice of
the real part of the unit ball of $L^\infty(X,\mu)$, and it is then a complete
sublattice by part (b) since $\bigvee f_\lambda$ and $\bigwedge f_\lambda$
are respectively the weak* limits of the net of finite joins of the
$f_\lambda$ and the net of finite meets of the $f_\lambda$.
\end{proof}

We include one more fundamental result.

\begin{lemma}\label{abdistfn}
Let $(X,\mu)$ be a finitely decomposable measure space and let
$\rho$ be a measurable pseudometric on $X$. Let $r \in L^\infty(X,\mu)$
be a nonzero projection and let $c > 0$. Then the function
$$\bigvee \min\{\rho(p,r), c\}\cdot p,$$
taking the join in $L^\infty(X,\mu)$ over all nonzero projections $p$,
has Lipschitz number at most 1.
\end{lemma}

\begin{proof}
Let $f$ be this join, let $p$ and $q$ be nonzero projections, and let
$\epsilon > 0$; we must show that $d({\rm ran}_p(f), {\rm ran}_q(f))
\leq \rho(p,q) + \epsilon$. Let $\Rc$ be the measurable relation
defined in Lemma \ref{metricrel} with $t = \rho(p,q) + \epsilon$
and let $p'$ and $q'$ be the projections provided by Lemma \ref{reduce}.
Find $p'' \leq p'$ such that ${\rm ran}_{p''}(f)$ has
diameter at most $\epsilon$, then apply Lemma \ref{reduce} to the
pair $(p'', q')$, then do the same thing with $p$'s and $q$'s reversed.
The result is a pair of nonzero projections $p_1 \leq p$ and $q_1 \leq q$
such that ${\rm ran}_{p_1}(f)$ and ${\rm ran}_{q_1}(f)$
both have diameter at most $\epsilon$ and $\rho(p_1,q_2),\rho(p_2,q_1)
< \rho(p,q) + \epsilon$ for every nonzero $p_2 \leq p_1$, $q_2 \leq q_1$.

Let $a = \rho(p_1, r)$ and $b = \rho(q_1, r)$. We may assume $b \leq a \leq c$.
Now apply Lemmas \ref{metricrel} and
\ref{reduce} to $q_1$ and $r$ to find $q_2 \leq q_1$ such
that $\rho(q_3,r) \leq \rho(q_1,r) + \epsilon$ for all $q_3 \leq q_2$.
Then by the measurable triangle inequality we have
$$a = \rho(p_1, r) \leq \sup_{q_3 \leq q_2}(\rho(p_1, q_3) + \rho(q_3,r))
\leq \rho(p,q) + b + 2\epsilon.$$
We claim that ${\rm ran}_{p_1}(f) \subseteq [a,a + \epsilon]$
and ${\rm ran}_{q_2}(f) \subseteq [b, b+\epsilon]$. This is because,
first, it is immediate from the definition of $f$ that $f \geq ap_1$,
and second, Lemmas \ref{metricrel} and \ref{reduce} guarantee, for any
$\delta > 0$, the existence of a nonzero projection $p_2 \leq p_1$
such that any nonzero projection under $p_2$ has distance at most
$a + \delta$ to $r$, so that $fp_2 \leq (a + \delta)p_2$.
This shows that $a \in {\rm ran}_{p_1}(f)$, and since ${\rm ran}_{p_1}(f)$
has diameter at most $\epsilon$
we conclude that ${\rm ran}_{p_1}(f) \subseteq [a,a+\epsilon]$.
The same argument applies to ${\rm ran}_{q_1}(f)$. We therefore have
$$d({\rm ran}_p(f), {\rm ran}_q(f))
\leq a-b \leq \rho(p,q) + 2\epsilon,$$
which is enough.
\end{proof}

The join in Lemma \ref{abdistfn} is a measurable version of the
pointwise distance function $x \mapsto \min\{d(x,S), c\}$. Note that
as long as there exists a nonzero projection $p$ with $0 < \rho(p,r)
< \infty$ the reverse inequality is easy, i.e., the Lipschitz number
of the join is exactly 1.

\section{Quantum relations}\label{qr}

We now proceed to our definition of a quantum relation on a von
Neumann algebra in terms of a bimodule over the commutant. The rough
intuition is that the bimodule consists of the operators that only
connect pairs of points that belong to the relation.

\subsection{Basic definitions}\label{basic}
Let $H$ be a complex Hilbert space, not necessarily separable. Recall
that the {\it weak*} (or {\it $\sigma$-weak}) topology on $\Bc(H)$ is the
weak topology arising from the pairing $\langle A,B\rangle \mapsto
{\rm tr}(AB)$ of $\Bc(H)$ with the trace class operators $\TC(H)$;
that is, it is the weakest topology that makes the map $A \mapsto
{\rm tr}(AB)$ continuous for all $B \in \TC(H)$. The weak* topology is
finer than the weak operator topology but the two agree on bounded sets. 
A {\it dual operator space} is a weak* closed subspace $\Vc$ of $\Bc(H)$;
it is a {\it W*-bimodule} over a von Neumann algebra $\Mx \subseteq \Bc(H)$
if $\Mx\Vc\Mx \subseteq \Vc$.

We will refer to \cite{Tak} for standard facts about von Neumann algebras.

\begin{defi}\label{quantrel}
A {\it quantum relation} on a von Neumann algebra $\Mx \subseteq \Bc(H)$
is a W*-bimodule over its commutant $\Mx'$, i.e., it is a weak* closed
subspace $\Vc \subseteq \Bc(H)$ satisfying $\Mx'\Vc\Mx' \subseteq \Vc$.
\end{defi}

The generalization to a quantum relation on a pair of von Neumann
algebras $\Mx \subseteq \Bc(H)$ and $\Nc \subseteq \Bc(K)$ would be:
a weak* closed subspace $\Vc \subseteq \Bc(K,H)$ satisfying
$\Mx'\Vc\Nc' \subseteq \Vc$. We need not develop this more general
notion separately since quantum relations on $\Mx$ and $\Nc$ can be
identified with quantum relations on the direct sum $\Mx\oplus\Nc \subseteq
\Bc(H \oplus K)$ satisfying $\Vc = I_H\Vc I_K$ (i.e., that live in the
$(H,K)$ corner of $\Bc(H \oplus K)$).

As we noted above, the intuition is that $\Vc$ consists of the
operators that only connect pairs of points that belong to the relation.
In the atomic abelian case this is exactly right. Recall the notations
$e_x = \chi_{\{x\}}$ and $M_f: g \mapsto fg$. Also let $V_{xy}$ be
the rank one operator $V_{xy}: g \mapsto \langle g,e_y\rangle e_x$
on $l^2(X)$.

\begin{prop}\label{atomiccase}
Let $X$ be a set and let $\Mx \cong l^\infty(X)$ be the von Neumann
algebra of bounded multiplication operators on $l^2(X)$. If
$R$ is a relation on $X$ then
\begin{eqnarray*}
\Vc_R &=& \{A \in \Bc(l^2(X)): (x,y) \not\in R\quad \Rightarrow\quad
\langle Ae_y, e_x\rangle = 0\}\cr
&=& \overline{\rm span}^{wk^*}\{V_{xy}: (x,y) \in R\}
\end{eqnarray*}
is a quantum relation on $\Mx$; conversely, if $\Vc$ is a
quantum relation on $\Mx$ then
$$R_\Vc = \{(x,y) \in X^2: \langle Ae_y, e_x\rangle\neq 0
\hbox{ for some }A \in \Vc\}$$
is a relation on $X$. The two constructions are inverse to each other.
\end{prop}

\begin{proof}
Note first that $\Mx = \Mx'$ in this case. Let $R \subseteq X^2$
and define $\Vc_R$ to be the set of operators $A$ such that
$\langle Ae_y,e_x\rangle = 0$ for all $(x,y) \not\in R$. Then it is
clear that $\Vc_R$ is a linear subspace of $\Bc(l^2(X))$. Also, $\Vc_R$ is
weak operator closed and therefore weak* closed. Finally, if $M_f$ and $M_g$
are any two multiplication operators in $\Mx$ then
$$\langle Ae_y, e_x\rangle = 0\quad \Rightarrow\quad
\langle M_fAM_ge_y, e_x\rangle =
f(x)g(y)\langle Ae_y,e_x\rangle= 0,$$
which shows that $A \in \Vc_R \Rightarrow M_fAM_g \in \Vc_R$.
So $\Vc_R$ is a W*-bimodule over $\Mx' = \Mx$.

We verify that $\Vc_R = \overline{\rm span}^{wk^*}\{V_{xy}: (x,y) \in R\}$.
We have $(x,y) \in R$ $\Rightarrow$ $V_{xy} \in \Vc_R$ because
$\langle V_{xy}e_{y'},e_{x'}\rangle \neq 0$ only if $x = x'$ and $y = y'$;
since
$\Vc_R$ is a weak* closed subspace of $\Bc(l^2(X))$ this proves the inclusion
$\supseteq$. For the reverse inclusion let $A \in \Vc_R$ and for any finite
subset $F \subseteq X$ let $P_F \in \Bc(l^2(X))$ be the orthogonal projection
onto ${\rm span}\{e_x: x \in F\} \subseteq l^2(X)$.
Then $P_FAP_F \to A$ boundedly weak
operator and hence weak*, so it will suffice to show that each $P_FAP_F$
belongs to ${\rm span}\{V_{xy}: (x,y) \in R\}$. But $P_FAP_F$ is a linear
combination of operators of the form $M_{e_x} A M_{e_y}$ for $x,y \in X$,
which are scalar multiples of the operators $V_{xy}$. Moreover,
$$M_{e_x}AM_{e_y} \neq 0\quad\Rightarrow\quad
\langle Ae_y,e_x\rangle \neq 0\quad\Rightarrow\quad
(x,y) \in R.$$
So $P_FAP_F$ is a linear combination of operators $V_{xy}$ with
$(x,y) \in R$, as desired. This proves the inclusion $\subseteq$.

The second assertion of the proposition is trivial: $R_\Vc$ is a
subset of $X^2$ directly from its definition.

Now let $R$ be a relation, let $\Vc = \Vc_R$, and let $\tilde{R} = R_\Vc$.
It is immediate that $\tilde{R} \subseteq R$. Conversely, let
$(x,y) \in R$; then $V_{xy}$
belongs to $\Vc$ and satisfies $\langle V_{xy}e_y, e_x\rangle \neq 0$,
so $(x,y) \in \tilde{R}$. Thus $R = \tilde{R}$.

Finally, let $\Vc$ be a quantum relation, let $R = R_\Vc$, and let
$\tilde{\Vc} = \Vc_R$. It is immediate that $\Vc \subseteq \tilde{\Vc}$.
For the reverse inclusion it will suffice to show that $V_{xy} \in \Vc$
for all $(x,y) \in R$. But if $(x,y) \in R$ then we must have
$\langle Be_y, e_x\rangle \neq 0$ for some $B \in \Vc$.   
Then $M_{e_x} B M_{e_y}$ is a nonzero scalar multiple of $V_{xy}$, and
it belongs to $\Vc$ since $\Vc$ is a bimodule over $\Mx$. So
$V_{xy} \in \Vc$, as desired.
\end{proof}

\subsection{Constructions with quantum relations}
Next we consider basic constructions that can be performed with
quantum relations. The following proposition is trivial.

\begin{prop}\label{quantverify}
Let $\Mx \subseteq \Bc(H)$ be a von Neumann algebra.

\noindent (a) The commutant $\Mx'$ is a quantum relation on $\Mx$.

\noindent (b) If $\Vc$ is a quantum relation on $\Mx$ then so is
$\Vc^* = \{A^*: A \in \Vc\}$.

\noindent (c) If $\Vc$ and $\Wc$ are quantum relations on $\Mx$ then
so is the weak* closure of their algebraic product.

\noindent (d) The intersection of any family of quantum relations on
$\Mx$ is a quantum relation on $\Mx$.

\noindent (e) The weak* closed sum of any family of quantum relations
on $\Mx$ is a quantum relation on $\Mx$.
\end{prop}

This justifies the following definition.

\begin{defi}\label{quanttypes}
Let $\Mx \subseteq \Bc(H)$ be a von Neumann algebra.

\noindent (a) The {\it diagonal quantum relation} on $\Mx$ is the
quantum relation $\Vc = \Mx'$.

\noindent (b) The {\it transpose} of a quantum relation $\Vc$ on
$\Mx$ is the quantum relation $\Vc^*$.

\noindent (c) The {\it product} of two quantum relations $\Vc$ and
$\Wc$ on $\Mx$ is the weak* closure of their algebraic product.

\noindent (d) A quantum relation $\Vc$ on $\Mx$ is
\begin{quote}
(i) {\it reflexive} if $\Mx' \subseteq \Vc$

\noindent (ii) {\it symmetric} if $\Vc^* = \Vc$

\noindent (iii) {\it antisymmetric} if $\Vc \cap \Vc^* \subseteq \Mx'$

\noindent (iv) {\it transitive} if $\Vc^2 \subseteq \Vc$.
\end{quote}
\end{defi}

We immediately note that Proposition \ref{atomiccase} reduces the
preceding notions to the classical ones in the atomic abelian case.

\begin{prop}\label{atomicprops}
Let $X$ be a set and let $\Mx \cong l^\infty(X)$ be the von Neumann
algebra of bounded multiplication operators on $l^2(X)$.
Also let $R_1$, $R_2$, and $R_3$ be relations on $X$ and let $\Vc_i
= \Vc_{R_i}$ ($i = 1,2,3$) be the corresponding quantum relations
on $\Mx$ as in Proposition \ref{atomiccase}. Then

\noindent (a) $R_1 \subseteq R_2$ $\Leftrightarrow$ $\Vc_1 \subseteq \Vc_2$

\noindent (b) $R_1$ is the diagonal relation $\Leftrightarrow$
$\Vc_1$ is the diagonal quantum relation

\noindent (c) $R_1$ is the transpose of $R_2$ $\Leftrightarrow$
$\Vc_1$ is the transpose of $\Vc_2$

\noindent (d) $R_3$ is the product of $R_1$ and $R_2$ $\Leftrightarrow$
$\Vc_3$ is the product of $\Vc_1$ and $\Vc_2$.
\end{prop}

The proof of this proposition is straightforward.

Using Definition \ref{quanttypes} we can define quantum versions
of equivalence relations, preorders, partial orders, and graphs.

\begin{defi}\label{quantthings}
Let $\Mx \subseteq \Bc(H)$ be a von Neumann algebra.

\noindent (a) A {\it quantum equivalence relation} on $\Mx$ is a
reflexive, symmetric, transitive quantum relation on $\Mx$. That is,
it is a von Neumann algebra that contains $\Mx'$.

\noindent (b) A {\it quantum preorder} on $\Mx$ is a reflexive,
transitive quantum relation on $\Mx$. That is, it is a weak* closed
operator algebra that contains $\Mx'$.

\noindent (c) A {\it quantum partial order} on $\Mx$ is a reflexive,
antisymmetric, transitive quantum relation on $\Mx$. That is, it is
a weak* closed operator algebra $\Ac$ such that $\Ac \cap \Ac^* =
\Mx'$.

\noindent (d) A {\it quantum graph} on $\Mx$ is a reflexive, symmetric
quantum relation on $\Mx$. That is, it is a weak* closed operator
system that is a bimodule over $\Mx'$.
\end{defi}

Note that by the double commutant theorem, von Neumann algebras
containing $\Mx'$ correspond to von Neumann algebras
contained in $\Mx$. So quantum equivalence relations on $\Mx$
correspond to von Neumann subalgebras of $\Mx$. This is the
expected definition.

If $\Vc$ is a quantum preorder on $\Mx$ then
$\Vc\cap \Vc^*$ is a quantum equivalence relation on $\Mx$, i.e.,
$\Vc \cap \Vc^*$ is the commutant of some von Neumann subalgebra
$\Mx_0 \subseteq \Mx$. Then $\Vc$ is a quantum partial order on
$\Mx_0$. Passing from $\Mx$ to $\Mx_0$ is the quantum version of
factoring out equivalent
elements to turn a preorder into a partial order.

As we noted following Definition \ref{measthings}, a graph can
classically be encoded as a reflexive, symmetric relation.
This justifies our definition of a quantum graph. 

Definition \ref{quantthings} becomes especially simple when
$\Mx = \Bc(H)$; in that case $\Mx' = \Cb I$, so that
\begin{quote}
\noindent $\bullet$ a quantum relation on $\Bc(H)$ is a dual
operator space in $\Bc(H)$;

\noindent $\bullet$ a quantum equivalence relation on $\Bc(H)$
is a von Neumann algebra in $\Bc(H)$;

\noindent $\bullet$ a quantum preorder on $\Bc(H)$ is a weak*
closed unital operator algebra in $\Bc(H)$;

\noindent $\bullet$ a quantum partial order on $\Bc(H)$ is a weak*
closed operator algebra $\Ac$ in $\Bc(H)$ satisfying $\Ac \cap \Ac^*
= \Cb I$;

\noindent $\bullet$ a quantum graph on $\Bc(H)$ is a dual
operator system in $\Bc(H)$.
\end{quote}

For finite dimensional $H$, this definition of quantum graph was proposed
in \cite{DSW}.

\subsection{Basic results}
Next we show that although the definition of a quantum relation is
framed in terms of a particular representation, the notion is
in fact representation independent. This is slightly
surprising because the W*-bimodules over $\Mx$ do vary with
the representation of $\Mx$: if we add multiplicity to a representation
(i.e., tensor with the identity on a nontrivial Hilbert space) the set
of bimodules over $\Mx$ that are contained in $\Bc(H)$ grows. But the
commutant also grows, in such a way that the set of bimodules over
$\Mx'$ does not essentially change.

\begin{theo}\label{repindep}
Let $H_1$ and $H_2$ be Hilbert spaces and let $\Mx_1 \subseteq \Bc(H_1)$
and $\Mx_2 \subseteq \Bc(H_2)$ be isomorphic von Neumann algebras. Then
there is a 1-1 correspondence between the quantum relations on
$\Mx_1$ and the quantum relations on $\Mx_2$ which respects the
conditions $\Vc \subseteq \Wc$, $\Vc = \Mx'$, $\Vc^* =
\Wc$, and $\Uc\Vc = \Wc$.
\end{theo}

\begin{proof}
Let $K$ be any nonzero Hilbert space. Then
$$\Mx_i \cong I_K \otimes \Mx_i\subseteq \Bc(K \otimes H_i)$$
($i = 1,2$). If the cardinality of $K$ is large enough then the
representations of $I_K \otimes \Mx_1$ and $I_K \otimes \Mx_2$ in
$\Bc(K \otimes H_1)$ and $\Bc(K \otimes H_2)$, respectively, are spatially
equivalent (\cite{Tak}, Theorem IV.5.5). So it is sufficient to consider
the case where $H_2 = K \otimes H_1$ and $\Mx_2 = I_K \otimes \Mx_1$.
We then have $\Mx_2' = \Bc(K) \overline{\otimes} \Mx_1'$ (\cite{Tak},
Theorem IV.5.9).

Given a W*-bimodule $\Vc \subseteq \Bc(H_1)$ over $\Mx_1'$, let
$\Bc(K) \overline{\otimes} \Vc \subseteq \Bc(K \otimes H_1)$ denote
the normal spatial tensor product, i.e., the weak* closure of the
algebraic tensor product in $\Bc(K \otimes H_1)$. It is clear that
the map $\Vc \mapsto \Bc(K) \overline{\otimes} \Vc$ respects the
conditions listed in the statement of the theorem. We must show that
(1) $\Bc(K) \overline{\otimes} \Vc$ is a W*-bimodule over $\Mx_2' =
\Bc(K) \overline{\otimes} \Mx_1'$; (2) if $\Wc$ is a distinct
W*-bimodule over $\Mx_1'$ then $\Bc(K) \overline{\otimes} \Vc \neq
\Bc(K) \overline{\otimes} \Wc$; and (3) every W*-bimodule over
$\Bc(K) \overline{\otimes} \Mx_1'$ is of the form
$\Bc(K) \overline{\otimes} \Vc$ for some W*-bimodule $\Vc$ over $\Mx_1'$.

It is clear that $\Bc(K) \overline{\otimes} \Vc$ is a weak* closed operator
space in $\Bc(K \otimes H_1)$. Now $\Vc$ is a bimodule over $\Mx_1'$, so it
is immediate that the algebraic tensor product $\Bc(K) \otimes \Vc$ is
a bimodule over the algebraic tensor product $\Bc(K) \otimes \Mx_1'$.
Taking weak* limits then shows that $\Bc(K) \overline{\otimes} \Vc$
is a bimodule over $\Bc(K) \overline{\otimes} \Mx_1'$. This verifies (1).

To verify (2), let $P$ be a rank 1 projection in $\Bc(K)$. We claim that
$$(P \otimes I_{H_1})(\Bc(K) \overline{\otimes} \Vc)(P \otimes I_{H_1})
= P \otimes \Vc$$
where $P \otimes \Vc = \{P \otimes A: A \in \Vc\}$. Now
$(P \otimes I_{H_1})(\Bc(K) \otimes \Vc)(P \otimes I_{H_1})
\subseteq P \otimes \Vc$ is clear, where $\Bc(K) \otimes \Vc$ is the
algebraic tensor product, and taking weak* limits
therefore establishes the inclusion $\subseteq$. The reverse
inclusion is trivial. This proves the claim and shows that
$\Vc \neq \Wc$ (hence $P \otimes \Vc \neq P \otimes \Wc$)
implies $\Bc(K) \overline{\otimes} \Vc \neq \Bc(K) \overline{\otimes} \Wc$.
Moreover, $\Vc \not\subseteq \Wc$ implies $\Bc(K)\overline{\otimes} \Vc
\not\subseteq \Bc(K) \overline{\otimes} \Wc$.

Finally, let $\tilde{\Vc}\subseteq \Bc(K \otimes H_1)$ be a W*-bimodule over
$\Bc(K) \overline{\otimes} \Mx_1'$ and let
$$\Vc = \{A \in \Bc(H_1): P \otimes A \in \tilde{\Vc}\}.$$
To prove (3) we will show that $\Vc$ is a W*-bimodule over $\Mx_1'$
and $\tilde{\Vc} = \Bc(K) \overline{\otimes} \Vc$. The first part is easy:
$\Vc$ is clearly a weak* closed operator space in $\Bc(H_1)$, and it is a
bimodule over $\Mx_1'$ because $\tilde{\Vc}$ is a bimodule over
$P \otimes \Mx_1' \subseteq \Bc(K) \overline{\otimes} \Mx_1'$. For the second
part, observe first that $P \otimes \Vc \subseteq \tilde{\Vc}$;
since $\tilde{\Vc}$ is a bimodule over $\Bc(K) \otimes I_{H_1} \subseteq
\Bc(K) \overline{\otimes} \Mx_1'$, multiplying on the left and the right
by operators of the form $B \otimes I_{H_1}$
and taking linear combinations yields $A \otimes \Vc \subseteq \tilde{\Vc}$
for any finite rank operator $A \in \Bc(K)$, and taking weak* limits
then shows that $\Bc(K) \overline{\otimes} \Vc \subseteq \tilde{\Vc}$.
Conversely, given $A \in \tilde{\Vc}$ it will suffice to show that
$$(Q \otimes I_{H_1})A(Q \otimes I_{H_1}) \in \Bc(K) \overline{\otimes} \Vc$$
for any finite rank projection $Q \in \Bc(K)$, as these operators
converge weak* to $A$. Then by linearity it is enough to show that
$$(Q_1 \otimes I_{H_1})A(Q_2 \otimes I_{H_1})
\in \Bc(K) \overline{\otimes} \Vc$$
for any rank 1 projections $Q_1$ and $Q_2$. But letting $V_1$ and $V_2$
be rank 1 partial isometries in $\Bc(K)$ such that $V_1V_1^* = V_2V_2^* = P$,
$V_1^* V_1 = Q_1$, and $V_2^* V_2 = Q_2$, we have
$$(V_1 \otimes I_{H_1})A(V_2^* \otimes I_{H_1}) = P \otimes B$$
for some $B \in \Vc$, and then
$$V_1^*V_2 \otimes B =
(V_1^* \otimes I_{H_1})(P \otimes B)(V_2 \otimes I_{H_2}) =
(Q_1 \otimes I_{H_1})A(Q_2 \otimes I_{H_1}).$$
So $(Q_1 \otimes I_{H_1})A(Q_2 \otimes I_{H_1})$ does belong to
$\Bc(K) \overline{\otimes} \Vc$.
\end{proof}

The following separation lemma will also be useful in the sequel.

\begin{lemma}\label{separation}
Let $\Vc$ be a quantum relation on a von Neumann algebra $\Mx \subseteq
\Bc(H)$ and let $A \in \Bc(H) - \Vc$. Then there is a pair of projections
$P$ and $Q$ in $\Mx \overline{\otimes} \Bc(l^2) \subseteq \Bc(H \otimes l^2)$
such that
$$P(A \otimes I)Q \neq 0$$
but
$$P(B \otimes I)Q = 0$$
for all $B \in \Vc$.
\end{lemma}

\begin{proof}
Since $\Vc$ is weak* closed there is a weak* continuous linear functional
on $\Bc(H)$ that annihilates $\Vc$ but not $A$. Thus (\cite{Tak}, p.\ 67)
there exist a pair of vectors $v$ and $w$ in $H \otimes l^2$ such that
$$\langle (A \otimes I)w,v\rangle \neq 0$$
but
$$\langle (B \otimes I)w,v\rangle = 0$$
for all $B \in \Vc$. Moreover, since $\Vc$ is a bimodule over $\Mx'$, we
have
$$\langle (B \otimes I)w', v'\rangle = 0$$
for all $B \in \Vc$, $v' \in (\Mx' \otimes I)v$, and
$w' \in (\Mx' \otimes I)w$. Let
$P,Q \in \Bc(H \otimes l^2)$ be the orthogonal projections onto the
closures of $(\Mx' \otimes I)v$ and $(\Mx' \otimes I)w$, respectively.
Then we immediately have
$$P(A \otimes I)Q \neq 0$$
and
$$P(B \otimes I)Q = 0$$
for all $B \in \Vc$. Also, by their construction the ranges of $P$ and
$Q$ are invariant for every operator in $\Mx' \otimes I$, hence $P$ and
$Q$ commute with every operator in $\Mx' \otimes I$, hence
$P,Q \in \Mx \overline{\otimes} B(l^2)$ (\cite{Tak}, Theorem IV.5.9).
\end{proof}

\subsection{The abelian case}\label{abcs}
Next we connect quantum relations to measurable relations when $\Mx$ is
abelian. This section is closely related to Arveson's
celebrated paper \cite{Arv}, and we will show in the next section that
some of Arveson's main results can easily be derived from ours. Actually,
the converse is also true, by a simple application of the linking algebra
construction: if $\Vc$ is a quantum relation on a von
Neumann algebra $\Mx \subseteq \Bc(H)$ then
$$\Ac = \left\{\left[
\begin{matrix}
A&B\cr
0&C
\end{matrix}
\right] \in \Bc(H \oplus H): A,C \in \Mx'\hbox{ and }B \in \Vc\right\}$$
is a unital weak* closed operator algebra that contains $\Mx' \oplus \Mx'$.
In this way quantum relations can be converted into operator algebras,
and using this device we could without too much effort deduce the main
results of this section from \cite{Arv}. However, we prefer to give
direct proofs (which are also not hard, given the machinery we have
already built up).

\begin{theo}\label{abelianrel}
Let $(X,\mu)$ be a finitely decomposable measure space and let
$\Mx \cong L^\infty(X,\mu)$ be the von Neumann algebra
of bounded multiplication operators on $L^2(X,\mu)$.
If $\Rc$ is a measurable relation on $X$ then
$$\Vc_\Rc = \{A \in \Bc(L^2(X,\mu)): (p,q) \not\in \Rc
\quad \Rightarrow \quad M_pAM_q = 0\}$$
is a quantum relation on $\Mx$; conversely, if $\Vc$ is a quantum
relation on $\Mx$ then
$$\Rc_\Vc = \{(p,q): M_pAM_q\neq 0\hbox{ for some }A \in \Vc\}$$
is a measurable relation on $X$. We have $\Rc = \Rc_{\Vc_\Rc}$ for
any measurable relation $\Rc$ on $X$ and $\Vc \subseteq
\Vc_{\Rc_\Vc}$ for any quantum relation $\Vc$ on $\Mx$.

If $\Vc_1$, $\Vc_2$, and $\Vc_3$ are quantum relations on $\Mx$
and $\Rc_i = \Rc_{\Vc_i}$ ($i = 1,2,3$) then

\noindent (a) $\Vc_1 \subseteq \Vc_2$ $\Rightarrow$ $\Rc_1 \subseteq
\Rc_2$

\noindent (b) $\Vc_1$ is the diagonal quantum relation $\Leftrightarrow$
$\Rc_1$ is the diagonal measurable relation

\noindent (c) $\Vc_1$ is the transpose of $\Vc_2$ $\Rightarrow$
$\Rc_1$ is the transpose of $\Rc_2$

\noindent (d) $\Vc_3$ is the product of $\Vc_1$ and $\Vc_2$ $\Rightarrow$
$\Rc_3$ is the product of $\Rc_1$ and $\Rc_2$.
\end{theo}

\begin{proof}
Note first that $\Mx = \Mx'$ in this case. Let $\Rc$ be a measurable
relation on $X$. It is clear that $\Vc_\Rc$ is a linear subspace of
$\Bc(L^2(X,\mu))$. Also, $\Vc_\Rc$ is weak operator closed
and therefore weak* closed. Finally, if $M_f$ and $M_g$ are any two
multiplication operators in $\Mx$ then
$$M_pAM_q = 0 \quad\Rightarrow\quad
M_p(M_fAM_g)M_q = M_f(M_pAM_q)M_g = 0,$$
which shows that $A \in \Vc_\Rc \Rightarrow M_fAM_g \in \Vc_\Rc$. So
$\Vc_\Rc$ is a W*-bimodule over $\Mx' = \Mx$.

Next, let $\Vc$ be a quantum relation on $\Mx$. We verify that $\Rc_\Vc$
satisfies the pair of conditions stated in Definition \ref{measrel}.
First, if $p' \leq p$, $q' \leq q$, and $M_{p'}AM_{q'} \neq 0$ then it is
clear that $M_pAM_q \neq 0$. Second, say $p = \bigvee p_\lambda$ and
$q = \bigvee q_\kappa$ and suppose $M_{p_\lambda} A M_{q_\kappa} = 0$
for all $\lambda$ and $\kappa$. Then $\langle Aw,v\rangle = 0$ for
all $v$ in the range of any $M_{p_\lambda}$ and all $w$ in the range of
any $M_{q_\kappa}$. Taking linear combinations and norm limits then
yields $\langle Aw,v\rangle = 0$ for all $v$ in the range of
$M_p$ and all $w$ in the range of $M_q$, so that $M_pAM_q = 0$. This
verifies the second condition and shows that $\Rc_\Vc$ is a measurable
relation.

Now let $\Rc$ be a measurable relation, let $\Vc = \Vc_\Rc$, and let
$\tilde{\Rc} = \Rc_\Vc$. It is immediate that $\tilde{\Rc}
\subseteq \Rc$. For the reverse inclusion, let $(p,q) \in \Rc$
and say $p = \chi_S$ and $q = \chi_T$. By Theorem \ref{connect}
there exists a nonzero bounded operator $A: L^2(T, \mu|_T) \to
L^2(S, \mu|_S)$ such that $(p', q') \not\in \Rc$ implies $M_{p'}AM_{q'}= 0$.
Extending $A$ to be zero on $L^2(X - T, \mu|_{X - T})$, we get an
operator $\tilde{A} \in B(L^2(X,\mu))$ which satisfies $M_p\tilde{A}M_q
\neq 0$ and $M_{p'}\tilde{A}M_{q'} = 0$ for all projections $p'$
and $q'$ such that $(p',q')\not\in \Rc$. Then $\tilde{A} \in \Vc$ and
this shows that $(p,q) \in \tilde{\Rc}$, so we conclude that
$\Rc = \tilde{\Rc}$.

All of the remaining assertions except for part (d) and the reverse
implication in
(b) are straightforward. For the reverse implication in (b), observe
that $\Rc_1 = \Delta$ implies that $(1-p, p) \not\in \Rc_1$ for any $p$,
and hence that every operator in $\Vc_1$ commutes with every projection
in $\Mx$. Since $\Mx$ is maximal abelian, it follows that $\Vc_1$ is contained
in, and therefore a weak* closed ideal of, $\Mx$. Thus $\Vc_1 = P\Mx$
for some projection $P$ in $\Mx$, and if $P$ is not the identity then
it is easy to see that we could not have $\Rc_1 = \Delta$. Thus
$\Vc_1 = \Mx$ is the diagonal quantum relation.

For part (d), suppose $\Vc_3 = \Vc_1\Vc_2$ and let $(p,r) \in \Rc_3$.
Then there exist $A \in \Vc_1$ and $B \in \Vc_2$ such that $M_pABM_r \neq 0$.
For any projection $q$ in $L^\infty(X,\mu)$, we therefore have either
$M_pAM_q \neq 0$, or else
$$M_pAM_q = 0\quad\Rightarrow\quad
M_pA = M_pAM_{1-q}\quad\Rightarrow\quad
M_pABM_r = M_pAM_{1-q}BM_r.$$
Since $M_pABM_r \neq 0$, the latter implies $M_{1-q}BM_r \neq 0$. Thus we have
shown that either $(p,q) \in \Rc_1$ or $(1-q,r) \in \Rc_2$, and we conclude
that $(p,r) \in \Rc_1\Rc_2$. So $\Rc_3 \subseteq \Rc_1\Rc_2$. Conversely,
suppose $(p,r) \in \Rc_1\Rc_2$. Then there is a nonzero projection $q$
in $L^\infty(X,\mu)$ such that $(p,q') \in \Rc_1$ and $(q',r) \in \Rc_2$
for every nonzero $q' \leq q$. Now observe that the set of vectors
$w \in L^2(X,\mu)$ such that $M_pAw = 0$ for all $A \in \Vc_1$ is a closed
subspace that is invariant for $\Mx$ (since $\Vc_1$ is a right $\Mx$-module);
therefore it is the range of a projection $r$ in $\Mx' = \Mx$. Then
$M_p AM_r = 0$ for all $A \in \Vc_1$, so $(p,r) \not\in \Rc_1$ and hence
$r \leq 1-q$. This shows that for every nonzero $w$ in the range of
$M_q$ we have $M_pAw \neq 0$ for some $A \in \Vc_1$. Now $(q,r) \in \Rc_2$
implies that $M_qBM_r \neq 0$ for some $B \in \Vc_2$, i.e.,
$M_qBM_rv \neq 0$ for some vector $v$. The preceding comment then shows
that $M_pAM_qBM_r \neq 0$ for some $A \in \Vc_1$ and $B \in \Vc_2$,
so that $(p,r) \in \Rc_3$ (since $AM_qB \in
\Vc_1\Vc_2 = \Vc_3$). Thus $\Rc_3 = \Rc_1\Rc_2$.
\end{proof}

In general we do not have $\Vc = \Vc_{\Rc_\Vc}$. We will see below
(Corollary \ref{qreflex} and Proposition \ref{reflexive}) that if
$\Vc = \Ac$ is an operator algebra that contains a maximal abelian von
Neumann algebra then this happens precisely if $\Ac$ is
reflexive in the sense that $\Ac = {\rm Alg}({\rm Lat}(\Ac))$,
and Arveson (\cite{Arv}, Section 2.5; see also \cite{Fro}) has given 
an example of a weak operator closed operator algebra
that contains (and hence is a bimodule over)
a maximal abelian von Neumann algebra but is not reflexive.
In general, we will see that $\Vc = \Vc_{\Rc_\Vc}$ if and only
if $\Vc$ is reflexive in the sense of Loginov and Sul'man (\cite{Dav},
Section 15.B); see Corollary \ref{qreflex}.

In any case, by Theorem \ref{abelianrel}
$\Vc_\Rc$ is the maximal quantum relation on $\Mx$ associated
to the measurable relation $\Rc$. There is also always a minimal
quantum relation associated to $\Rc$. To describe it, recall first
(\cite{Tak}, p.\ 257) that the map $f \otimes v \mapsto f\cdot v$
implements an isometric isomorphism of the Hilbert space tensor
product $L^2(X,\mu) \otimes l^2$
with the space $L^2(X;l^2)$ of weakly
measurable functions $h: X \to l^2$, up to modification on null sets,
such that $\int \|h(x)\|^2 d\mu$ is finite. (``Weakly measurable''
means that for every $v \in H$ the scalar-valued function
$x \mapsto \langle h(x),v\rangle$ is measurable. Note that this
implies that the function
$$x \mapsto \|h(x)\|^2 = \sum_n |\langle h(x),e_n\rangle|^2$$
is measurable, where $\{e_n\}$ is the standard basis of $l^2$.)
In the following we will identify $L^2(X,\mu)\otimes l^2$ with
$L^2(X;l^2)$.

Now given a measurable relation $\Rc$ on $X$ and $h,k \in L^2(X;l^2)$, say
that $h$ is {\it $\Rc$-orthogonal} to $k$, and write $h \perp_\Rc k$, if
$$\inf\{|\langle k(y), h(x)\rangle|: x \in S, y \in T\} = 0$$
for any $S,T \subseteq X$ such that $(\chi_S, \chi_T) \in \Rc$. It should
be understood that this means the infimum must be zero
irrespective of any modification of $h$ and $k$ on null sets.
Equivalently, $h \not\perp_\Rc k$ if there exists $\epsilon > 0$ and
$S, T \subseteq X$ such that $(\chi_S,\chi_T) \in \Rc$ and
$|\langle k(y), h(x)\rangle| \geq \epsilon$ for $x \in S$ and $y \in T$.

\begin{lemma}\label{unifvector}
Let $S,T \subseteq X$, $h \in L^2(S;l^2)$, $k \in L^2(T;l^2)$, and
$\epsilon > 0$, and suppose $\mu(S),\mu(T) < \infty$ and $h \perp_\Rc k$.
Then there exist partitions $\{S_1, \ldots, S_m\}$
and $\{T_1, \ldots, T_n\}$ of $S$ and $T$ and simple functions
$$h_\epsilon = \sum_{i=1}^m \chi_{S_i}\cdot v_i
\quad{\rm and}\quad k_\epsilon
= \sum_{j=1}^n \chi_{T_j}\cdot w_j$$
($v_i, w_j \in l^2$) such that (1) $\|h - h_\epsilon\|, \|k - k_\epsilon\|
\leq \epsilon$ and (2) $|\langle w_j,v_i\rangle| \leq \epsilon$ for
any $i$ and $j$ such that $(\chi_{S_i}, \chi_{T_j}) \in \Rc$.
\end{lemma}

\begin{proof}
First find $N$ large enough that
$$h_N = \chi_{\{x: \|h(x)\| \leq N\}}\cdot h
\quad{\rm and}\quad 
k_N = \chi_{\{y: \|k(y)\| \leq N\}}\cdot k$$
satisfy $\|h - h_N\|, \|k - k_N\| \leq \epsilon/3$, and note that
$h_N$ and $k_N$ are still $\Rc$-orthogonal. Then since $l^2$ is
separable we can uniformly approximate
$h_N$ and $k_N$ with functions of the form
$$h_\epsilon' = \sum_{i=1}^\infty \chi_{S_i}\cdot v_i
\quad{\rm and}\quad k_\epsilon'
= \sum_{j=1}^\infty \chi_{T_j}\cdot w_j$$
where the $S_i$ partition $S$ and the $T_j$ partition $T$.
If the uniform approximation is sufficiently close
then (since $h_N$ and $k_N$ are bounded) $\Rc$-orthogonality of $h_N$
and $k_N$ will imply that $|\langle w_j,v_i\rangle| \leq \epsilon$
whenever $(\chi_{S_i}, \chi_{T_j}) \in \Rc$, and
we will also have $\|h_N - h_\epsilon'\|, \|k_N -k_\epsilon'\| < \epsilon/3$.
Finally, define $h_\epsilon$ and $k_\epsilon$ by truncating the
sums that define $h_\epsilon'$ and $k_\epsilon'$; that is, for suitable
$m$ and $n$ replace $S_m$ and $T_n$ with $\bigcup_{i \geq m} S_i$ and
$\bigcup_{j \geq n}T_j$ and take $v_m = w_n = 0$. This can be done at
a cost in norm of at most $\epsilon/3$, so we will have achieved conditions
(1) and (2).
\end{proof}

\begin{theo}\label{minimal}
Let $\Rc$ be a measurable relation on a finitely decomposable
measure space $(X,\mu)$ and let $\Mx \cong L^\infty(X,\mu)$ be the
von Neumann algebra of bounded multiplication operators on $L^2(X,\mu)$.
Then
$$\tilde{\Vc}_\Rc = \{A \in \Bc(L^2(X,\mu)): h \perp_\Rc k\quad
\Rightarrow\quad \langle (A \otimes I)k, h\rangle = 0\},$$
with $h$ and $k$ ranging over $L^2(X,\mu) \otimes l^2 \cong
L^2(X;l^2)$,
is a quantum relation on $\Mx$ whose associated measurable relation
(Theorem \ref{abelianrel}) is $\Rc$. If $\Vc$ is any quantum relation
on $\Mx$ whose associated measurable relation contains $\Rc$ then
$\tilde{\Vc}_\Rc \subseteq \Vc$.
\end{theo}

\begin{proof}
It is easy to see that $\tilde{\Vc}_\Rc$ is a weak* closed linear
subspace of $\Bc(L^2(X,\mu))$. To check that it is a bimodule over $\Mx$,
let $A \in \tilde{\Vc}_\Rc$ and $f,g \in L^\infty(X,\mu)$ and suppose
$h \perp_\Rc k$; then $\bar{f}\cdot h \perp_\Rc g\cdot k$ and so
$$\langle (M_fAM_g \otimes I)k,h\rangle =
\langle(A\otimes I)(g\cdot k), \bar{f}\cdot h\rangle = 0,$$
showing that $M_f AM_g \in \tilde{\Vc}_\Rc$.
So $\tilde{\Vc}_\Rc$ is a quantum relation on $\Mx$.

Next, we show that the measurable relation associated to $\tilde{\Vc}_\Rc$
is $\Rc$. Let $S,T \subseteq X$ and suppose $(\chi_S,\chi_T) \not\in \Rc$.
Then for any $f \in L^2(S,\mu|_S)$ and $g \in L^2(T,\mu|_T)$ and any unit
vector $v \in l^2$ we have
$f \cdot v \perp_\Rc g \cdot v$, so if $A \in \tilde{\Vc}_\Rc$ then
$$\langle A g, f\rangle =
\langle (A \otimes I) (g \cdot v), f \cdot v\rangle = 0.$$
It follows that $M_{\chi_S}AM_{\chi_T} = 0$ and this shows that
the measurable relation associated to $\tilde{\Vc}_\Rc$ is
contained in $\Rc$. Conversely, suppose $(\chi_S, \chi_T) \in \Rc$; we
must find $A \in \tilde{\Vc}_\Rc$ such that $M_{\chi_S}AM_{\chi_T} \neq 0$.
Find finite measure subsets $S' \subseteq S$ and $T' \subseteq T$
such that $(\chi_{S'},\chi_{T'}) \in \Rc$; it will suffice to
show that the operator $A$ with $M_{\chi_{S'}}AM_{\chi_{T'}} \neq 0$
provided by Theorem \ref{connect} belongs to $\tilde{\Vc}_\Rc$.
Fix $h,k \in L^2(X;l^2)$ such that $h \perp_\Rc k$; we want
$\langle (A \otimes I)k,h\rangle = 0$. We show this by considering the
approximating operators $A_{\Sc,\Tc}$ defined for finite partitions
of $S'$ and $T'$.

Let $h' = \chi_{S'}\cdot h$ and $k' = \chi_{T'}\cdot k$ and let
$\epsilon > 0$. Then apply Lemma \ref{unifvector} to $S'$, $T'$, $h'$, $k'$.
Now if $A_{\Sc,\Tc}$ is the operator constructed in the proof of Theorem
\ref{connect} for any finite partitions $\Sc$ and $\Tc$ of $S'$ and $T'$
which are subordinate to $\{S_i\}$ and $\{T_j\}$ then
$$|\langle (A_{\Sc,\Tc} \otimes I) k_\epsilon, h_\epsilon\rangle|
= \left|\sum \langle w_l', v_l'\rangle \mu(S_l')\right|
\leq \epsilon \mu(S').$$
(Here $\{S_1', \ldots, S_{m'}'\}$ and $\{T_1', \ldots, T_{n'}'\}$ are
the refined partitions produced in the construction of $A_{\Sc,\Tc}$
and $h_\epsilon = \sum_{i=1}^{m'} \chi_{S_i'}\cdot v_i'$ and $k_\epsilon
= \sum_{j=1}^{n'} \chi_{T_j'}\cdot w_j'$ are the corresponding expressions
for $h_\epsilon$ and $k_\epsilon$.)
Taking the limit in $\Sc$ and $\Tc$ then yields $|\langle (A\otimes I)
k_\epsilon, h_\epsilon\rangle| \leq \epsilon \mu(S')$ and taking $\epsilon
\to 0$ yields $\langle (A\otimes I)k,h\rangle = 0$, as desired.

Now let $\Vc$ be any quantum relation on $\Mx$ whose associated
measurable relation contains $\Rc$. If $\tilde{\Vc}_\Rc \subsetneq \Vc$
then as in the proof of Lemma \ref{separation} there must exist $h,k
\in L^2(X; l^2)$ such that $\langle (A \otimes I)k,h\rangle \neq 0$ for
some $A \in \tilde{\Vc}_\Rc$ --- and hence $h \not\perp_\Rc k$ --- but
$\langle (B\otimes I)k, h\rangle = 0$ for all $B \in \Vc$. Thus
suppose $h,k \in L^2(X;l^2)$ are not $\Rc$-orthogonal; we complete the
proof by showing that there exists $B \in \Vc$ such that
$\langle (B \otimes I)k,h \rangle \neq 0$. Our argument is a
straightforward adaptation of the ingenious proof of Theorem 2.1.5 in
\cite{Arv}. First, since $h \not\perp_\Rc k$ there exist $S,T \subseteq X$
and $\epsilon > 0$
such that $(\chi_S,\chi_T) \in \Rc$ and $|\langle k(y),h(x)\rangle|
\geq \epsilon$ for all $x \in S$ and $y \in T$. For some $N$ we must
have $(\chi_S, \chi_{T_N}) \in \Rc$ where $T_N = \{y \in T: \|k(y)\|
\leq N\}$, so we can assume that $k$ is bounded on $T$. By scaling $k$
(which could change the value of $\epsilon$), we may suppose
$\|k(y)\| \leq 1$ for all $y \in T$. Now find a countable
partition $\{S_i\}$ of $S$ together with a sequence
$\{v_i\} \subseteq l^2$ such that $v_i \in h(S_i)
\subseteq {\rm ball}(v_i,\epsilon/2)$ for all $i$. Then we must have
$(\chi_{S_i},\chi_T) \in \Rc$ for some $i$ and $j$. Without
loss of generality we may then replace $S$ and $h$ with
$S_i$ and $\chi_{S_i}\cdot h$. In particular,
we can assume that there is a vector $v = v_i \in l^2$ such that
$\|h(x) - v\| \leq \epsilon/2$ for all $x \in S$ and $|\langle k(y), v\rangle|
\geq \epsilon$ for all $y \in T$.

Let $P, Q \in \Mx \overline{\otimes} \Bc(l^2) \cong L^\infty(X; \Bc(l^2))$
(\cite{Tak}, Theorem IV.7.17) respectively be the orthogonal projections
onto $\overline{(\Mx\otimes I)h}$ and $\overline{(\Mx\otimes I)k}$. Suppose
for the sake of contradiction that $P(A \otimes I)Q = 0$ for all $A \in \Vc$.
Then for any $w \in L^2(X,\mu)$ we have
$$\|P(A \otimes I)(w \otimes v)\|^2
= \|P(A \otimes I)(I-Q)(w\otimes v)\|^2
\leq \|A\|^2\|(I - Q)(w\otimes v)\|^2;$$
letting $f(x) = \|P(\chi_X\cdot v)(x)\|^2$ and $g(x) =
\|(I-Q)(\chi_X\cdot v)(x)\|^2$ (both in $L^\infty(X,\mu)$),
this can be expressed as $A^*M_fA \leq \|A\|^2 M_g$, since
\begin{eqnarray*}
\langle A^*M_fAw,w\rangle &=& \int \|P(\chi_X \cdot v)(x)\|^2
\|(Aw)(x)\|^2\, d\mu\cr
&=& \int \|P(Aw \otimes v)(x)\|^2\, d\mu\cr
&=& \|P(A\otimes I)(w\otimes v)\|^2
\end{eqnarray*}
and similarly $\langle M_g w,w\rangle = \|(I - Q)(w\otimes v)\|^2$.
This inequality holds for all $A \in \Vc$.

Now given $A \in \Vc$ and $\delta > 0$, let $B =
M_{\sqrt{f}}AM_{1/\sqrt{g+\delta}} \in \Vc$. We then have
$$B^*B = M_{1/\sqrt{g+\delta}}A^*M_fAM_{1/\sqrt{g+\delta}}
\leq \|A\|^2 M_{1/\sqrt{g+\delta}}M_gM_{1/\sqrt{g+\delta}}
= \|A\|^2 M_{g/(g+\delta)}$$
so that $\|B\| \leq \|A\|$, and hence
$$(M_{1/\sqrt{g+\delta}}A^*M_{\sqrt{f}})M_f(M_{\sqrt{f}}AM_{1/\sqrt{g+\delta}})
= B^*M_fB \leq \|B\|^2M_g \leq \|A\|^2M_g.$$
Multiplying on both sides by $M_{\sqrt{g + \delta}}$ then yields
$A^*M_{f^2}A \leq \|A\|^2M_{g(g+\delta)}$, and taking $\delta \to 0$,
we get $A^*M_{f^2}A \leq \|A\|^2M_{g^2}$. Applying this argument inductively
establishes that
$A^*M_{f^n}A \leq \|A\|^2M_{g^n}$ for all $A \in \Vc$ and all $n \in \Nb$.
Now $h \in {\rm ran}(P)$ implies that $f(x) \geq \|v\|^2 - (\epsilon/2)^2$
for all $x \in S$, while $k \in {\rm ran}(Q)$ implies that $g(y) \leq
\|v\|^2 - \epsilon^2$ for all $y \in T$. Thus
\begin{eqnarray*}
(\|v\|^2 - (\epsilon/2)^2)^nA^*M_{\chi_S}A &\leq& A^*M_{f^n}A \leq
\|A\|^2M_{g^n}\cr
&\leq& \|A\|^2[(\|v\|^2 - \epsilon^2)^nM_{\chi_T} +
\|v\|^2M_{\chi_{X - T}}]
\end{eqnarray*}
and so
$$M_{\chi_T}A^*M_{\chi_S}AM_{\chi_T} \leq
\|A\|^2\left(\frac{\|v\|^2 - \epsilon^2}{\|v\|^2 - (\epsilon/2)^2}\right)^n
M_{\chi_T}.$$
Taking $n \to \infty$ then yields $M_{\chi_T}A^*M_{\chi_S}AM_{\chi_T} = 0$,
and hence $M_{\chi_S}AM_{\chi_T} = 0$. Since this is true for all $A \in \Vc$
we cannot have $(\chi_S, \chi_T) \in \Rc$, a contradiction. We conclude
that $P(A \otimes I)Q \neq 0$ for some $A \in \Vc$, and hence that
$\langle (B \otimes I)k,h \rangle \neq 0$ for some $B \in \Vc$. This
completes the proof.
\end{proof}

\subsection{Operator reflexivity}\label{apps}
Specializing the preceding work to the case of measurable partial orders,
we recover Arveson's basic results on commutative subspace lattices.
In particular, the formula $\Rc = \Rc_{\Vc_\Rc}$ in Theorem
\ref{abelianrel} emerges as an attractive generalization of Arveson's
reflexivity theorem (from which it can, alternatively, be deduced; see
the comment at the beginning of Section \ref{abcs}).

\begin{theo}\label{CSLchar}
(\cite{Arv}, Theorem 1.3.1)
Let $\Lc$ be a complete 0,1-lattice of commuting projections in some $\Bc(H)$.
Then there is a measurable preorder $\Rc$ on a finitely decomposable
measure space $(X,\mu)$ and an isomorphism $H \cong L^2(X,\mu)$ that
takes $\Lc$ to
$$\{M_{\chi_S}: S\hbox{ is a lower set for }\Rc\}.$$
\end{theo}

\begin{proof}
$\Lc$ generates an abelian von Neumann algebra and hence is contained
in a maximal abelian von Neumann algebra $\Mx$. Then there exists a
finitely decomposable measure space $(X,\mu)$ and an isomorphism
$H \cong L^2(X,\mu)$ that takes $\Mx$ to the algebra of bounded
multiplication operators. That is, $\Mx \cong L^\infty(X,\mu)$.
This isomorphism takes $\Lc$ to a complete 0,1-sublattice of the
lattice of projections in $L^\infty(X,\mu)$ and the result now
follows from Theorem \ref{versuslattice}.
\end{proof}

(Theorem 1.3.1 of \cite{Arv} is expressed in terms of pointwise preorders;
this version, when $\mu$ is $\sigma$-finite, follows from Theorem
\ref{pointwise} and the comment preceding that result.)

For any set of projections $\Lc \subset \Bc(H)$ let ${\rm Alg}(\Lc)$ be
the algebra of operators for which the range of every projection in
$\Lc$ is invariant; that is, ${\rm Alg}(\Lc) = \{A \in \Bc(H): PAP = AP$
for all $P \in \Lc\}$. For any set of operators $\Ac \subseteq \Bc(H)$
let ${\rm Lat}(\Ac)$ be the lattice of projections whose range is
invariant for every operator in $\Ac$.

\begin{theo}
(\cite{Arv}, Theorem 1.6.1)
Let $\Lc$ be a complete 0,1-lattice of commuting projections in some
$\Bc(H)$. Then $\Lc = {\rm Lat}({\rm Alg}(\Lc))$.
\end{theo}

\begin{proof}
By Theorem \ref{CSLchar} we may assume that $H = L^2(X,\mu)$ and
there is a measurable preorder $\Rc$ on $X$ such that $\Lc$ consists
of the operators $M_p$ for $p$ a projection in $L^\infty(X,\mu)$
satisfying $(1-p,p) \not\in \Rc$. Define $\Vc_\Rc$
as in Theorem \ref{abelianrel}; we claim that
$\Vc_\Rc = {\rm Alg}(\Lc)$. To see this first let $A \in \Vc_\Rc$ and $M_p
\in \Lc$. Then $(1-p,p) \not\in \Rc$, so $M_{1-p}AM_p = 0$, which
shows that $A \in {\rm Alg}(\Lc)$. Conversely, let $A \in {\rm Alg}(\Lc)$. If
$p,q \in L^\infty(X,\mu)$ satisfy $(p,q) \not\in \Rc$ then there
exists $M_{q'} \in \Lc$ such that $q \leq q'$ and $pq' = 0$
(Theorem \ref{versuslattice}), and
$M_{q'}AM_{q'} = AM_{q'}$ then implies that
$$M_pAM_q = M_pAM_{q'}M_q = M_pM_{q'}AM_{q'}M_q = 0,$$
which shows that $A \in \Vc_\Rc$. This proves the claim.

Observe that ${\rm Alg}(\Lc)$
contains all bounded multiplication operators, so any
projection in $\Bc(H)$ whose range is invariant for ${\rm Alg}(\Lc)$
must commute with all multiplication operators and hence must
have the form $M_p$ with $p \in L^\infty(X,\mu)$. Now
by Theorem \ref{abelianrel} we have $\Rc = \{(p,q): M_pAM_q
\neq 0$ for some $A \in \Vc_\Rc\}$. So
\begin{eqnarray*}
M_p \in \Lc &\Leftrightarrow& (1-p,p) \not\in \Rc\cr
&\Leftrightarrow& M_{1-p}AM_p = 0\hbox{ for all }A \in \Vc_\Rc\cr
&\Leftrightarrow& M_pAM_p = AM_p\hbox{ for all }A \in {\rm Alg}(\Lc).
\end{eqnarray*}
This shows that a projection of the form $M_p$ belongs to $\Lc$
if and only if the range of $M_p$ is invariant for ${\rm Alg}(\Lc)$.
We saw just above that every projection in ${\rm Lat}(\rm Alg(\Lc))$ must
take this form, so we conclude that $\Lc = {\rm Lat}({\rm Alg}(\Lc))$.
\end{proof}

Next we relate our approach to Loginov and Sul'man's generalized
notion of reflexivity (\cite{Dav}, Section 15.B).
We know from Lemma \ref{separation} that a quantum
relation $\Vc$ is determined by the pairs of projections $P$ and $Q$ in
$\Bc(H\otimes l^2)$ that annihilate it (i.e., such that $P(A\otimes I)Q
= 0$ for all $A \in \Vc$). We also noted in the comment following
Theorem \ref{abelianrel} that $\Vc$ in general is not determined by the
pairs of projections in $\Bc(H)$ that annihilate it. This suggests
the following definition:

\begin{defi}\label{refdef}
A subspace $\Vc \subseteq \Bc(H)$ is {\it operator reflexive} if
$$\Vc = \{B \in \Bc(H): P\Vc Q = 0\quad \Rightarrow\quad PBQ = 0\},$$
with $P$ and $Q$ ranging over projections in $\Bc(H)$.
\end{defi}

We use the term ``operator reflexive'' to avoid confusion with the notion
of reflexivity of a quantum relation (Definition \ref{quanttypes} (d)).

Definition \ref{refdef} makes sense for any subspace $\Vc$, but in the case of
quantum relations it can be slightly modified:

\begin{prop}\label{relatref}
Let $\Vc$ be a quantum relation over a von Neumann algebra $\Mx \subseteq
\Bc(H)$. Then $\Vc$ is operator reflexive if and only if
$$\Vc = \{B \in \Bc(H): P\Vc Q = 0\quad\Rightarrow\quad PBQ = 0\},$$
with $P$ and $Q$ ranging over projections in $\Mx$.
\end{prop}

\begin{proof}
Let $P$ and $Q$ be projections in $\Bc(H)$ and suppose $P\Vc Q = 0$.
Since $\Vc$ is a bimodule over $\Mx'$, we have $P\Mx'\Vc\Mx' Q = 0$,
and hence $\tilde{P}\Vc\tilde{Q} = 0$ where $\tilde{P}$ is the
orthogonal projection onto the closure of $\Mx'({\rm ran}(P))$ and
$\tilde{Q}$ is the orthogonal projection onto the closure of
$\Mx'({\rm ran}(Q))$. Also, $\tilde{P}$ and $\tilde{Q}$ belong to $\Mx$
because their ranges are invariant for $\Mx'$. So for any projections
$P$ and $Q$ such that $P\Vc Q = 0$ there are larger projections
$\tilde{P}$ and $\tilde{Q}$ in $\Mx$ such that $\tilde{P}\Vc\tilde{Q} = 0$.
This entails that the two conditions are equivalent.
\end{proof}

Operator reflexivity is of particular interest for quantum relations over
maximal abelian von Neumann algebras because of the following result.

\begin{coro}\label{qreflex}
Let $(X,\mu)$ be a finitely decomposable measure space, let
$\Mx \cong L^\infty(X,\mu)$ be the von Neumann algebra of
bounded multiplication operators on $L^2(X,\mu)$, and let $\Vc$ be
a quantum relation on $\Mx$. Then in the notation of Theorem
\ref{abelianrel}, $\Vc = \Vc_{\Rc_\Vc}$ if and only if $\Vc$ is
operator reflexive.
\end{coro}

The proof of this corollary is trivial, as $\Vc_{\Rc_\Vc}$
by definition consists
of precisely those operators $B$ which satisfy $P\Vc Q = 0$ $\Rightarrow$
$PBQ = 0$, with $P$ and $Q$ ranging over projections in $\Mx$.

Loginov and Sul'man's version of operator reflexivity is stated in part (iii)
of the following result. Our definition is also equivalent to one formulated
by Erdos \cite{Erd}. Given a subspace $\Vc \subseteq \Bc(H)$, for any
projection $Q \in \Bc(H)$ let $\phi(Q)$ be the orthogonal projection
onto the closure of $\Vc({\rm ran}(Q))$. That is, $\phi(Q) =
I - \bigvee\{P: P\Vc Q = 0\}$. Erdos's definition is stated in part
(ii) of the next result. Part (v) is Larson's characterization
of operator reflexivity (\cite{Lar}, Lemma 2).

\begin{prop}\label{equivalences}
Let $\Vc$ be a subspace of $\Bc(H)$. The following are equivalent:
\begin{quote}
\noindent (i) $\Vc$ is operator reflexive

\noindent (ii) $\Vc = \{B \in \Bc(H): BQ = \phi(Q)BQ$ for all projections
$Q \in \Bc(H)\}$

\noindent (iii) $\Vc = \{B \in \Bc(H): Bv \in \overline{\Vc v}$ for
all $v \in H\}$

\noindent (iv) for any $A \in \Bc(H) - \Vc$ there exist $v,w \in H$ such that
$$\langle Aw,v\rangle \neq 0$$
but
$$\langle Bw,v\rangle = 0$$
for all $B \in \Vc$

\noindent (v) $\Vc$ is weak* closed and its preannihilator $\Vc_\perp \subseteq \TC(H)$
is generated by rank one operators.
\end{quote}
\end{prop}

\begin{proof}
(i) $\Leftrightarrow$ (ii): This follows from the fact that $P\Vc Q = 0$
$\Leftrightarrow$ $P \leq I - \phi(Q)$.

(ii) $\Leftrightarrow$ (iii): Trivial.

(i) $\Leftrightarrow$ (iv): $P\Vc Q = 0$ $\Rightarrow$ $PBQ = 0$ holds
for all projections $P$ and $Q$ if and only if it holds for all rank one
projections $P$ and $Q$, and $\langle Aw,v\rangle = 0$ $\Leftrightarrow$
$PAQ = 0$ where $P$ and $Q$ are respectively the orthogonal projections
onto $\Cb v$ and $\Cb w$.

(iv) $\Leftrightarrow$ (v): The linear functionals $A \mapsto {\rm tr}(AB)$
on $\Bc(H)$ with $B$ a rank one operator are precisely the linear functionals
$A \mapsto \langle Aw,v\rangle$ with $v,w \in H$.
\end{proof}

Every subspace $\Vc$ of $\Bc(H)$ has a {\it reflexive closure}
$$\overline{\Vc} =
\{B \in \Bc(H): P\Vc Q = 0\quad \Rightarrow\quad PBQ = 0\},$$
with $P$ and $Q$ ranging over projections in $\Bc(H)$. It is easy to
see that $\overline{\Vc}$ is the smallest operator reflexive subspace that
contains $\Vc$. Moreover, if $\Vc$ is a quantum relation then so is
$\overline{\Vc}$:

\begin{prop}
Let $\Vc$ be a quantum relation on a von Neumann algebra $\Mx
\subseteq \Bc(H)$. Then its reflexive closure $\overline{\Vc}$ is
also a quantum relation on $\Mx$, and we have
$$\overline{\Vc} = \{B \in \Bc(H): P\Vc Q = 0\quad\Rightarrow\quad
PBQ = 0\}$$
with $P$ and $Q$ ranging over projections in $\Mx$.
\end{prop}

\begin{proof}
The first statement follows from the second because if $B \in \overline{\Vc}$,
$A,C \in \Mx'$, and $PBQ = 0$ for all projections $P,Q \in \Mx$ with
$P\Vc Q = 0$ then
$$P(ABC)Q = A(PBQ)C = 0$$
for all such projections $P$ and $Q$. This shows that $ABC \in \overline{\Vc}$.

The second assertion of the proposition follows from the observation
made in the proof of Proposition \ref{relatref} that if $P\Vc Q = 0$
for some projections $P,Q \in \Bc(H)$ then there exist projections
$\tilde{P},\tilde{Q} \in \Mx$ with $P \leq \tilde{P}$, $Q \leq \tilde{Q}$,
and $\tilde{P}\Vc\tilde{Q} = 0$.
\end{proof}

Next we note that if $\Vc$ is an operator algebra then our definition
of operator reflexivity is equivalent to the standard one. This follows from
Proposition \ref{equivalences} (i) $\Leftrightarrow$ (iii) above.

\begin{prop}\label{reflexive}
Let $\Ac \subseteq \Bc(H)$ be a unital operator algebra. Then $\Ac$ is
operator reflexive if and only if $\Ac = {\rm Alg}({\rm Lat}(\Ac))$.
\end{prop}

The next result is given in Lemma 15.4 of \cite{Dav}, but for the sake
of completeness we include a short proof here.

\begin{prop}\label{tensorref} (\cite{Dav}, Lemma 15.4)
Let $\Vc$ be a weak* closed subspace of $\Bc(H)$. Then
$\Vc \otimes I$ is an operator reflexive subspace of $\Bc(H \otimes l^2)$.
\end{prop}

\begin{proof}
Let $A \in \Bc(H\otimes l^2) - \Vc \otimes I$. By Proposition
\ref{equivalences} (iv) it will suffice to find $v,w \in H \otimes l^2$
such that $\langle Aw,v\rangle \neq 0$ but $\langle (B\otimes I)w,v\rangle
= 0$ for all $B \in \Vc$. There are two cases. First, suppose
$A \not\in \Bc(H)\otimes I$. Since $\Bc(H) \otimes I$ is a von
Neumann algebra it is operator reflexive (Proposition \ref{reflexive}),
so by Proposition \ref{equivalences} there exist $v,w \in H \otimes l^2$
such that $\langle Aw,v\rangle \neq 0$ but $\langle (B \otimes I)w,v\rangle
=0$ for all $B \in \Bc(H)$, in particular for all $B \in \Vc$, as desired.
The other case is that $A \in \Bc(H) \otimes I$, say $A = A_0 \otimes I$.
Then $A_0 \not\in \Vc$ and, as in Lemma \ref{separation}, the desired pair
of vectors $v,w \in H \otimes l^2$ exist since $\Vc$ is weak* closed.
This completes the proof.
\end{proof}

We also recover Erdos's generalization of Arveson's theorem on
the operator reflexivity of commutative subspace lattices. Given a map
$\phi$ from projections in $\Bc(H)$ to projections in $\Bc(H)$,
let its {\it co-map} be the map
$$\psi: P \mapsto \bigvee\{Q: \phi(Q) \leq P\}.$$
Let $[A]$ denote the range projection of the operator $A \in \Bc(H)$.

\begin{theo}\label{Erdos}
(\cite{Erd}, Theorem 4.4)
Let $\phi$ be a join preserving map from the set of projections in
$\Bc(H)$ to itself such that $\phi(0) = 0$, and suppose that all of
the projections in the ranges of $\phi$ and its co-map $\psi$ commute. Then
$$\phi(R) = \bigvee \{[AR]: A \in \Bc(H)\hbox{ and }
\phi(Q)AQ = AQ\hbox{ for all projections } Q \in \Bc(H)\}$$
for every projection $R \in \Bc(H)$.
\end{theo}

\begin{proof}
Let $\Mx$ be a maximal abelian von Neumann algebra containing
all of the projections in the ranges of $\phi$
and $\psi$, let $\Phi: L^\infty(X,\mu) \cong \Mx$ be an isomorphism,
and define a measurable relation $\Rc$ on $X$ by
setting $(p,q) \in \Rc$ if $\Phi(p)\phi(\Phi(q)) \neq 0$ (cf.\ Proposition
\ref{erd}). Let
\begin{eqnarray*}
\Vc &=& \{A \in \Bc(H): \phi(Q)AQ = AQ\hbox{ for all projections }
Q \in \Bc(H)\}\cr
&=& \{A \in \Bc(H): \phi(Q)A\psi(\phi(Q)) = A\psi(\phi(Q))
\hbox{ for all projections }
Q \in \Bc(H)\}\cr
&=& \{A \in \Bc(H): (p,q) \not\in \Rc\quad\Rightarrow\quad
\Phi(p)A\Phi(q) = 0\}.
\end{eqnarray*}
That is, $\Vc = \Vc_\Rc$ as in Theorem \ref{abelianrel}.
Now fix a projection $R \in \Bc(H)$, let $P = \bigvee\{[AR]: A \in \Vc\}
\leq \phi(R)$,
and suppose $P < \phi(R)$. Then $\phi(R) \in \Mx$ by construction,
and ${\rm ran}(P)$ is invariant for $\Mx$ so $P \in \Mx$ since $\Mx$
is maximal abelian, so say $P = \Phi(p)$ and $\phi(R) = \Phi(r)$. Also
let $Q = \bigvee\{[AR]: A \in \Mx\} \in \Mx$ and say $Q = \Phi(q)$. Now
$(r-p,q) \in \Rc$ since $\phi(R) \leq \phi(Q)$, so Theorem \ref{abelianrel}
implies that there exists $A \in \Vc$ such that $\Phi(r-p)A\Phi(q) \neq 0$,
contradicting the definition of $P$. We conclude that $P = \phi(R)$,
as desired.
\end{proof}

Theorem 4.4 of \cite{Erd} is apparently more general than this since
it covers maps from projections in $\Bc(H)$ to projections in $\Bc(K)$,
but this version of the result follows easily from Theorem \ref{Erdos} by
working in $\Bc(H \oplus K)$.

Finally, we have the following partially new result. It characterizes
various classes of operator reflexive quantum relations over a maximal abelian
von Neumann algebra.

\begin{theo}\label{char}
Let $\Mx$ be a maximal abelian von Neumann algebra
in $\Bc(H)$, let $\Phi: L^\infty(X,\mu) \cong \Mx$
be an isomorphism, and let $\Vc \subseteq \Bc(H)$ be an operator reflexive
operator space satisfying $\Mx\Vc\Mx \subseteq \Vc$. Then there is a measurable
relation $\Rc$ on $X$ such that
$$\Vc = \{A \in \Bc(H): (p,q) \not\in \Rc\quad\Rightarrow\quad
\Phi(p)A\Phi(q) = 0\}.$$
If $\Vc$ is a von Neumann algebra then $\Rc$ is a measurable equivalence
relation.
If $\Vc$ is an operator system then $\Rc$ is a measurable graph. If
$\Vc$ is an operator algebra then $\Rc$ is a measurable preorder. If
$\Vc$ is a triangular operator algebra then $\Rc$ is a measurable
partial order.
\end{theo}

\begin{proof}
This follows from Proposition \ref{qreflex} together with the last
part of Theorem \ref{abelianrel}, which implies that reflexivity,
symmetry, antisymmetry, and transitivity of $\Vc$ all carry over
to $\Rc_\Vc$.
\end{proof}

The converse assertions, that for any measurable relation (equivalence
relation, graph, preorder, partial order) $\Rc$ on $X$ the set $\Vc =
\{A \in \Bc(H): (p,q) \not\in \Rc$ $\Rightarrow$ $\Phi(p)A\Phi(q) =0\}$
is an operator reflexive operator space (von Neumann algebra, operator system,
operator algebra, triangular operator algebra) satisfying
$\Mx\Vc\Mx \subseteq \Vc$, are trivial. So this gives us a complete
characterization of these classes of operator reflexive operator bimodules
over maximal abelian von Neumann algebras.

Theorem \ref{char} reduces operator reflexive operator bimodules to various classes
of measurable relations, but recall that we could reduce further to
pointwise relations by Theorem \ref{pointwise}.

\subsection{Intrinsic characterization}
Since quantum relations are effectively representation independent
(Theorem \ref{repindep}), there should be an intrinsic
characterization of them. We provide such a characterization
in this section by axiomatizing the family of annihilating pairs of
projections in $\Mx \overline{\otimes} \Bc(l^2)$ introduced in Lemma \ref{separation}.

First we note that in finite dimensions, quantum relations on $\Mx$
naturally correspond to projections in $\Mx \otimes \Mx^{op}$. (We already
know this when $\Mx$ is atomic and abelian by Proposition \ref{atomiccase},
and we know it is false in the general abelian case by Theorem
\ref{abelianrel} and the discussion at the beginning of Section \ref{reduct}.)

\begin{prop}\label{fdchar}
Let $\Mx \subseteq \Bc(H)$ be a finite dimensional von Neumann algebra.
Define an action $\Phi$ of $\Mx \otimes \Mx^{op}$ on $\Bc(H)$ by setting
$$\Phi_{A\otimes C}(B) = ABC$$
for $A \in \Mx$, $C \in \Mx^{op}$, and $B \in \Bc(H)$ and extending linearly.
Then for any quantum relation $\Vc$ on $\Mx$ the set
$$\Ic_\Vc = \{X \in \Mx \otimes \Mx^{op}: \Phi_X(B) = 0\hbox{ for all }
B \in \Vc\}$$
is a left ideal of $\Mx \otimes \Mx^{op}$, and for any left ideal
$\Ic$ of $\Mx \otimes \Mx^{op}$ the set
$$\Vc_\Ic = \{B \in \Bc(H): \Phi_X(B) = 0\hbox{ for all }X \in \Ic\}$$
is a quantum relation on $\Mx$. The two constructions are inverse to
each other. The lattice of quantum relations on $\Mx$ is order isomorphic
to the lattice of projections in $\Mx \otimes \Mx^{op}$.
\end{prop}

\begin{proof}
It is straightforward to check that $\Ic_\Vc$ is a left ideal and
$\Vc_\Ic$ is a quantum relation. We verify that the two constructions
are inverse to each other. By Theorem \ref{repindep} we can choose the
representation of $\Mx$, so take $H = \bigoplus \Cb^{n_i}$,
$\Mx \cong \bigoplus_i M_{n_i}(\Cb)$, and $\Mx\otimes\Mx^{op}
\cong \bigoplus_{i,j} M_{n_i}(\Cb)\otimes M_{n_j}(\Cb)^{op}$.
Then the left ideals of
$\Mx \otimes\Mx^{op}$ are all of the form $\bigoplus_{i,j} \Ic_{i,j}$
where $\Ic_{i,j}$ is a left ideal of $M_{n_i}(\Cb)\otimes M_{n_j}(\Cb)^{op}$.
The commutant of $\Mx$ consists of the diagonal matrices that are
constant on each $\Cb^{n_i}$, and consequently the bimodules
over $\Mx'$ are all of the form $\bigoplus_{i,j} \Vc_{i,j}$
where $\Vc_{i,j}$ is a subspace of $M_{n_i,n_j}(\Cb)$. We now
work in the summand corresponding to a single pair $(i,j)$.

The natural vector space isomorphism $M_{n_i,n_j}(\Cb) \cong \Cb^{n_in_j}$
converts the action of $M_{n_i}(\Cb)\otimes M_{n_j}(\Cb)^{op}$
to the standard action of $M_{n_in_j}(\Cb)$ in a way that is compatible
with the natural isomorphism of $M_{n_i}(\Cb)\otimes M_{n_j}(\Cb)^{op}$
with $M_{n_in_j}(\Cb)$, as can be seen by checking matrix units.
So we reduce to showing that the map taking a left ideal of
$M_k(\Cb)$ to the subspace of $\Cb^k$ it annihilates is inverse
to the map taking a subspace of $\Cb^k$ to the left ideal of
$M_k(\Cb)$ that annihilates it. This follows from the fact
that the left ideals of $M_k(\Cb)$ are all of the form
$M_k(\Cb)P$ for $P$ a projection in $M_k(\Cb)$ (\cite{Tak},
Theorem I.7.4).

This correspondence between quantum relations and left ideals is
order reversing, but the map $(\Mx \otimes \Mx^{op})P \mapsto
I - P$ is an order inverting 1-1 correspondence beween the left
ideals and the projections, so the lattice of quantum relations
is naturally order isomorphic to the lattice of projections.
\end{proof}

The main result of this section gives an intrinsic characterization of
quantum relations over any von Neumann algebra.
Recall that $[A]$ denotes the range projection of the operator $A$.

\begin{defi}\label{abstractqrel}
Let $\Mx$ be a von Neumann algebra and let $\Pc$ be the set of projections
in $\Mx \overline{\otimes} \Bc(l^2)$, equipped with the restriction of the weak
operator topology. An {\it intrinsic quantum relation} on $\Mx$ is an
open subset $\Rc \subset \Pc \times \Pc$ satisfying
\begin{quote}
\noindent (i) $(0,0)\not\in \Rc$

\noindent (ii) $(\bigvee P_\lambda, \bigvee Q_\kappa) \in \Rc$
$\Leftrightarrow$ some $(P_\lambda, Q_\kappa) \in \Rc$

\noindent (iii) $(P,[BQ]) \in \Rc$ $\Leftrightarrow$ $([B^*P],Q) \in \Rc$
\end{quote}
for all projections $P,Q,P_\lambda,Q_\kappa \in \Pc$
and all $B \in I \otimes \Bc(l^2)$.
\end{defi}

This abstract version of quantum relations is helpful because some
constructions become more natural when framed in these terms. Most
significantly, this is true of the pullback construction described
in part (b) of the following proposition. (On the other hand, some
constructions are more natural in the concrete setting, for instance
the product of quantum relations (Definition \ref{quanttypes} (c)).)

\begin{prop}\label{qpb}
Let $\Mx$ and $\Nc$ be von Neumann algebras.

\noindent (a) Any union of intrinsic quantum relations on $\Mx$ is an
intrinsic quantum relation on $\Mx$.

\noindent (b) If $\phi: \Mx \to \Nc$ is a unital weak* continuous
$*$-homomorphism and $\Rc$ is an intrinsic quantum relation on $\Nc$
then
$$\phi^*(\Rc) = \{(P,Q):
((\phi \otimes {\rm id})(P), (\phi \otimes {\rm id})(Q)) \in \Rc\}$$
(with $P$ and $Q$ ranging over projections in $\Mx \overline{\otimes}
\Bc(l^2)$) is an intrinsic quantum relation on $\Mx$.
\end{prop}

The proof is straightforward. In part (b) we verify condition (iii)
of Definition \ref{abstractqrel} using the identity
$$(\phi \otimes {\rm id})([BQ]) =[(\phi \otimes {\rm id})(BQ)]
= [B(\phi \otimes {\rm id})(Q)].$$

Pullbacks are not compatible with products. We already noted this in
the atomic abelian case; see the comment following Definition
\ref{meastypes}.

As we mentioned in Section \ref{abcs}, the linking algebra construction
allows us to embed any quantum relation in a quantum partial order,
i.e., a weak* closed unital operator algebra, and for many purposes
the two points of view do not substantively differ. However, pullbacks
are clearly more natural in the quantum relation setting, as the
pullback of a quantum partial order need not be a quantum partial order.

Before proceeding to the equivalence of Definitions \ref{quantrel}
and \ref{abstractqrel} we give a nontrivial example. Let $(X,\mu)$
be a finitely decomposable measure space, let $\Mx
\cong L^\infty(X,\mu)$ be the von Neumann algebra of bounded
multiplication operators on $L^2(X,\mu)$, and let $\Rc$ be a measurable
relation on $X$. Recall the notion of $\Rc$-orthogonality for
vectors in $L^2(X,\mu) \otimes l^2 \cong L^2(X;l^2)$ introduced in
Section \ref{abcs}. Let $\tilde{\Rc}$ be the set of pairs of
projections $P, Q \in \Mx \overline{\otimes} \Bc(l^2)$ such
that $h \not\perp_\Rc k$ for some $h \in {\rm ran}(P)$ and $k \in
{\rm ran}(Q)$. We will now show that $\tilde{\Rc}$ is an intrinsic
quantum relation. In fact, the quantum relation $\tilde{\Vc}_\Rc$ defined
in Theorem \ref{minimal} is the quantum relation associated to $\tilde{\Rc}$
according to the correspondence to be established in Theorem
\ref{abstractchar} below.

\begin{lemma}\label{orthogapprox}
Let $\Rc$ be a measurable relation on a finitely decomposable
measure space $(X,\mu)$ and let $h,k, h_n, k_n \in L^2(X;l^2)$.
Suppose that $h_n \to h$ and $k_n \to k$ in norm and that
$h_n \perp_\Rc k_n$ for all $n$. Then $h \perp_\Rc k$.
\end{lemma}

\begin{proof}
By passing to a subsequence, we can assume that $\|h - h_n\|,
\|k - k_n\| \leq 1/n^2$ for all $n$. Now let $S,T \subseteq X$ and suppose
$\epsilon = \inf\{|\langle k(y), h(x)\rangle|: x \in S, y \in T\} > 0$;
we must show that $(\chi_S, \chi_T) \not\in \Rc$. For each $n$
let
$$S_n = \{x \in S: \|h(x) - h_n(x)\| \leq 1/n\hbox{ and }
\|h_n(x)\| \leq n\epsilon/3\}$$
and let
$$T_n = \{y \in T: \|k(x) - k_n(x)\| \leq 1/n\hbox{ and }
\|k(y)\| \leq n\epsilon/3\}.$$
Then
\begin{eqnarray*}
|\langle k_n(y), h_n(x)\rangle|
&\geq& |\langle k(y), h(x)\rangle| - |\langle k(y), h(x) - h_n(x)\rangle|
- |\langle k(y) - k_n(y), h_n(x)\rangle|\cr
&\geq& \epsilon - \epsilon/3 - \epsilon/3 = \epsilon/3
\end{eqnarray*}
for all $x \in S_n$ and $y \in T_n$, which implies that $(\chi_{S_n},
\chi_{T_n}) \not\in \Rc$ since $h_n$ is $\Rc$-orthogonal to $k_n$.
Since $\|h - h_n\|, \|k - k_n\| \leq 1/n^2$, a simple computation shows
that $S = \liminf S_n$ and $T = \liminf T_n$, and this yields
$(\chi_S, \chi_T) \not\in \Rc$, as desired. (See the proof of Lemma
\ref{partit} for a detailed explication of this final step.)
\end{proof}

\begin{lemma}\label{sums}
Let $\Rc$ be a measurable relation on a finitely decomposable
measure space $(X,\mu)$ and let $h,h',k \in L^2(X;l^2)$.
Suppose that $h \perp_\Rc k$ and $h' \perp_\Rc k$.
Then $(h + h') \perp_\Rc k$.
\end{lemma}

\begin{proof}
Let $\epsilon > 0$ and suppose $(\chi_S, \chi_T) \in \Rc$. Then
for some $N \in \Nb$ we must have $(\chi_{S_N},\chi_{T_N}) \in \Rc$
where $S_N = \{x \in S: \|h(x)\|, \|h'(x)\| \leq N\}$ and
$T_N = \{y \in T: \|k(y)\| \leq N\}$. Now partition $S_N$ and $T_N$ into
sets $S_N^i$ and $T_N^j$ such that $\|h(x) - h(x')\|, \|h'(x) - h'(x')\|
\leq \epsilon/4N$ for all $x,x' \in S_N^i$ and $\|k(y) - k(y')\| \leq
\epsilon/4N$ for all $y,y' \in T_N^j$. Then for some $i$ and $j$ we
must have $(\chi_{S_N^i}, \chi_{T_N^j}) \in \Rc$, but this implies
$$\inf\{|\langle k(y), h(x)\rangle|: x \in S_N^i, y \in T_N^j\}
= \inf\{|\langle k(y), h'(x)\rangle|: x \in S_N^i, y \in T_N^j\} = 0.$$
It follows that $|\langle k(y), h(x)\rangle|, |\langle k(y), h'(x)\rangle|
\leq \epsilon/2$ for all $x \in S_N^i$ and $y \in T_N^j$, and hence
$|\langle k(y), (h + h')(x)\rangle| \leq \epsilon$ for all $x \in S_N^i$
and $y \in T_N^j$. We conclude that
$$\inf\{|\langle k(y), (h + h')(x)\rangle|: x \in S, y \in T\} = 0.$$
This shows that $h + h'$ is $\Rc$-orthogonal to $k$.
\end{proof}

Note that Lemma \ref{sums} also holds for sums in the second variable,
since $h \perp_\Rc k$ if and only if $k \perp_{\Rc^T} h$, where
$\Rc^T$ is the transpose of $\Rc$ (Definition \ref{meastypes} (b)).

\begin{theo}
Let $(X,\mu)$ be a finitely decomposable measure space, let $\Mx
\cong L^\infty(X,\mu)$ be the von Neumann algebra of bounded
multiplication operators on $L^2(X,\mu)$, and let $\Rc$ be a measurable
relation on $X$. Then the set $\tilde{\Rc}$ of pairs of
projections $P, Q \in \Mx \overline{\otimes} \Bc(l^2)$ such
that $h \not\perp_\Rc k$ for some $h \in {\rm ran}(P)$ and $k \in
{\rm ran}(Q)$ is an intrinsic quantum relation on $\Mx$.
\end{theo}

\begin{proof}
First we check that the complement of $\tilde{\Rc}$ is closed. To
see this suppose $P_\lambda \to P$ and $Q_\lambda \to Q$ and
$(P_\lambda, Q_\lambda) \not\in \tilde{\Rc}$ for all $\lambda$.
Let $h \in {\rm ran}(P)$ and $k \in {\rm ran}(Q)$.
Since the weak and strong operator topologies agree on $\Pc$, for
any $n \in \Nb$ there exists $\lambda$ and $h_n \in {\rm ran}(P_\lambda)$,
$k_n \in {\rm ran}(Q_\lambda)$ such that $\|h_n - h\|, \|k_n - k\|
\leq 1/n$. Lemma \ref{orthogapprox} therefore yields $h \perp_\Rc k$,
and we conclude that $(P,Q) \not\in \tilde{\Rc}$.

Next, it is clear that $(0,0) \not\in \tilde{\Rc}$ and that
$P' \leq P$, $Q' \leq Q$, $(P', Q') \in \tilde{\Rc}$ implies
$(P,Q) \in \tilde{\Rc}$. Now suppose $(P_\lambda, Q_\kappa) \not\in \tilde{\Rc}$
for all $\lambda$, $\kappa$. By a double application of Lemma \ref{sums}
we have $h \perp_\Rc k$ for any $h$ in the unclosed sum of the ranges of
the $P_\lambda$ and any $k$ in the unclosed sum of the ranges of the
$Q_\kappa$, and $(\bigvee P_\lambda, \bigvee Q_\kappa) \not\in \tilde{\Rc}$
then follows from Lemma \ref{orthogapprox}. This verifies condition
(ii) of Definition \ref{abstractqrel}.

Finally, let $P$ and $Q$ be projections in $\Mx \overline{\otimes} \Bc(l^2)$
and let $B \in I \otimes \Bc(l^2)$. Lemma
\ref{orthogapprox} implies that $(P,[BQ]) \in \tilde{\Rc}$ if and only if
$h \not\perp_\Rc Bk$ for some $h \in {\rm ran}(P)$ and $k \in {\rm ran}(Q)$.
Writing $B = I \otimes B_0$, we have $(B^*h)(x) = B_0^*(h(x))$ and
$(Bk)(y) = B_0(k(y))$ for all $x$ and $y$, so that
$$\langle (Bk)(y), h(x)\rangle = \langle B_0(k(y)), h(x)\rangle
= \langle k(y), B_0^*(h(x))\rangle = \langle k(y), (B^*h)(x)\rangle.$$
It follows that $h \not\perp_\Rc Bk$ for some $h \in {\rm ran}(P)$ and
$k \in {\rm ran}(Q)$ if and only if $B^*h \not\perp_\Rc k$ for some such
$h$ and $k$, and another application of Lemma \ref{orthogapprox} shows
that this is equivalent to $([B^*P],Q) \in \tilde{\Rc}$. This verifies condition
(iii) of Definition \ref{abstractqrel}.
\end{proof}

We now begin preparing for Theorem \ref{abstractchar}, which intrinsically
charaterizes quantum relations. We first collect some easy consequences
of Definition \ref{abstractqrel}.

\begin{lemma}\label{basicqrel}
Let $\Rc$ be an intrinsic quantum relation on a von Neumann algebra
$\Mx$.

\noindent (a) For any projections $P$ and $Q$ in
$\Mx \overline{\otimes} \Bc(l^2)$ we
have $(P,0), (0,Q) \not\in \Rc$.

\noindent (b) If $P$ and $Q$ are projections in $I \otimes\Bc(l^2)$
and $PQ = 0$ then $(P, Q) \not\in \Rc$.

\noindent (c) If $B \in I \otimes \Bc(l^2)$ is an isometry then
$$(P,Q) \in \Rc\quad\Leftrightarrow\quad (BPB^*, BQB^*) \in \Rc$$
for any projections $P$ and $Q$ in $\Mx \overline{\otimes} \Bc(l^2)$.

\noindent (d) If $P$ and $Q$ are projections in $I \otimes \Bc(l^2)$ with
orthogonal ranges and $P_1, P_2, Q_1, Q_2$ are projections in
$\Mx \overline{\otimes} \Bc(l^2)$ satisfying $P_1, P_2 \leq P$
and $Q_1, Q_2\leq Q$
then
$$(P_1 + Q_1, P_2 + Q_2) \in \Rc\quad \Leftrightarrow \quad
(P_1, P_2) \in \Rc\hbox{ or }(Q_1, Q_2) \in \Rc.$$
\end{lemma}

\begin{proof}
(a) Since $(0,0) \not\in \Rc$, this follows by taking $B = Q = 0$ or
$B = P = 0$ in condition (iii) of Definition \ref{abstractqrel}.

(b) Let $B = Q$; then $B^*P = 0$ and
$BQ = Q$, and so
$$(P, Q) = (P, [BQ]) \in \Rc
\quad\Leftrightarrow\quad
(0, Q) = ([B^*P], Q) \in \Rc.$$
But $(0, Q) \not\in \Rc$ by part (a), so
$(P,Q) \not\in \Rc$.

(c) We have $[B^*(BPB^*)] = [PB^*] = P$ and $[BQ] = BQB^*$, so
\begin{eqnarray*}
(P,Q) \in \Rc&\Leftrightarrow&([B^*(BPB^*)],Q) \in \Rc\cr
&\Leftrightarrow& (BPB^*, [BQ]) \in \Rc\cr
&\Leftrightarrow& (BPB^*, BQB^*) \in \Rc
\end{eqnarray*}
by condition (iii) of Definition \ref{abstractqrel}.

(d) By part (b) we have $(P, Q), (Q,P) \not\in \Rc$.
Since $P_1,P_2 \leq P$ and $Q_1,Q_2 \leq Q$,
this implies that $(P_1,Q_2), (Q_1,P_2) \not\in \Rc$ (using condition (ii)
of Definition \ref{abstractqrel}).
Now since $P_1 + Q_1 = P_1 \vee Q_1$ and $P_2 + Q_2 = P_2 \vee Q_2$,
condition (ii) of Definition \ref{abstractqrel} implies that
$(P_1 + Q_1, P_2 + Q_2) \in \Rc$ if and only if at least
one of $(P_1,P_2)$, $(P_1, Q_2)$, $(Q_1, P_2)$, or $(Q_1, Q_2)$
belongs to $\Rc$. As $(P_1, Q_2)$ and $(Q_1, P_2)$
cannot belong to $\Rc$, the desired conclusion follows.
\end{proof}

We introduce the following temporary notation. For any Hilbert space
$H$ and any vectors $v,w \in H \otimes l^2$, let $\omega_{v,w}$ be
the weak* continuous linear functional on $\Bc(H)$ defined by
$\omega_{v,w}(A) = \langle (A \otimes I)w,v\rangle$. Also, given a
von Neumann algebra $\Mx \subseteq \Bc(H)$, for
any $v \in H \otimes l^2$ let $P_{[v]}$ be the smallest projection
in $\Mx \overline{\otimes} \Bc(l^2)$ whose range contains $v$.

\begin{lemma}\label{mainlemma}
Let $\Rc$ be an intrinsic quantum relation on $\Mx \subseteq \Bc(H)$ and
suppose $v, w, v', w' \in H \otimes l^2$ satisfy
$\omega_{v,w} = \omega_{v',w'}$. Also assume $v-v', w-w' \in
H \otimes K_0$ for some finite dimensional subspace $K_0$ of $l^2$.
Then $(P_{[v]}, P_{[w]}) \in \Rc$ if and
only if $(P_{[v']}, P_{[w']}) \in \Rc$.
\end{lemma}

\begin{proof}
First, let $B_0$ be an isometry from $l^2$ onto an infinite-codimensional
subspace of $l^2$, let $B = I \otimes B_0$, and replace $v, w, v', w'$ with
$Bv, Bw, Bv', Bw'$. We can do this because
$$\omega_{Bv,Bw} = \omega_{v,w} = \omega_{v',w'} = \omega_{Bv',Bw'},$$
and
$$(P_{[v]},P_{[w]}) \in \Rc\quad\Leftrightarrow\quad
(P_{[Bv]}, P_{[Bw]}) = (BP_{[v]}B^*, BP_{[w]}B^*) \in \Rc$$
by part (c) of Lemma \ref{basicqrel} (and similarly for $v', w'$). Also
replace $K_0$ with $B_0(K_0)$.

Now let $P_0 \in \Bc(l^2)$ be the orthogonal projection onto $K_0$.
Then $(I \otimes P_0)v$
belongs to $H \otimes K_0$ and hence (since $K_0$ is finite dimensional)
is a finite linear combination of elementary tensors.
The same is true of $w$, $v'$, and $w'$. So there is a finite dimensional
subspace $H_0$ of $H$ such that the projections
of all four vectors onto $H \otimes K_0$ lie in $H_0 \otimes K_0$ and
$v - v', w - w' \in H_0 \otimes K_0$.

If $K_1$ is a finite dimensional subspace of $l^2$ that is orthogonal
to the range of $B_0$ then all four vectors are orthogonal to
$H \otimes K_1$. In the remainder of the proof we will work in the
finite dimensional space $H_0 \otimes (K_0 \oplus K_1)$, where $K_1$ is
chosen large enough to accomodate all computations below. (Specifically,
${\rm dim}(K_1) = ({\rm dim}(H_0) - 1)^2\cdot{\rm dim}(K_0) + 1$ would
suffice, but this number is not important.)

Identify $H_0$ with $\Cb^k$, $K_0 \oplus K_1$ with $\Cb^n$,
and $H_0 \otimes (K_0 \oplus K_1)$ with $\Cb^k \oplus \cdots \oplus \Cb^k$
($n$ summands). Let $(e_i)$ be the standard basis for $\Cb^k$ and let
$\bar{v} = (I \otimes P_0)v$, $\bar{w} = (I \otimes P_0)w$,
$\bar{v}' = (I \otimes P_0)v'$, and $\bar{w}' = (I \otimes P_0)w'$. 
The main step is to incrementally
convert $\bar{v}$ and $\bar{w}$, which initially lie in
$H_0 \otimes K_0$, into vectors of the form $a_1e_{i_1} \oplus
\cdots \oplus a_ne_{i_n}$ and $b_1e_{j_1} \oplus \cdots \oplus b_n e_{j_n}$,
now lying in $H_0 \otimes (K_0 \oplus K_1)$, without changing $\omega_{v,w}$ or
affecting whether $(P_{[v]},P_{[w]})$ lies in $\Rc$, and similarly to put
$\bar{v}'$ and $\bar{w}'$ in the form
$a_1'e_{i_1'} \oplus \cdots \oplus a_n'e_{i_n'}$ and
$b_1'e_{j_1'} \oplus \cdots \oplus b_n'e_{j_n'}$.

The main step is achieved in the following way. Say
$\bar{v} = v_1 \oplus \cdots \oplus v_n$ with each
$v_i \in H_0 \cong \Cb^k$ and suppose some $v_r$ is not of the form
$a_re_{i_r}$.  Let the corresponding decomposition of $\bar{w}$ be
$\bar{w} = w_1 \oplus \cdots \oplus w_n$. Because the dimension of
$K_1$ was sufficiently large, some of the $v_i$'s and $w_i$'s are zero
regardless of where we are in the construction. For notational simplicity
say $v_1 = w_1 = 0$ and $v_2 = a_1'e_1 + a_1''u$ with
$a_1'e_1 \neq 0$, $a_1''u \neq 0$, and $u \perp e_1$.
Now consider the vectors
$$v^0 = (a_1'e_1 - a_1''u) \oplus (a_1'e_1 + a_1''u) \oplus v_3 \oplus \cdots
\oplus v_n \oplus\tilde{v}$$
and
$$v^1 = (-a_1'e_1 + a_1''u) \oplus (a_1'e_1 + a_1''u) \oplus v_3 \oplus \cdots
\oplus v_n \oplus \tilde{v},$$
where $\tilde{v} = v - \bar{v} \in (H \otimes (K_0 \oplus K_1))^\perp$.
We have
$$(P_{[v^0]},P_{[w]}) \in \Rc\quad\Leftrightarrow\quad
(P_{[v^1]},P_{[w]}) \in \Rc$$
since $v^0 = Uv^1$ and $w = Uw$, and hence
$P_{[v^1]} = U^*P_{[v^0]}U$ and $P_{[w]} = U^*P_{[w]}U$, where
$U = -I_H \oplus I_H \oplus \cdots \oplus I_H \oplus I_H \in I \otimes \Bc(l^2)$.
Hence both pairs belong to $\Rc$ if and only if
$(P_{[v^0]}\vee P_{[v^1]}, P_{[w]}) \in \Rc$.
But $P_{[v^0]} \vee P_{[v^1]} = P_{[v]} \vee P_{[\hat{v}]}$ where
$\hat{v} = (a_1'e_1 - a_1''u) \oplus 0 \oplus \cdots \oplus 0\oplus 0$, and
we have $(P_{[\hat{v}]},P_{[w]}) \not\in \Rc$ by Lemma \ref{basicqrel} (d).
Thus
$$(P_{[v^0]},P_{[w]}) \in \Rc\quad\Leftrightarrow\quad
(P_{[v]},P_{[w]}) \in \Rc.$$
Also $\omega_{v,w} = \omega_{v^0,w}$.
So replace $v$ with $v^0$ and then apply the unitary
$$V =\frac{1}{\sqrt{2}}
\left[
\begin{matrix}
I&  I\cr
-I& I
\end{matrix}
\right]
\oplus I \oplus \cdots \oplus I \oplus I \in I \otimes B(l^2)$$
to $v$ and $w$ so that $\bar{v}$ becomes
$\sqrt{2}a_1'e_1 \oplus \sqrt{2}a_1''u \oplus v_3 \oplus \cdots \oplus v_r$
and $\bar{w}$ becomes $\frac{1}{\sqrt{2}}w_2 \oplus \frac{1}{\sqrt{2}}w_2
\oplus w_3 \oplus \cdots \oplus w_n$. The
end result is that $\bar{v}$ has moved one step closer to being in the
desired form, $\omega_{v,w}$ has not changed, and whether
$(P_{[v]},P_{[w]}) \in \Rc$ has not
changed. The vector $w$ has also been replaced by $Vw$, but if
$\bar{w}$ was already in the desired form this will still be the case.
So we achieve the main step by first putting $\bar{w}$ in the desired form and
then putting $\bar{v}$ in the desired form.

Now we proceed in four additional steps. As above let $\bar{v} =
v_1 \oplus \cdots \oplus v_n$ and $\bar{w} = w_1 \oplus \cdots \oplus w_n$.
We can also write $v_r = a_re_{i_r}$ and $w_r = b_re_{j_r}$ for $1 \leq r \leq
n$. First, for any $r$, if $v_r = 0$ then we also set $w_r = 0$ and if
$w_r = 0$ then we also set $v_r = 0$. This clearly does not change
$\omega_{v,w}$; to see that it also does not affect whether
$(P_{[v]}, P_{[w]}) \in \Rc$, suppose for simplicity that $w_1 = 0$
and consider the vector $v^1 = Uv$ where $U = -I \oplus I \oplus
\cdots \oplus I \oplus I$ as above. Letting $v^0 = v$ and arguing
exactly as in the main step yields the desired conclusion. (We do not need
to apply $V$ for this argument.) Make the same argument for $v'$ and $w'$.

We now have $v_rw_r \neq 0$ or $v_r = w_r = 0$ for all $r$. The next
step eliminates duplications where $i_r = i_s$ and $j_r = j_s$ but
$r \neq s$ (and $a_r, b_r, a_s, b_s$ are all nonzero). To do this,
for notational simplicity suppose $r = 1$ and $s = 2$ and apply a
unitary of the form
$$\left[
\begin{matrix}
\alpha I& \beta I\cr
-\bar{\beta} I& \bar{\alpha} I
\end{matrix}
\right]
\oplus I \oplus \cdots \oplus I \oplus I \in I \otimes B(l^2)$$
to $v$ and $w$ with $\alpha$ and $\beta$ chosen so that $\alpha a_1
+ \beta a_2 = 0$. This leads to $v_1 = 0$ and $v_2 = (-\bar{\beta}a_1
+ \bar{\alpha} a_2)e_{i_2}$. We may have $w_1 \neq 0$ but the argument of
the previous step can now be repeated to remedy this. Applying the
preceding construction repeatedly, we reach a point where $r \neq s$
and $v_r, w_r, v_s,w_s$ all nonzero implies either $i_r \neq i_s$ or
$j_r \neq j_s$. Make the same argument for $v'$ and $w'$.

In the next step we leave $v'$ and $w'$ intact and
apply a unitary in $I \otimes \Bc(l^2)$ to $v$ and
$w$ to ensure $i_r = i'_r$ and $j_r =j'_r$ for all $r$ such that
$v_r, w_r \neq 0$. We just use a permutation unitary to achieve this;
the pairs $(i_r, j_r)$ appearing in nonzero components
of $v$ and $w$ are the same up to rearrangement as the pairs $(i_r', j_r')$
appearing in nonzero components of $v'$ and $w'$ since $\omega_{v,w} =
\omega_{v',w'}$. This is because applying $\omega_{v,w}$ to the operator
$V_{e_{i_r}e_{j_r}}: u \mapsto \langle u,e_{j_r}\rangle e_{i_r}$ in
$\Bc(H)$ yields the result $\bar{a}_r b_r$, so we can diagnose whether
$(i_r,j_r)$ appears in a nonzero component of $(v,w)$ in this way.

We have reached the final step. By applying both sides of $\omega_{v,w} =
\omega_{v',w'}$ to $V_{e_{i_r}e_{j_r}}$ we see that
$a_r\bar{b}_r = a_r'\bar{b}_r'$ for all
values of $r$ such that $v_r,w_r,\bar{v}_r,\bar{w}_r$ are nonzero, so let
$$B = \frac{b_1}{b_1'}I \oplus \cdots \oplus \frac{b_n}{b_n'}I \oplus I
\in I \otimes \Bc(l^2)$$
(with the convention that $\frac{0}{0} = 1$, so that $B$ is invertible) and
observe that
$$(P_{[v]},P_{[w]}) = (P_{[v]}, [BP_{[w']}]) \in \Rc\quad\Leftrightarrow\quad
(P_{[v']}, P_{[w']}) = ([B^*P_{[v]}], P_{[w']}) \in \Rc.$$
Since the truth values of the conditions
$(P_{[v]},P_{[w]}) \in \Rc$ and
$(P_{[v']}, P_{[w']}) \in \Rc$ have not changed throughout the entire
process we conclude that $(P_{[v]},P_{[w]}) \in \Rc$ $\Leftrightarrow$
$(P_{[v']}, P_{[w']}) \in \Rc$ for the original values of $v,w,v',w'$.
\end{proof}

Let $P_S(A)$ denote the spectral projection of a self-adjoint operator
$A$ for the Borel set $S \subseteq \Rb$.

\begin{lemma}\label{nets}
Let $\{A_\lambda\}$ be a bounded net of self-adjoint
elements of a von Neumann algebra $\Mx$ and suppose
$A_\lambda \to A$ weak operator. Then for any $\epsilon > 0$ there is a
net of projections $\{P_\kappa\}$ in $\Mx$ which converges weak operator
to $P_{(-\infty, 0]}(A)$ and such that every
$P_\kappa$ is less than or equal to
$P_{(-\infty,\epsilon]}(A_\lambda)$ for some $\lambda$.
\end{lemma}

\begin{proof}
Fix $\epsilon > 0$. Then let $\delta > 0$, let $v_1, \ldots, v_m$ be unit
vectors in ${\rm ran}(P_{(-\infty, 0]}(A))$, and let $w_1, \ldots, w_n$ be
unit vectors in ${\rm ran}(P_{(0, \infty)}(A))$. It will suffice to find a
projection $P \leq P_{(-\infty,\epsilon]}(A_\lambda)$ in $\Mx$ for some
$\lambda$ such that (1) $\|Pv_i\|^2 \geq 1 - \delta$ for all $i$ and
(2) $\|Pw_j\|^2 \leq \delta$ for all $j$. We will achieve this with a
projection of the form $P = P_{[1/2,1]}((QRQ)^n)$ where $Q =
P_{(-\infty,\alpha]}(A_\lambda)$, for some $\lambda$ and some
$\alpha \leq \epsilon$, and $R = P_{(-\infty, 0]}(A)$.
It is easy to see that any such $P$ belongs to $\Mx$ and
satisfies $P \leq P_{(-\infty,\epsilon]}(A_\lambda)$.

We first check that property (2) can be assured independently of the
choice of $\alpha$ and $\lambda$ simply by chosing $n$
large enough. Since $P \leq 2(QRQ)^n$, this follows from the following
claim: if $Q$ and $R$ are any projections and $Rw = 0$ then
$$\|(RQ)^nw\| \leq \frac{1}{\sqrt{2n-1}}\|w\|$$
(for $n \geq 1$).
This can be seen by using the general form of two projections given
in \cite{Tak}, p.\ 308. Namely, we can decompose the Hilbert space so that
$$R = \left[
\begin{matrix}
R_0&0&0\cr
0&I&0\cr
0&0&0
\end{matrix}
\right]
\quad{\rm and}\quad
Q = \left[
\begin{matrix}
Q_0&0&0\cr
0&C^2&CS\cr
0&CS&S^2
\end{matrix}
\right]$$
with $R_0$ and $Q_0$ commuting projections, $0 \leq C,S \leq I$, and
$C^2 + S^2 = I$. For then $Rw = 0$ means that $w$ has the form $w = \left[
\begin{matrix}
w'\cr
0\cr
w''
\end{matrix}
\right]$ with $R_0w' = 0$ and hence that $(RQ)^nw = \left[
\begin{matrix}
0\cr
C^{2n-1}Sw''\cr
0
\end{matrix}
\right]$. So we just need to estimate the norm of $C^{2n-1}Sw'' =
C(I - S^2)^{n-1}Sw''$. But if $S$ is realized as multiplication by
$x$ then $(I-S^2)^{n-1}S$ becomes multiplication by $x(1-x^2)^{n-1}$,
which is extremized on $[0,1]$ at $x = \pm\frac{1}{\sqrt{2n-1}}$ and
hence has operator norm at most $\frac{1}{\sqrt{2n-1}}$ (ignoring the
$(1 - x^2)^{n-1}$ factor, which is at most one). Thus
$\|C^{2n-1}Sw''\| \leq \frac{1}{\sqrt{2n-1}}\|w\|$ and
this completes the proof of the claim.

So fix a value of $n$ that ensures property (2). For property (1),
let $\alpha = \min\{\delta/2(n+1),\epsilon\}$ and choose $\lambda$ so that
$\langle A_\lambda v_i, v_i\rangle \leq \alpha^3$ for $1 \leq i \leq m$.
Then set $Q = P_{(-\infty,\alpha]}(A_\lambda)$, $R = P_{(-\infty, 0]}(A)$, and
$P = P_{[1/2, 1]}((QRQ)^n)$. We must verify that $\|Pv_i\|^2
\geq 1 - \delta$ for all $i$. First, we have
$$\alpha\langle(I - Q)v_i,v_i\rangle \leq 
\langle A_\lambda v_i, v_i\rangle \leq \alpha^3$$
so that
$$\|(I - Q)v_i\|^2 = \langle (I-Q)v_i, v_i\rangle \leq \alpha^2,$$
i.e., $\|(I - Q)v_i\| \leq \alpha$. Also $(I - R)v_i = 0$, so
\begin{eqnarray*}
\|v_i - (QRQ)^nv_i\| &\leq& \|v_i - Qv_i\| + \|Qv_i - QRv_i\| +
\cdots + \|(QR)^nv_i - (QR)^nQv_i\|\cr
&\leq& \|(I-Q)v_i\| + \|Q\|\|(I-R)v_i\| + \cdots + \|(QR)^n\|\|(I-Q)v_i\|\cr
&\leq& (n+1)\alpha\cr
&\leq& \delta/2
\end{eqnarray*}
and hence
\begin{eqnarray*}
\|(I - P)v_i\|^2 &=& \|P_{[0,1/2)}((QRQ)^n) v_i\|^2\cr
&=& \|P_{(1/2,1]}(I - (QRQ)^n)v_i\|^2\cr
&=& \langle P_{(1/2,1]}(I - (QRQ)^n)v_i, v_i\rangle\cr
&\leq& 2\langle (I-(QRQ)^n)v_i, v_i\rangle\cr
&\leq&\delta.
\end{eqnarray*}
This shows that $\|Pv_i\|^2 \geq 1 - \delta$, as desired.
\end{proof}

\begin{theo}\label{abstractchar}
Let $\Mx \subseteq \Bc(H)$ be a von Neumann algebra and let $\Pc$ be
the set of projections in $\Mx \overline{\otimes} \Bc(l^2)$. If $\Vc$ is a
quantum relation on $\Mx$ (Definition \ref{quantrel}) then
$$\Rc_\Vc =
\{(P,Q) \in \Pc^2: P(A \otimes I)Q \neq 0\hbox{ for some }A \in \Vc\}$$
is an intrinsic quantum relation on $\Mx$ (Definition \ref{abstractqrel});
conversely, if $\Rc$ is an intrinsic quantum relation on $\Mx$ then
$$\Vc_\Rc = \{A \in \Bc(H): (P,Q) \not\in \Rc\quad\Rightarrow\quad
P(A \otimes I)Q = 0\}$$
is a quantum relation on $\Mx$. The two constructions are inverse to
each other.
\end{theo}

\begin{proof}
Observe first that
\begin{eqnarray*}
P(A \otimes I)Q = 0&\Leftrightarrow&\langle (A\otimes I)w,v\rangle = 0
\hbox{ for all }v \in {\rm ran}(P), w \in {\rm ran}(Q)\cr
&\Leftrightarrow& \omega_{v,w}(A) = 0\hbox{ for all }
v \in {\rm ran}(P), w \in {\rm ran}(Q).
\end{eqnarray*}

Now let $\Vc$ be a quantum relation on $\Mx$. Then conditions (i)
and (ii) of Definition \ref{abstractqrel} are easily seen to hold for
$\Rc_\Vc$, and condition (iii) holds because $B \in I \otimes \Bc(l^2)$
implies
\begin{eqnarray*}
P(A\otimes I)[BQ] \neq 0&\Leftrightarrow& P(A\otimes I)BQ \neq 0\cr
&\Leftrightarrow& PB(A\otimes I)Q \neq 0\cr
&\Leftrightarrow&[B^*P](A\otimes I)Q \neq 0.
\end{eqnarray*}
Also, using the fact that the weak operator topology agrees with the
strong operator topology on $\Pc$ it is easy to see that the complement
of $\Rc_\Vc$ is closed in $\Pc^2$. So $\Rc_\Vc$
is an intrinsic quantum relation.

Next let $\Rc$ be an intrinsic quantum relation on $\Mx$.
It is clear that $\Vc_\Rc$ is a linear subspace of $\Bc(H)$, 
it is weak* closed by the observation made at the start of the proof,
and it is a bimodule over $\Mx'$ because if $A,C \in \Mx'$, $B \in \Vc_\Rc$,
and $P,Q \in \Pc$ then $P(B\otimes I)Q = 0$ implies
$$P(ABC\otimes I)Q = P(A\otimes I)(B\otimes I)(C\otimes I)Q
= (A\otimes I)P(B\otimes I)Q(C\otimes I) = 0.$$
So $\Vc_\Rc$ is a quantum relation.

Now let $\Vc$ be a quantum relation, let $\Rc = \Rc_\Vc$, and let
$\tilde{\Vc} = \Vc_\Rc$. Then it is immediate that
$\Vc \subseteq \tilde{\Vc}$, and the reverse inclusion is just
the content of Lemma \ref{separation}.

Finally, let $\Rc$ be an intrinsic quantum relation, let $\Vc = \Vc_\Rc$,
and let $\tilde{\Rc} = \Rc_\Vc$. It is immediate that
$\tilde{\Rc} \subseteq \Rc$. For the reverse inclusion, fix $P$ and
$Q$ and suppose $P(A\otimes I)Q = 0$ for all $A \in \Vc$; we must show
that $(P,Q) \not\in \Rc$. By condition (ii) of Definition
\ref{abstractqrel} it will suffice to show that $(P_{[v]},P_{[w]})
\not\in \Rc$ for any $v \in {\rm ran}(P)$ and $w \in {\rm ran}(Q)$.

Let $E \subseteq \Bc(H)_*$ be the norm closure of
$\{\omega_{v,w}: (P_{[v]},P_{[w]}) \not\in \Rc\}$. We claim that
$E$ is a linear subspace. To see this, suppose $(P_{[v_1]},P_{[w_1]}),
(P_{[v_2]},P_{[w_2]}) \not\in \Rc$ and let $V_1, V_2 \in I \otimes\Bc(l^2)$
be isometries with orthogonal ranges; then $v = V_1v_1 + V_2v_2$
and $w = V_1w_1 + V_2w_2$ satisfy $P_{[v]} \leq V_1P_{[v_1]}V_1^* +
V_2P_{[v_2]}V_2^*$ and $P_{[w]} \leq V_1P_{[w_1]}V_1^* + V_2P_{[w_2]}V_2^*$,
and hence $(P_{[v]},P_{[w]}) \not\in \Rc$ by Lemma \ref{basicqrel} (c) and (d).
But $\omega_{v,w} = \omega_{v_1,w_1} + \omega_{v_2,w_2}$, so we have
shown that $E$ is stable under addition. Stability under scalar
multiplication is easy. This proves the claim.

Now let $v \in {\rm ran}(P)$ and $w \in {\rm ran}(Q)$; we must show
that $(P_{[v]}, P_{[w]}) \not\in \Rc$.
We may suppose $\|v\| = \|w\| = 1$.
If $\omega_{v,w} \not\in E$ then there would
exist $A \in \Bc(H)$ such that $\omega_{v,w}(A) \neq 0$ but $\omega(A) = 0$
for all $\omega \in E$; then $A \in \Vc$ but $P_{[v]}(A\otimes I)P_{[w]} \neq 0$,
contradicting the fact that $P_{[v]} \subseteq P$ and $P_{[w]} \subseteq Q$.
Thus $\omega_{v,w} \in E$. We conclude the proof by showing that any
unit vectors $v,w \in H \otimes l^2$ with $\omega_{v,w} \in E$ satisfy
$(P_{[v]},P_{[w]}) \not\in\Rc$.

Let $B_0 \in \Bc(l^2)$ be an isometry with infinite codimensional
range and let $B = I \otimes B_0$. We may replace $v$ and $w$ with
$Bv$ and $Bw$ since $\omega_{v,w} = \omega_{Bv,Bw}$ and
$(P_{[v]}, P_{[w]}) \in \Rc$ $\Leftrightarrow$
$(P_{[Bv]},P_{[Bw]}) = (BP_{[v]}B^*, BP_{[w]}B^*) \in \Rc$.
Now since $\omega_{v,w} \in E$, for any $0 < \epsilon \leq 1$ there exist
$v',w' \in H\otimes l^2$ such that $\|\omega_{v,w} - \omega_{v',w'}\|
\leq \epsilon$ and $(P_{[v']},P_{[w']}) \not\in \Rc$. We may also replace
$v'$ and $w'$ with $Bv'$ and $Bw'$.
Find a finite rank projection $R_0 \in \Bc(l^2)$ such that
$\|(I - R)v\|, \|(I - R)w\|, \|(I - R)v'\|, \|(I - R)w'\|
\leq \epsilon$ where $R = I \otimes R_0$, and let $v_1 = Rv + (I - R)v'$ and
$w_1 = Rw + (I - R)w'$. Then $\|v - v_1\|, \|w - w_1\| \leq 2\epsilon$,
\begin{eqnarray*}
\|\omega_{v_1,w_1} - \omega_{v',w'}\|
&\leq& \|\omega_{v_1,w_1} - \omega_{v_1,w}\|
+ \|\omega_{v_1,w} - \omega_{v,w}\|
+ \|\omega_{v,w} - \omega_{v',w'}\|\cr
&\leq& \|v_1\|\|w_1 - w\| + \|v_1 - v\|\|w\| + \epsilon\cr
&\leq& (1+\epsilon)\cdot 2\epsilon + 2\epsilon + \epsilon\cr
&\leq& 7\epsilon,
\end{eqnarray*}
and $v_1- v', w_1 - w' \in H \otimes K_0$ where $K_0 = {\rm ran}(R_0)$.
This implies that $\omega_{v_1,w_1} - \omega_{v',w'}$ is ${\rm tr}(\cdot A)$
for some finite rank operator $A$ with ${\rm tr}(|A|) \leq 7\epsilon$.
Thus $\omega_{v_1,w_1} - \omega_{v',w'} = \omega_{v_2,w_2}$ for some
vectors $v_2,w_2 \in H\otimes K_1$ such that $\|v_2\|, \|w_2\| \leq
\sqrt{7\epsilon}$ and where $K_1$ is a finite dimensional subspace of
$l^2$ that we can take to be
orthogonal to ${\rm ran}(B_0)$. Finally let $v_3 = v_1 - v_2$
and $w_3 = w_1 + w_2$. We obtain $\|v - v_3\|, \|w - w_3\|
\leq 2\epsilon + \sqrt{7\epsilon}$, $\omega_{v_3,w_3} = \omega_{v',w'}$,
and $v_3 - v', w_3 - w' \in H\otimes(K_0 \oplus K_1)$.
Since $(P_{[v']}, P_{[w']}) \not\in\Rc$, Lemma \ref{mainlemma} implies
that $(P_{[v_3]}, P_{[w_3]}) \not\in\Rc$.

Letting $\epsilon \to 0$, we thus get a sequence of pairs of projections
$P_n,Q_n \in \Pc$ such that $(P_n,Q_n) \not\in \Rc$ and
$\|P_n v\|, \|Q_nw\| \to 1$. Passing to a weak operator convergent
subnet $\{P_\lambda \oplus Q_\lambda\}$ of the sequence $\{P_n \oplus Q_n\}$,
we have $(I -P_\lambda) \oplus (I -Q_\lambda) \to A_1 \oplus A_2$ weak
operator for some positive $A_1,A_2 \in \Mx\overline{\otimes} \Bc(l^2)$ such
that $v \in {\rm ker}(A_1)$ and $w \in {\rm ker}(A_2)$. By Lemma \ref{nets}
with $\epsilon = 1/2$ we can find a net of projections $P_\lambda' \oplus
Q_\lambda' \leq P_\lambda \oplus Q_\lambda$ which converge weak operator
to $\tilde{P} \oplus \tilde{Q}$ where $\tilde{P}$ and $\tilde{Q}$ are the
orthogonal projections onto ${\rm ker}(A_1)$ and ${\rm ker}(A_2)$,
respectively. So $(\tilde{P},\tilde{Q}) \not\in \Rc$ since $\Rc$ is open,
and then $P_{[v]} \leq \tilde{P}$, $P_{[w]} \leq \tilde{Q}$ implies
$(P_{[v]},P_{[w]}) \not\in \Rc$. This is what we needed to prove.
\end{proof}

\subsection{Quantum tori}\label{torisect}
In this section we analyze quantum relations on quantum tori which
satisfy a condition that is naturally understood as ``translation
invariance''. We find that this class of quantum relations is quite
tractable.

Quantum tori are the simplest examples of noncommutative manifolds.
They are related to the quantum plane, which plays the role of the
phase space of a spinless one-dimensional particle.
The classical version of such a system has phase space $\Rb^2$, with
the point $(q,p) \in \Rb^2$ representing a state with position $q$ and
momentum $p$, so that the position and momentum observables are just
the coordinate functions on phase space. When such a system is quantized
the position and momentum observables are modelled by unbounded
self-adjoint operators $Q$ and $P$ satisfying $QP - PQ = i\hbar I$.
Polynomials in $Q$ and $P$ can then be seen as a quantum
analog of polynomial functions on $\Rb^2$. The quantum analog of the
continuous functions on the torus --- equivalently, the
$(2\pi, 2\pi)$-periodic continuous functions on the plane --- is the
C*-algebra generated by the unitary operators $e^{iQ}$ and $e^{iP}$,
which satisfy the commutation relation $e^{iQ}e^{iP} =
e^{-i\hbar}e^{iP}e^{iQ}$. For more background see \cite{Rie} or
Sections 4.1, 4.2, 5.5, and 6.6 of \cite{W5}.

Let $\Tb = \Rb/2\pi\Zb$ and fix $\hbar \in \Rb$. Let $\{e_{m,n}\}$ be
the standard basis of $l^2(\Zb^2)$. We model the quantum
tori on $l^2(\Zb^2)$ as follows.

\begin{defi}
Let $U_\hbar$ and $V_\hbar$ be the unitaries in $\Bc(l^2(\Zb^2))$ defined by
\begin{eqnarray*}
U_\hbar e_{m,n} &=& e^{-i\hbar n/2} e_{m+1,n}\cr
V_\hbar e_{m,n} &=& e^{i\hbar m/2} e_{m,n+1}.
\end{eqnarray*}
The {\it quantum torus von Neumann algebra} for the given value of
$\hbar$ is the von Neumann algebra
$W^*(U_\hbar,V_\hbar)$ generated by $U_\hbar$ and $V_\hbar$.
\end{defi}

If $\hbar$ is an irrational multiple of $\pi$ then $W^*(U_\hbar,V_\hbar)$ is a
hyperfinite $II_1$ factor. We will not need this fact.

Conjugating $U_\hbar$ and $V_\hbar$ by the Fourier transform
$\Fc: L^2(\Tb^2) \to l^2(\Zb^2)$ yields the operators
\begin{eqnarray*}
\hat{U}_\hbar f(x,y) &=& e^{ix} f\left(x,y-\frac{\hbar}{2}\right)\cr
\hat{V}_\hbar f(x,y) &=& e^{iy} f\left(x + \frac{\hbar}{2},y\right)
\end{eqnarray*}
on $L^2(\Tb^2)$, with $W^*(\hat{U}_\hbar,\hat{V}_\hbar)$
reducing to the algebra of bounded
multiplication operators when $\hbar = 0$. However, for our purposes
the $l^2(\Zb^2)$ picture is more convenient.

The following commutation relations will be useful. For $f \in
l^\infty(\Zb^2)$ and $k,l \in \Zb$ let $\tau_{k,l}f$ be the translated
function $\tau_{k,l}f(m,n) = f(m-k,n-l)$. Then
\begin{eqnarray*}
U_\hbar V_\hbar &=& e^{-i\hbar}V_\hbar U_\hbar\cr
U_\hbar^kV_\hbar^lM_f &=& M_{\tau_{k,l}f}U_\hbar^kV_\hbar^l.
\end{eqnarray*}
In particular, $U_\hbar^kV_\hbar^lM_{e^{-i(mx + ny)}} =
e^{i(kx+ly)}M_{e^{-i(mx + ny)}}U_\hbar^kV_\hbar^l$.

Our main technical tool will be a kind of Fourier analysis. We introduce
the relevant definitions.

\begin{defi}\label{fdef}
Let $A \in \Bc(l^2(\Zb^2))$.

\noindent (a) For $x,y \in \Tb$ define
$$\theta_{x,y}(A) = M_{e^{i(mx + ny)}} A M_{e^{-i(mx + ny)}}.$$

\noindent (b) For $k,l \in \Zb$ define
$$A_{k,l} = \frac{1}{4\pi^2} \int_0^{2\pi} \int_0^{2\pi}
e^{-i(kx + ly)} \theta_{x,y}(A)\, dxdy.$$
We call $A_{k,l}$ the {\it $(k,l)$ Fourier term} of $A$.

\noindent (c) For $k,l \in \Nb$ define
$$S_{k,l}(A) = \sum_{|k'| \leq k, |l'| \leq l} A_{k',l'}$$
and for $N \in \Nb$ define
$$\sigma_N(A) = \frac{1}{N^2} \sum_{0 \leq k,l \leq N-1} S_{k,l}(A).$$
\end{defi}

In the $L^2(\Tb^2)$ picture the operator $M_{e^{i(mx + ny)}}$ on
$l^2(\Zb^2)$ becomes translation by $(-x,-y)$, so that $\theta_{x,y}$ is
conjugation by a translation.

The integral used to define $A_{k,l}$ can be understood in a weak
sense: for any vectors $v,w \in l^2(\Zb^2)$ we take
$\langle A_{k,l}w,v\rangle$ to be $\frac{1}{4\pi^2} \int_0^{2\pi}\int_0^{2\pi}
e^{-i(kx + ly)} \langle \theta_{x,y}(A)w,v\rangle\, dxdy$.
In particular, if $w = e_{m,n}$ and $v = e_{m',n'}$ then we have
$$\langle A_{k,l} e_{m,n}, e_{m',n'}\rangle =
\begin{cases}
\langle Ae_{m,n}, e_{m',n'}\rangle&\hbox{if $m' = m + k$ and
$n' = n + l$}\cr
0&\hbox{otherwise}.
\end{cases}\eqno{(*)}$$
The $A_{k,l}$ are something like Fourier coefficients, the $S_{k,l}(A)$
like partial sums of a Fourier series, and the $\sigma_N(A)$ like
Ces\`{a}ro means. The next few results are minor reworkings of
material in \cite{W5}.

\begin{prop}\label{cesaro}
Let $A \in \Bc(l^2(\Zb^2))$. Then $\|\sigma_N(A)\| \leq \|A\|$ for all $N$
and $\sigma_N(A) \to A$ weak operator.
\end{prop}

\begin{proof}
We have
$$\sigma_N(A) = \frac{1}{4\pi^2}\int_0^{2\pi}\int_0^{2\pi}
K_N(x)K_N(y)\theta_{x,y}(A)\, dxdy$$
where $K_N$ is the Fej\'{e}r kernel,
$$K_N(x) = \sum_{n=-N+1}^{N-1} \left(1 - \frac{|n|}{N}\right)e^{inx}
= \frac{1}{N}\left(\frac{\sin(Nx/2)}{\sin(x/2)}\right)^2.$$
Since $\|K_N\|_1 = 2\pi$,
this shows that $\|\sigma_N(A)\| \leq \|A\|$, so the sequence
$\{\sigma_N(A)\}$ is bounded. So it will suffice to check weak operator
convergence against the vectors $e_{m,n}$. But if $|m'-m|,
|n'-n| \leq N$ then
$$\langle \sigma_N(A) e_{m,n}, e_{m',n'}\rangle =
\left(1 - \frac{|m'-m|}{N}\right)\left(1 - \frac{|n'-n|}{N}\right)
\langle Ae_{m,n},e_{m',n'}\rangle,$$
and this converges to $\langle Ae_{m,n},e_{m',n'}\rangle$ as
$N \to \infty$, as desired.
\end{proof}

\begin{lemma}\label{fourierlemma}
(a) For any $k,l \in \Zb$ the map $A \mapsto A_{k,l}$ is weak* continuous
from $\Bc(l^2(\Zb^2))$ to $\Bc(l^2(\Zb^2))$.

\noindent (b) Let $\Mx \cong l^\infty(\Zb^2)$ be the von Neumann
algebra of bounded multiplication operators in $\Bc(l^2(\Zb^2))$.
Then for any $A \in \Bc(l^2(\Zb^2))$ and any $k,l \in \Zb$
we have $A_{k,l} \in \Mx\cdot U_\hbar^kV_\hbar^l$.
\end{lemma}

\begin{proof}
(a) By the Krein-Smuliyan theorem we need only check that if
$A_\lambda \to A$ boundedly weak* then their $(k,l)$ Fourier
terms converge, for every $k$ and $l$. Then since the
net is bounded it is enough to check convergence against basis
vectors, and this follows immediately from the formula ($*$) above.

(b) A simple change of variable in the formula for $A_{0,0}$
shows that it commutes with the operator $M_{e^{i(mx + ny)}}$
for all $x,y \in \Tb$. But these operators generate the maximal
abelian von Neumann algebra $\Mx$, so we must have $A_{0,0} \in \Mx$.
The result for arbitrary $k$ and $l$ now follows from the observation
that the $(0,0)$ Fourier term of $AV_\hbar^{-l}U_\hbar^{-k}$ is
$A_{k,l}V_\hbar^{-l}U_\hbar^{-k}$.
\end{proof}

\begin{prop}\label{belongs}
Let $A \in \Bc(l^2(\Zb^2))$. Then $A \in W^*(U_\hbar,V_\hbar)$ if and only if
$A_{k,l}$ is a scalar multiple of $U_\hbar^kV_\hbar^l$ for all $k,l \in \Zb$.
\end{prop}

\begin{proof}
Suppose $A = \sum \alpha_{k,l}U_\hbar^kV_\hbar^l$ is a polynomial in $U_\hbar$
and $V_\hbar$. Then the formula ($*$) given above shows that $A_{k,l}$
equals $\alpha_{k,l}U_\hbar^kV_\hbar^l$ for all $k$ and $l$, so that the
$(k,l)$ Fourier term of $A$ is a scalar multiple of $U_\hbar^kV_\hbar^l$.
The forward implication now follows for all $A \in W^*(U_\hbar,V_\hbar)$
since the map $A \mapsto A_{k,l}$ is weak* continuous and the polynomials
in $U_\hbar$ and $V_\hbar$ are weak* dense in $W^*(U_\hbar, V_\hbar)$.
Conversely, if every Fourier term belongs to $W^*(U_\hbar,V_\hbar)$ then
$A$ must also belong to $W^*(U_\hbar,V_\hbar)$ by Proposition \ref{cesaro}.
\end{proof}

\begin{coro}
The commutant of $W^*(U_\hbar,V_\hbar)$ is $W^*(U_{-\hbar}, V_{-\hbar})$.
\end{coro}

\begin{proof}
A straightforward calculation shows that $U_\hbar$ and $V_\hbar$ each commute
with both of $U_{-\hbar}$ and $V_{-\hbar}$. From this it easily
follows that $W^*(U_\hbar,V_\hbar)$ is contained in the commutant of
$W^*(U_{-\hbar}, V_{-\hbar})$. Conversely, suppose $A \in \Bc(l^2(\Zb^2))$
commutes with $U_{-\hbar}$ and $V_{-\hbar}$. For any $k,l,m,n \in \Zb$
we have 
\begin{eqnarray*}
\langle V_{-\hbar}^{-n}U_{-\hbar}^{-m} A e_{m,n}, e_{k,l}\rangle
&=& \langle A e_{m,n}, U_{-\hbar}^mV_{-\hbar}^ne_{k,l}\rangle\cr
&=& e^{i\hbar(nk- ml - mn)/2}\langle A e_{m,n}, e_{m+k,n+l}\rangle
\end{eqnarray*}
and
$$\langle AV_{-\hbar}^{-n}U_{-\hbar}^{-m} e_{m,n}, e_{k,l}\rangle
= e^{-i\hbar mn/2}\langle A e_{0,0}, e_{k,l}\rangle.$$
Thus, letting $\alpha = e^{i\hbar kl/2}\langle Ae_{0,0}, e_{k,l}\rangle$,
we have
\begin{eqnarray*}
\langle A_{k,l} e_{m,n}, e_{m+k,n+l}\rangle
&=& \langle A e_{m,n}, e_{m+k,n+l}\rangle\cr
&=& e^{i\hbar(ml - nk)/2}\langle Ae_{0,0}, e_{k,l}\rangle\cr
&=& \langle \alpha U_\hbar^kV_\hbar^l e_{m,n}, e_{m+k,n+l}\rangle.
\end{eqnarray*}
And since
$$\langle A_{k,l} e_{m,n}, e_{m',n'}\rangle = 0
= \langle \alpha U_\hbar^kV_\hbar^l e_{m,n}, e_{m',n'}\rangle$$
if either $m' \neq m + k$ or $n' \neq n + l$, we conclude that
$A_{k,l} = \alpha U_\hbar^kV_\hbar^l$. Thus $A \in W^*(U_\hbar,V_\hbar)$ by Proposition
\ref{belongs}. This completes the proof.
\end{proof}

With this background material in place, we now proceed to analyze
translation invariant quantum relations on quantum tori.

\begin{defi}\label{titori}
(a) A quantum relation $\Vc \subseteq \Bc(l^2(\Zb^2))$ on the quantum
torus von Neumann algebra $W^*(U_\hbar,V_\hbar)$ is {\it translation
invariant} if $\theta_{x,y}(\Vc) = \Vc$ for all $x,y \in \Tb$.

\noindent (b) A subspace $\Ec$ of the von Neumann algebra $\Mx \cong
l^\infty(\Zb^2)$ of bounded multiplication operators on $l^2(\Zb^2)$
is {\it translation invariant} if
$$M_f \in \Ec \quad\Rightarrow\quad M_{\tau_{k,l}f} \in \Ec$$
for all $k,l \in \Zb$, where $\tau_{k,l}f$ is the translation operator
defined just before Definition \ref{fdef}.
\end{defi}

(The two notions of translation invariance are not directly related.
One refers to invariance under an action of $\Tb^2$, the other to
invariance under an action of $\Zb^2$.)

First we indicate an equivalent formulation of translation invariance
framed in terms of projections.

\begin{prop}\label{tirel}
Let $\Vc$ be a quantum relation on $W^*(U_\hbar,V_\hbar)$ and let $\Rc$ be
the corresponding intrinsic quantum relation (Theorem \ref{abstractchar}). Then
$\Vc$ is translation invariant if and only if
$$(P,Q) \in \Rc\quad\Leftrightarrow\quad ((\theta_{x,y} \otimes I)(P),
(\theta_{x,y}\otimes I)(Q)) \in \Rc$$
for all $x,y \in \Tb$ and all projections
$P,Q \in W^*(U_\hbar,V_\hbar) \overline{\otimes} \Bc(l^2)$.
\end{prop}

\begin{proof}
The forward implication holds because
$$(\theta_{x,y}\otimes I)(P(A \otimes I)Q) =
(\theta_{x,y}\otimes I)(P)(\theta_{x,y}(A) \otimes I)(\theta_{x,y}\otimes I)(Q),$$
so that translation invariance of $\Vc$ implies that there exists
$A \in \Vc$ such that $P(A \otimes I)Q \neq 0$ if and only if there
exists $A \in \Vc$ such that
$(\theta_{x,y}\otimes I)(P)(A \otimes I)(\theta_{x,y}\otimes I)(Q)
\neq 0$. Conversely, suppose $\Vc$ is not translation invariant and
find $A \in \Vc$ and $x,y \in \Tb$ such that $\theta_{x,y}(A) \not\in \Vc$.
By Lemma \ref{separation} we can find projections $P, Q \in
W^*(U_\hbar,V_\hbar) \overline{\otimes} \Bc(l^2)$ such that
$P(\theta_{x,y}(A)\otimes I)Q \neq 0$ but $P(B \otimes I)Q = 0$ for
all $B \in \Vc$. Then $(P,Q) \not\in \Rc$ but
$((\theta_{-x,-y}\otimes I)(P),(\theta_{-x,-y}\otimes I)(Q)) \in \Rc$
because
$$(\theta_{-x,-y}\otimes I)(P)(A \otimes I)(\theta_{-x,-y}\otimes I)(Q)
= (\theta_{-x,-y}\otimes I)(P(\theta_{x,y}(A)\otimes I)Q) \neq 0.$$
This proves the reverse implication.
\end{proof}

\begin{theo}\label{transinvqr}
Let $\Mx \cong l^\infty(\Zb^2)$ be the von Neumann
algebra of bounded multiplication operators in $\Bc(l^2(\Zb^2))$
and let $\Ec$ be a weak* closed, translation invariant subspace of
$\Mx$. Then
$$\Vc_\Ec = \{A \in \Bc(l^2(\Zb^2)):
A_{k,l} \in \Ec\cdot U_{-\hbar}^kV_{-\hbar}^l\hbox{ for all }k,l \in \Zb\}$$
is a translation invariant quantum relation on $W^*(U_\hbar,V_\hbar)$. Every
translation invariant quantum relation on $W^*(U_\hbar,V_\hbar)$ is of this
form.
\end{theo}

\begin{proof}
Since $\Ec$ is weak* closed, so is $\Ec\cdot U_{-\hbar}^kV_{-\hbar}^l$.
Together with weak* continuity of the map $A \mapsto A_{k,l}$
(Lemma \ref{fourierlemma} (a)), this
implies that $\Vc_\Ec$ is weak* closed.
$\Vc_\Ec$ is clearly a linear subspace of $\Bc(l^2(\Zb^2))$.
To see that it is a bimodule over $W^*(U_{-\hbar},V_{-\hbar})$ it
suffices by weak* continuity to demonstrate stability under left
and right multiplication by monomials in $U_{-\hbar}$ and $V_{-\hbar}$;
this holds because the $(k,l)$ Fourier term of
$AU_{-\hbar}^mV_{-\hbar}^n$ is $A_{k-m,l-n}U_{-\hbar}^mV_{-\hbar}^n$,
and the $(k,l)$ Fourier term of $U_{-\hbar}^mV_{-\hbar}^nA$ is
$U_{-\hbar}^mV_{-\hbar}^nA_{k-m,l-n}$. So if $A \in \Vc_\Ec$ then
$A_{k-m,l-n} \in \Ec\cdot U_{-\hbar}^{k-m}V_{-\hbar}^{l-n}$, say $A_{k-m,l-n} =
M_fU_{-\hbar}^{k-m}V_{-\hbar}^{l-n}$, and the $(k,l)$ Fourier term
of $AU_{-\hbar}^mV_{-\hbar}^n$ is
$$M_fU_{-\hbar}^{k-m}V_{-\hbar}^{l-n}U_{-\hbar}^mV_{-\hbar}^n
= e^{i\hbar(nm -lm)}M_fU_{-\hbar}^kV_{-\hbar}^l \in \Ec\cdot U_{-\hbar}^kV_{-\hbar}^l,$$
while the $(k,l)$ Fourier term of $U_{-\hbar}^mV_{-\hbar}^nA$ is
$$U_{-\hbar}^mV_{-\hbar}^nM_fU_{-\hbar}^{k-m}V_{-\hbar}^{l-n} =
e^{i\hbar(nm-nk)}M_{\tau_{m,n}f}U_{-\hbar}^kV_{-\hbar}^l \in \Ec\cdot U_{-\hbar}^kV_{-\hbar}^l$$
since $\Ec$ is translation invariant. Finally, $\Vc_\Ec$ is
translation invariant because the $(k,l)$ Fourier term of
$\theta_{x,y}(A)$ is $\theta_{x,y}(A_{k,l})$, so if $A \in \Vc_\Ec$ then
$A_{k,l} \in \Ec\cdot U_{-\hbar}^kV_{-\hbar}^l$, say $A_{k,l} =
M_fU_{-\hbar}^kV_{-\hbar}^l$, and the $(k,l)$ Fourier term of
$\theta_{x,y}(A)$ is
$$\theta_{x,y}(M_fU_{-\hbar}^kV_{-\hbar}^l) =
e^{i(kx + ly)}M_fU_{-\hbar}^kV_{-\hbar}^l \in
\Ec\cdot U_{-\hbar}^kV_{-\hbar}^l.$$
So $\Vc_\Ec$ is a translation invariant quantum relation on
$W^*(U_\hbar,V_\hbar)$.

Now let $\Vc$ be any translation invariant quantum relation on
$W^*(U_\hbar,V_\hbar)$
and let $\Ec = \Vc \cap \Mx$. Then $\Ec$ is clearly a weak* closed subspace
of $\Mx$, and it is translation invariant because $M_f \in \Ec$ implies
$$M_{\tau_{k,l}f} = U_{-\hbar}^kV_{-\hbar}^lM_fV_{-\hbar}^{-l}U_{-\hbar}^{-k}
\in \Ec.$$
We claim that $\Vc = \Vc_\Ec$. To see this, first let $A \in \Vc$; then
for any $k,l \in \Zb$ we have $A_{k,l} \in \Vc$ by translation invariance
and weak* closure of $\Vc$, and
$A_{k,l} \in \Mx \cdot U_{-\hbar}^kV_{-\hbar}^l$ by Lemma \ref{fourierlemma},
so $A_{k,l}V_{-\hbar}^{-l}U_{-\hbar}^{-k} \in \Ec$. This shows that $A
\in \Vc_\Ec$ and we conclude that $\Vc \subseteq \Vc_\Ec$. Conversely,
if $A \in \Vc_\Ec$ then $A_{k,l} \in \Ec\cdot U_{-\hbar}^kV_{-\hbar}^l
\subseteq \Vc$ for all $k,l \in \Zb$, and this implies that $A \in \Vc$
by Proposition \ref{cesaro}. So $\Vc_\Ec \subseteq \Vc$.
\end{proof}

Retaining the notation of Theorem \ref{transinvqr},
for any closed subset $S \subseteq \Tb^2$ the smallest weak* closed
translation invariant subspace of $\Mx$ that contains
$M_{e^{i(mx + ny)}}$ if and only if $(x,y) \in S$ is
\begin{eqnarray*}
\Ec_0(S) &=& \{M_f: f \in l^\infty(\Zb^2)\hbox{ and }\sum f\bar{g} = 0
\hbox{ for all } g \in l^1(\Zb^2)\hbox{ such that }\hat{g}|_S = 0\}\cr
&=& \overline{\rm span}^{wk^*}\{M_{e^{i(mx + ny)}}: (x,y) \in S\}
\end{eqnarray*}
and the largest is
\begin{eqnarray*}
\Ec_1(S) &=& \{M_f: f \in l^\infty(\Zb^2)\hbox{ and }\sum f\bar{g} = 0
\hbox{ for all } g \in l^1(\Zb^2)\hbox{ such that }\hat{g}|_T = 0\cr
&&\phantom{\limsup\limsup}\hbox{for some neighborhood $T$ of $S$}\}\cr
&=& \bigcap_{\epsilon > 0} \Ec_0(N_\epsilon(S))
\end{eqnarray*}
where $N_\epsilon(S)$ is the open $\epsilon$-neighborhood of $S$
(see, e.g., Section 3.6.16 of \cite{Pal}).

Now for any weak* closed translation invariant subspace $\Ec$ of $\Mx$ let
$$S(\Ec) = \{(x,y) \in \Tb^2: M_{e^{i(mx + ny)}} \in \Ec\}.$$
Then $S(\Ec)$ is a closed subset of $\Tb^2$ and $\Ec_0(S) \subseteq \Ec
\subseteq \Ec_1(S)$, and we immediately infer the following corollary.

\begin{coro}\label{transinvcor}
Let $\Mx \cong l^\infty(\Zb^2)$ be the von Neumann algebra of bounded
multiplication operators in $\Bc(l^2(\Zb^2))$, let $\Vc$ be a translation
invariant quantum relation on $W^*(U_\hbar,V_\hbar)$,
and let $\Ec = \Vc \cap \Mx$. Then
$$\Vc_{\Ec_0(S)} \subseteq \Vc \subseteq \Vc_{\Ec_1(S)}$$
where $S = S(\Ec)$ and $\Vc_\Ec$ is as in Theorem \ref{transinvqr}.
\end{coro}

In particular, if $S(\Ec)$ is a set of spectral synthesis then
$\Ec_0(S) = \Ec_1(S)$ and hence $\Vc = \Vc_{\Ec_0(S)} = \Vc_{\Ec_1(S)}$.



\bigskip
\bigskip
\begin{thebibliography}{aaaaaaaa}

\bibitem{Arv}
W.\ Arveson, Operator algebras and invariant subspaces, {\it Ann.\ of
Math.\ \bf 100} (1974), 433-532.

\bibitem{Con}
A.\ Connes, {\it Noncommutative Geometry}, Academic Press, 1994.

\bibitem{Dav}
K.\ R.\ Davidson, {\it Nest Algebras}, Longman, 1988.

\bibitem{DSW}
R.\ Duan, S.\ Severini, and A.\ Winter, Zero-error communication via
quantum channels, non-commutative graphs and a quantum Lov\'asz function,
arXiv:1002.2514.

\bibitem{Erd}
J.\ A.\ Erdos, Reflexivity for subspace maps and linear spaces of
operators, {\it Proc.\ London Math.\ Soc.\ \bf 52} (1986), 582-600.

\bibitem{Fro}
J.\ Froelich, Compact operators, invariant subspaces, and spectral
synthesis, {\it J.\ Funct.\ Anal.\ \bf 81} (1988), 1-37. 

\bibitem{H}
F.\ Hirsch, Intrinsic metrics and Lipschitz functions,
{\it J.\ Evol.\ Equ.\ \bf 3} (2003), 11-25.

\bibitem{H1}
{---------}, Measurable metrics, intrinsic metrics and Lipschitz functions,
{\it Current trends in potential theory} (2005), 47-61.

\bibitem{H2}
{---------}, Measurable metrics and Gaussian concentration,
{\it Forum Math.\ \bf 18} (2006), 345-363.

\bibitem{KW}
G.\ Kuperberg and N.\ Weaver, A von Neumann algebra approach to
quantum metrics, manuscript.

\bibitem{Lar}
D.\  Larson, Annihilators of operator algebras, {\it Topics in
Modern Operator Theory \bf 6} (1982), 119-130.

\bibitem{Pal}
T.\ Palmer, {\it Banach Algebras and the General Theory of
$*$-Algebras}, Vol.\ 2, Cambridge University Press, 2001.

\bibitem{Rie}
M.\ Rieffel, Noncommutative tori---a case study of noncommutative
differentiable manifolds, in {\it Geometric and Topological
Invariants of Elliptic Operators} (1990), 191-211.

\bibitem{Tak}
M.\ Takesaki, {\it Theory of Operator Algebras I}, Springer, 1979.

\bibitem{W0}
N.\ Weaver, Nonatomic Lipschitz spaces, {\it Studia Math.\ \bf 115}
(1995), 277-289.

\bibitem{W1}
{---------}, Weak*-closed derivations from $C[0,1]$ into $L^\infty[0,1]$,
{\it Canad.\ Math.\ Bull.\ \bf 39} (1996), 367-375.

\bibitem{W2}
{---------}, Lipschitz algebras and derivations of von Neumann
algebras, {\it J.\ Funct.\ Anal.\ \bf 139} (1996), 261-300.

\bibitem{W4}
{---------}, {\it Lipschitz Algebras}, World Scientific, 1999.

\bibitem{W3}
{---------}, Lipschitz algebras and derivations II: exterior
differentiation, {\it J.\ Funct.\  Anal.\ \bf 178} (2000), 64-112.

\bibitem{W5}
{---------}, {\it Mathematical Quantization}, CRC Press, 2001.

\end{thebibliography}
\end{document}